\documentclass[12pt]{article}
\usepackage{amsmath, amsfonts, tikz, geometry, ytableau, mathrsfs,amssymb}
\usepackage[indent,margin=1cm]{caption}
\usepackage{fullpage}
\usepackage[colorlinks,linkcolor=blue,anchorcolor=blue,citecolor=blue]{hyperref}
\usepackage{enumitem}
\usepackage{amsthm}

\numberwithin{equation}{section}
\numberwithin{figure}{section}

\hypersetup{colorlinks = true}
\newtheorem{thm}{Theorem}[section]
\newtheorem{conj}[thm]{Conjecture}

\newtheorem{defi}[thm]{Definition}

\newtheorem{cor}[thm]{Corollary}
\newtheorem{lem}[thm]{Lemma}
\newtheorem{prop}[thm]{Proposition}

\newtheorem{rem}[thm]{Remark}

\allowdisplaybreaks

\allowdisplaybreaks

\begin{document}
\begin{center}
	{\large \bf Unimodality of independence polynomials of two family of trees}
\end{center}

\begin{center}
Grace M.X. Li\\[6pt]
\end{center}

\begin{center}
School of Mathematics and Data Science, \\
Shaanxi University of Science and
Technology, Xi'an, Shaanxi 710021, P. R. China\\[6pt]

Email: {\tt grace\_li@sust.edu.cn}
\end{center}

\noindent\textbf{Abstract.}

In 1987, Alavi, Malde, Schwenk and Erd\H{o}s conjectured that the independence polynomials of trees are unimodal. Subsequently, many researchers proposed strengthening this conjecture to log-concavity.
In 2023, Kadrawi, Levit, Yosef, and Mizrachi  discovered independence polynomials of trees of order 26 that are not log-concave, which led them to construct two infinite families of such polynomials, denoted by $T_{3,m,n}$ and $T_{3,m,n}^*$. In this paper, we show that these two infinite families also satisfy the unimodal conjecture raised by Alavi, Malde, Schwenk, and Erd\H{o}s.



\section{Introduction}\label{intro}

Let $G$ be a simple graph with finite vertex set $V(G)$ and edge set $E(G)$. An \textit{independent set} $I$ (or \textit{stable set}) is a subset of $V(G)$ in which no two vertices are adjacent.
The size of a largest independent set is the \textit{independence number}, denoted $\alpha(G)$. Let $i_k$ denote the number of independent sets in $G$ of size $k$, with $i_0 = 1$ by convention. The generating function:
\[
I_G(t) = \sum_{k=0}^{\alpha(G)} i_k t^k,
\]
is called the \textit{independence polynomial} of $G$ \cite{GH83}.

A sequence of real numbers $a_0, a_1, \ldots, a_n$ is said to be \textit{unimodal} if there exists an index $j$ ($0 \leq j \leq n$) such that:
\[
a_0 \leq a_1 \leq \cdots \leq a_{j-1} \leq a_j \geq a_{j+1} \geq \cdots \geq a_n;
\]
it is \textit{log-concave} if $a_j^2 \ge a_{j-1}a_{j+1}$ for all $j$ with $1 \le j \le n-1$.
It is known that for a positive sequence, log-concavity implies unimodality.
A polynomial is called unimodal (respectively, log-concave) if its coefficient sequence is unimodal (respectively, log-concave).

In 1987, Alavi, Malde, Schwenk and Erd\H{o}s \cite{AMSE87} conjectured that the independence polynomial of every tree or forest is unimodal.

\begin{conj}[{\cite[Problem 3]{AMSE87}}]\label{conj-Erdos}
For every forest $F$, the independence polynomial $I_F(t)$ is unimodal.
\end{conj}

This celebrated conjecture has received considerable attention. For more recent work
on this conjecture, see \cite{BES18,BG21,Ben18,BH25,
HLMP24,LTZ25,WZ11,Zhu18} and references therein.

In particular, Brown, Dilcher and Nowakowski \cite{BDN00} conjectured that every well-covered graph (i.e., one in which all maximal independent sets have the same cardinality) has a unimodal independence polynomial.
Although Michael and Traves \cite{MT03} later provided counterexamples to this unimodality conjecture, it remains open for the subclass of very well-covered graphs.
Levit and Mandrescu \cite{LM04} showed that for any integer $\alpha$, there exists a very well-covered tree $T$ with $\alpha(T)=\alpha$ whose independence polynomial is log-concave. Based on this, they conjectured that every forest has a log-concave independence polynomial.
This conjecture was later verified by Yosef, Mizrachi and Kadrawi \cite{YMK21} up to 20 vertices, and by Radcliffe (see \cite{BGHW22}) up to 25 vertices.
Using a dynamic programming algorithm, Kadrawi, Levit, Yosef and Mizrachi \cite{KLYM23} checked all trees up to 26 vertices and found two trees of order 26 with non-log-concave independence polynomials, specifically $T_{3,4,4}$ and $T^*_{3,3,4}$.
For convenience, we define two families of trees $T_{3,m,n}$ and $T^*_{3,m,n}$, where $m, n$ are nonnegative integers. Both $T_{3,m,n}$ and $T^*_{3,m,n}$ have a root vertex $v_0$ with three children $v_1, v_2, v_3$. In $T_{3,m,n}$, $v_1$ has three children $v_{11}, v_{12},  v_{13}$, $v_2$ has $m$ children $v_{21}, v_{22}, \ldots, v_{2m}$, and $v_3$ has $n$ children $v_{31}, v_{32}, \ldots, v_{3n}$, and then each $v_{ij}$ has a child $v'_{ij}$. See Figure \ref{fig-tmnk} for an example.
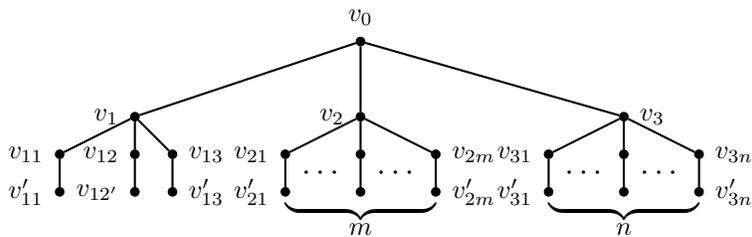
\begin{figure}[h]
  \centering
\begin{tikzpicture}
[thick, every label/.style={font=\footnotesize}, place/.style={thick,fill=black!100,circle,inner sep=0pt,minimum size=1mm,draw=black!100}]
\node [place,label=above:{$v_0$}] (v4) at (0,2) {};
\node [place,label=left:{$v_2$}] (v5) at (0,1) {};
\node [place,label=left:{$v_{21}$}] (v6) at (-1,0.5) {};
\node [place,label=left:{$v_{21}'$}] (v7) at (-1,0) {};
\node[place] (v14) at (0,0.5) {};
\node[place] (v15) at (0,0) {};
\node[place,label=right:{$\footnotesize{v_{2m}}$}] (v16) at (1,0.5) {};
\node[place,label=right:{$v_{2m}'$}] (v17) at (1,0) {};
\node[place,label=right:{$v_{13}$}] (v10) at (-2.5,0.5) {};
\node[place,label=right:{$v_{13}'$}] (v11) at (-2.5,0) {};
\node [place,label=left:{$v_{12}$}] (v8) at (-3,0.5) {};
\node [place,label=left:{$v_{12'}$}] (v9) at (-3,0) {};
\node[place,label=left:{$v_{11}$}] (v2) at (-4,0.5) {};
\node[place,label=left:{$v_{11}'$}] (v1) at (-4,0) {};
\node[place,label=left:{$v_{1}$}] (v3) at (-3,1) {};
\node[place,label=left:{$v_{31}$}] (v19) at (2.5,0.5) {};
\node[place,label=left:{$v_{31}'$}] (v20) at (2.5,0) {};
\node[place] (v21) at (3.5,0.5) {};
\node[place] (v22) at (3.5,0) {};
\node[place,label=right:{$v_{3n}$}] (v13) at (4.5,0.5) {};
\node[place,label=right:{$v_{3n}'$}] (v12) at (4.5,0) {};
\node[place,label=right:{$v_{3}$}] (v18) at (3.5,1) {};
\node at (-0.5,0.25) {$\cdots$};
\node at (0.5,0.25) {$\cdots$};
\node at (3,0.25) {$\cdots$};
\node at (4,0.25) {$\cdots$};
\draw (v1) -- (v2) -- (v3) -- (v4) -- (v18) -- (v13) -- (v12);
\draw (v9) -- (v8) -- (v3) -- (v10) -- (v11);
\draw (v7) -- (v6) -- (v5) -- (v4);
\draw (v15) -- (v14) -- (v5) -- (1,0.5) -- (v17);
\draw (v20) -- (v19) -- (v18) -- (v21) -- (v22);
\node[rotate = 0] at (0,-0.25) {$\underbrace{\hspace{2cm}}$};
\node at (0,-0.5) {\footnotesize $m$};
\node[rotate = 0] at (3.5,-0.25) {$\underbrace{\hspace{2cm}}$};
\node at (3.5,-0.5) {\footnotesize $n$};
\end{tikzpicture}
\caption{The graph $T_{3,m,n}$.}\label{fig-tmnk}
\end{figure}
If we replace the edge $v_{13}v'_{13}$ with a path $P_4$, which we lable as $v_{13}, v'_{13}, x, y$, then we get the tree $T^*_{3,m,n}$. See Figure \ref{fig-tmnk-1} for an illustration.
\begin{figure}[h]
  \centering
\begin{tikzpicture}
[thick, every label/.style={font=\footnotesize}, place/.style={thick,fill=black!100,circle,inner sep=0pt,minimum size=1mm,draw=black!100}]
\node [place,label=above:{$v_0$}] (v4) at (0,2) {};
\node [place,label=left:{$v_2$}] (v5) at (0,1) {};
\node [place,label=left:{$v_{21}$}] (v6) at (-1,0.5) {};
\node [place,label=left:{$v_{21}'$}] (v7) at (-1,0) {};
\node[place] (v14) at (0,0.5) {};
\node[place] (v15) at (0,0) {};
\node[place,label=right:{$v_{2m}$}] (v16) at (1,0.5) {};
\node[place,label=right:{$v_{2m}'$}] (v17) at (1,0) {};
\node[place,label=right:{$v_{13}$}] (v10) at (-2.5,0.5) {};
\node[place,label=right:{$v_{13}'$}] (v11) at (-2.5,0) {};
\node [place,label=left:{$v_{21}$}] (v8) at (-3,0.5) {};
\node [place,label=left:{$v_{21'}$}] (v9) at (-3,0) {};
\node[place,label=left:{$v_{11}$}] (v2) at (-4,0.5) {};
\node[place,label=left:{$v_{11}'$}] (v1) at (-4,0) {};
\node[place,label=left:{$v_{1}$}] (v3) at (-3,1) {};
\node[place,label=left:{$v_{31}$}] (v19) at (2.5,0.5) {};
\node[place,label=left:{$v_{31}'$}] (v20) at (2.5,0) {};
\node[place] (v21) at (3.5,0.5) {};
\node[place] (v22) at (3.5,0) {};
\node[place,label=right:{$v_{3n}$}] (v13) at (4.5,0.5) {};
\node[place,label=right:{$v_{3n}'$}] (v12) at (4.5,0) {};
\node[place,label=right:{$v_{3}$}] (v18) at (3.5,1) {};
\node at (-0.5,0.25) {$\cdots$};
\node at (0.5,0.25) {$\cdots$};
\node at (3,0.25) {$\cdots$};
\node at (4,0.25) {$\cdots$};
\draw (v1) -- (v2) -- (v3) -- (v4) -- (v18) -- (v13) -- (v12);
\draw (v9) -- (v8) -- (v3) -- (v10) -- (v11);
\draw (v7) -- (v6) -- (v5) -- (v4);
\draw (v15) -- (v14) -- (v5) -- (1,0.5) -- (v17);
\draw (v20) -- (v19) -- (v18) -- (v21) -- (v22);
\node[rotate = 0] at (0,-0.25) {$\underbrace{\hspace{2cm}}$};
\node at (0,-0.5) {\footnotesize $m$};
\node[rotate = 0] at (3.5,-0.25) {$\underbrace{\hspace{2cm}}$};
\node at (3.5,-0.5) {\footnotesize $n$};
\node [place,label=left:{$x$}] (v23) at (-2.5,-0.5) {};
\node [place,label=left:{$y$}] (v24) at (-2.5,-1) {};
\draw (v11) -- (v23) -- (v24);
\end{tikzpicture}
\caption{The graph $T^{\ast}_{3,m,n}$.}\label{fig-tmnk-1}
\end{figure}

Kadrawi, Levit, Yosef and Mizrachi \cite{KLYM23}  showed the following result.
\begin{thm}\cite{KLYM23}
For any $k\ge 3$, both $T_{3,k+1,k+1}$ and $T_{3,k,k+1}^*$ have non-log-concave independence polynomials.
\end{thm}

Kadrawi and Levit \cite{KL25} found more type of $T_{3,m,n}$ and $T^*_{3,m,n}$ which have non-log-concave independence polynomials. To be specific, they showed the following result.
\begin{thm}[\cite{KL25}]
For any $k\ge 4$,  $T_{3,k,k+1}$, $T_{3,k,k+2}$,  $T_{3,k-1,k+1}^*$, $T_{3,k,k+3}^*$ and $T_{3,k,k}^*$ all have non-log-concave independence polynomials.
\end{thm}

Further Ramos and Sun \cite{RS25} used AI to find more counterexamples, and Galvin \cite{Gal25} and Bautista-Ramos \cite{Bau25} studied the specific positions where the coefficients of independence polynomials of trees may fail to be log-concave (referred to as breaks).

The main result of this paper is to give the unimodality of $T_{3,m,n}$ and $T^*_{3,m,n}$, thereby providing further evidence supporting  of Conjecture \ref{conj-Erdos}. To be specific, we prove the following theorems.

\begin{thm}\label{thm-uni-t3mn}
For any $m,n\ge 1$, the independence polynomial of $T_{3,m,n}$ is unimodal.
\end{thm}

\begin{thm}\label{thm-uni-t3mn-star}
For any $m,n\ge 1$, the independence polynomial of $T_{3,m,n}^*$ is unimodal.
\end{thm}

The  proofs of Theorem \ref{thm-uni-t3mn} and Theorem \ref{thm-uni-t3mn-star}  rely on the theory of chromatic symmetric functions. Recently, Li, Yang, Zhang and the author \cite{LLYZ25} established a connection between the log-concavity of $I_G(t)$ and the  2-$s$-positivity of $Y_G$ (see Section \ref{sec2} for detailed definitions.) Using this connection, they showed that the independence polynomial of any spider graph is unimodal.

This paper is organized as follows. In Section \ref{sec2}, we recall some necessary results related to the chromatic symmetric functions. Section \ref{sec3} introduces  notations and properties that will be frequently used through out this paper. The proof of Theorem \ref{thm-uni-t3mn} will be presented in Section \ref{sec4}, and the proof of Theorem \ref{thm-uni-t3mn-star} is given in Section \ref{sec5}.

\section{Preliminary}\label{sec2}

In this section, we first recall some basic definitions of symmetric functions and then review the essential concepts related to chromatic symmetric functions.

In this section, we first recall some basic definitions of symmetric functions and then review the essential concepts related to chromatic symmetric functions. The \textit{algebra of symmetric functions} $\Lambda_{\mathbb{Q}}(\mathbf{x})$ is defined to be the subalgebra of $\mathbb{Q}[[\mathbf{x}]]$ consisting of formal power series $f(\mathbf{x})$ of bounded degree and satisfying
\[
f(\mathbf{x})=f(x_1,x_2,\ldots) = f(x_{\omega(1)},x_{\omega(2)},\ldots)
\]
for every permutation $\omega$ of positive integers.

The bases of $\Lambda_{\mathbb{Q}}(\mathbf{x})$ are naturally indexed by (integer) partitions. A \textit{partition} of $n$ is a sequence $\lambda = (\lambda_1,\ldots,\lambda_\ell)$ such that
\[
\lambda_1 \ge \lambda_2 \ge \cdots \ge \lambda_\ell > 0 \quad \mbox{and} \quad \lambda_1 + \lambda_2 + \cdots + \lambda_\ell = n.
\]
The number $\ell=\ell(\lambda)$ is called the \textit{length} of $\lambda$. By identifying $\lambda = (\lambda_1,\ldots,\lambda_\ell)$ with the infinite sequence $(\lambda_1,\ldots,\lambda_\ell,0,0,\ldots)$,
the \textit{Schur function} $s_{\lambda}(\mathbf{x})$ can be defined by the dual Jacobi-Trudi identity
\[
s_{\lambda'}(\mathbf{x}) = \det (e_{\lambda_i-i+j}(\mathbf{x}))_{1 \le i,j \le \ell},
\]
where $e_{\lambda}(\mathbf{x})$ denotes the {elementary symmetric function}, and $\lambda'$ is the conjugate partition of $\lambda$, see \cite{StaEC2} for more details. It is well known that $\{s_{\lambda}(\mathbf{x})\}$ is a base of $\Lambda_{\mathbb{Q}}(\mathbf{x})$. A symmetric function $f(\mathbf{x})$ is said to be  \textit{Schur-positive} or simply \textit{$s$-positive}, if it can be expressed as a nonnegative linear combination of Schur functions.  As usual,  we use $[s_{\lambda}]f(\mathbf{x})$ to denote the coefficient of $s_{\lambda}(\mathbf{x})$ in the expansion of $f(\mathbf{x})$ in terms of $\{s_{\lambda}(\mathbf{x})\}$.

We next introduce the notions of multicolorings and clan graphs. Let $G$ be a graph with vertex set $V(G)$. Given a map $\alpha: V(G) \to \mathbb{N}$,
a \textit{multicoloring of type $\alpha$} is a map $\kappa: V(G) \to 2^{\mathbb{N}_+}$ such that $|\kappa(v)| = \alpha(v)$, where $2^{\mathbb{N}_+}$ denotes the set of all finite subsets of positive integers. A multicoloring is called \textit{proper} if $\kappa(u) \cap \kappa(v) = \emptyset$ for all $uv \in E(G)$. Stanley \cite{Sta98} defined
\[
X^{\alpha}_G = \sum x_1^{a_1}x_2^{a_2} \cdots,
\]
where the sum ranges over all proper multicolorings $\kappa$ of type $\alpha$ and $a_i$ is the number of vertices $v$ such that $i \in \kappa(v)$. In the special case where $\alpha(v) = 1$ for all $v \in V(G)$, $X^\alpha_G$ reduces to $X_G$, the chromatic symmetric function of $G$ introduced by Stanley in    \cite{Sta95}.

Given $\alpha$ as above, the clan graph $G^\alpha$ is obtained from $G$ by replacing each vertex $v$ with a complete graph $K_{\alpha(v)}$, while preserving adjacency in the sense that every vertex in $K_{\alpha(u)}$ is adjacent to every vertex in $K_{\alpha(v)}$ whenever $uv \in E(G)$; see Figure \ref{fig-Galpha} for an example. Stanley \cite{Sta98} noted that
\begin{equation}\label{eq-xg-alpha}
X_{G^{\alpha}} = X_G^{\alpha} \prod_{v \in V(G)}\alpha(v)!,
\end{equation}
and we shall refer to $X_G^\alpha$ as the normalized chromatic symmetric function of $X_{G^\alpha}$.

 Let $P(t)$ be a polynomial with real coefficients satisfying $P(0) = 1$, say
$$P(t) =a_0 + a_1t + \cdots + a_dt^d.$$
Stanley \cite{Sta98} defined an inhomogeneous symmetric function
$$F_P(\mathbf{x}) = \prod_{i\geq 1} P(x_i)$$ and established a relation between real-rootedness of $P(t)$ and positivity of $F_P(\mathbf{x})$.

\begin{thm}[{\cite[Theorem 2.11]{Sta98}}]\label{thm-spos-rr}
Let $P(t)$ and $F_P(\mathbf{x})$ be defined as above. Then the following conditions are equivalent:
\begin{itemize}
  \item The coefficient of $s_{\lambda}(\mathbf{x})$ in $F_P(\mathbf{x})$ is nonnegative for each integer partition $\lambda$.
  \item The coefficient of $e_{\lambda}(\mathbf{x})$ in $F_P(\mathbf{x})$ is nonnegative for each integer partition $\lambda$.
  \item All zeros of $P(t)$ are negative real numbers.
\end{itemize}
\end{thm}

Recently, Li, Yang, Zhang and the author \cite{LLYZ25} obtained an analogue of Theorem \ref{thm-spos-rr}, which we now recall.
\begin{thm}[\cite{LLYZ25}]\label{thm-2s-lc}
Let $P(t)$ and $F_P(\mathbf{x})$ be defined as above with $a_0,\ldots,a_d$ being positive. Then the following conditions are equivalent:
\begin{itemize}
  \item[\textup{(i)}] The coefficient of $s_{\lambda}$ in $F_P(\mathbf{x})$ is nonnegative for any partition $\lambda$ of length at most~2.
  \item[\textup{(ii)}] The coefficient of $s_{(k,k)}$ in $F_P(\mathbf{x})$ is nonnegative for any $k \ge 1$.
  \item[\textup{(iii)}] $P(t)$ is log-concave.
  \item[\textup{(iv)}] $P(t)$ is strongly log-concave.
\end{itemize}
\end{thm}

The following result is, in fact, already implicit in the proof of Theorem \ref{thm-2s-lc},
for the sake of clarity, we state it explicitly here.

\begin{lem}\label{lem-skk-1}
Let $P(t)$ and $F_P(\mathbf{x})$ be defined as above with $a_0,\ldots,a_d$ being positive.
We have
\begin{equation}\label{equ-skk}
[s_{(k,k)}]F_P(\mathbf{x})=a_k^2 - a_{k-1}a_{k+1}.
\end{equation}
\end{lem}

\begin{proof}
Suppose that $P(t) = \prod_{j=1}^d (1+\theta_j t)$ with $0 \neq \theta_j \in \mathbb{C}$.
Let $\mathbf{y}=\{y_1,y_2,\ldots\}$ and $\mathbf{\theta}=\{\theta_1,\theta_2,\ldots,\theta_d\}$.
By the Cauchy identity
\[
\prod_{i,j} (1+y_jx_i) = \sum_{\lambda} s_{\lambda'}(\mathbf{y})s_{\lambda}(\mathbf{x})
\]
we have
\begin{equation}\label{eq-fpx}
F_P(\mathbf{x}) = \sum_{\lambda} s_{\lambda'}(\mathbf{\theta})s_{\lambda}(\mathbf{x}),
\end{equation}
where $s_{\lambda'}(\mathbf{\theta})$ stands for the specialization of $s_{\lambda'}(\mathbf{y})$ by setting $y_1 = \theta_1, \ldots, y_d = \theta_d$ and $y_{l} = 0$ for $l\geq d+1$.
Note that the coefficient $a_i =[t^i]P(t)$ is equal to $e_i(\mathbf{\theta})$.
By \eqref{eq-fpx}, we find that
\[
[s_{(k,k)}]F_P(\mathbf{x}) = s_{(k,k)'}(\mathbf{\theta}) = \begin{vmatrix}
                                           e_k(\theta) & e_{k+1}(\theta) \\
                                           e_{k-1}(\theta) & e_k(\theta)
                                         \end{vmatrix} = a_k^2 - a_{k-1}a_{k+1}.
\]
\end{proof}

We adopt the concept 2-Schur-positivity which was first introduced by Li, Yang, Zhang and the author \cite{LLYZ25}.
Recall that a symmetric function
$f$ is called \textit{2-Schur-positive} (or simply \textit{2-
$s$-positive}) if $[s_{\lambda}]f \ge 0$ for all $\lambda$ with $\ell(\lambda) \le 2$.
 For convenience, we write $f \ge_{2s} 0$ to indicate that $f$ is 2-$s$-positive..
Furthermore, we extend these notations to differences: if $f - g ={2s} 0$, we write $f ={2s} g$; and if $f - g \ge_{2s} 0$, we write $f \ge_{2s} g$.
In addition, if $f-g =_{2s} 0$ we also write $f =_{2s} g$, and $f-g\ge_{2s} 0$ we also write $f\ge_{2s} g$. Moreover, $f<_{2s} g$ means $g\ge_{2s} f$ and $f\ne_{2s} g$.

\begin{figure}[htbp]
  \centering
  \begin{tikzpicture}[scale = 1.5]
\fill (-5,-0.5) circle (0.3ex);
    \fill (-4,-0.5) circle (0.3ex);
    \fill (-3,-0.5) circle (0.3ex);
    \fill (-2,-0.5) circle (0.3ex);
    \draw (-5,-0.5) -- (-4,-0.5) -- (-3,-0.5);
    \draw (-3,-0.5) -- (-2,-0.5);
    \fill (-1,0) circle (0.3ex);
    \draw (-1,0) -- (-1,-1);
    \fill (-1,-1) circle (0.3ex);
    \fill (2,-0.5) circle (0.3ex);
    \fill (3,0) circle (0.3ex);
    \fill (1,-0.5) circle (0.3ex);
    \draw (1,-0.5) -- (2,-0.5);
    \draw (3,0) -- (1,-0.5);
    \draw (3,-1) -- (2,-0.5);
    \fill (3,-1) circle (0.3ex);
    \draw (2,-0.5) -- (3,0);
    \draw (1,-0.5) -- (3,-1);
    \draw (3,0) -- (3,-1);
\node at (-5,0) {$v_1$};
\node at (-4,0) {$v_2$};
\node at (-3,0) {$v_3$};
\node at (-2,0) {$v_4$};
\node at (-0.5,0) {$v_1^{(1)}$};
\node at (-0.5,-1) {$v_1^{(2)}$};
\node at (0.5,-0.5) {$v_3^{(1)}$};
\node at (2.5,-0.5) {$v_4^{(1)}$};
\node at (3.5,0) {$v_4^{(2)}$};
\node at (3.5,-1) {$v_4^{(3)}$};
\end{tikzpicture}
  \caption{The path $P_4$ and the clan graph $P_4^{(2,0,1,3)}$}\label{fig-Galpha}
\end{figure}
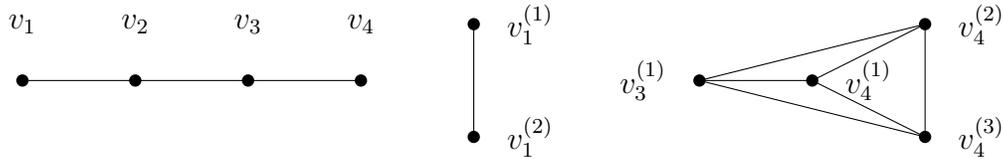

Stanley \cite[Corollary 2.12]{Sta98} directly presented the following identity
\[
\sum_{\alpha:V \to \mathbb{N}} X_G^{\alpha} = \prod_i I_G(x_i).
\]
Define $Y_G=\sum_{\alpha:V \to \mathbb{N}} X_G^{\alpha}$, then
Theorem \ref{thm-2s-lc} implies the following result.

\begin{cor}[\cite{LLYZ25}]\label{cor-2s-lc}
Let $G$ be a graph. Then the following conditions are equivalent:
\begin{itemize}
  \item $[s_{\lambda}] Y_G\ge 0$ for all partitions with $\ell(\lambda) \le 2$.
  \item $[s_{(k,k)}]Y_G \ge 0$ for all $k \ge 1$.
  \item The independence polynomial $I_G(t)$ is log-concave.
  \item The independence polynomial $I_G(t)$ is strongly log-concave.
\end{itemize}
\end{cor}

Moreover, Lemma \ref{lem-skk-1} implies the following result, which plays a crucial role in the proof of Theorem \ref{thm-uni-t3mn} and Theorem \ref{thm-uni-t3mn-star}.

\begin{cor}\label{lem-skk}
Let $G$ be a graph and let $I_G(t)=\sum_{j=0}^{\alpha(G)}i_j t^j$. Then for any $1\le k\le \alpha(G)-1$, the following two conditions are equivalent:
\begin{itemize}
\item $[s_{(k,k)}]Y_G\ge_{2s} 0$;
\item $i_k^2\ge i_{k-1}i_{k+1}$.
\end{itemize}
\end{cor}

We also need the following two results which are stated in \cite{LLYZ25}.

\begin{prop}[\cite{LLYZ25}]\label{prop-2s-coe}
Let $G$ be a graph with $|V(G)| = n$. Let $(k,l)$ be a partition of $n$ with $k \ge l \ge 1$ and let $\tilde{a}_{(k,l)}$ denote the number of semi-ordered stable partition of $G$ with type $(k,l)$. Then
\[
[s_{(k,l)}] X_G = \tilde{a}_{(k,l)} - \tilde{a}_{(k+1,l-1)}, \,\, [s_{(n)}] X_G = \tilde{a}_{(n)} = \begin{cases}
                          1, & \mbox{if $G$ consists of $n$ isolated vertices;} \\
                          0, & \mbox{otherwise}.
                        \end{cases}
\]
\end{prop}

\begin{cor}[\cite{LLYZ25}]\label{cor-2s-con-bipar}
For a connected bipartite graph $G$, $X_G$ is 2-$s$-positive if and only if its unique bipartition is balanced, i.e., it has type $(k,l)$ with $l \le k \le l+1$. Equivalently, $X_G$ is not 2-$s$-positive if and only if its bipartition is of type $(k,l)$ with $k \ge l+2$. In addition, $X_G =_{2s} 0$ for all non-bipartite graph $G$.
\end{cor}

Obviously, Corollary \ref{cor-2s-con-bipar} yields the following result.

\begin{cor}\label{cor-2s0}
Let $G$ be a graph and let $\alpha$ be a map from $V(G)$ to $\mathbb{N}$. If there exists an edge $uv\in E(G)$ such that $\alpha(u)+\alpha(v)\ge 3$, then
$X_G^{\alpha} =_{2s} 0$.
\end{cor}

Corollary \ref{cor-2s0} directly implies the following proposition.
\begin{prop}\label{prop-416}
For any graph $G$ and any map $\alpha\colon V(G)\rightarrow \mathbb{N}$, if $T$ is a connected component of $G^\alpha$ and $X_G^{\alpha}\neq_{2s} 0$, then for any $v\in  V(\hat{T})$, we have $\alpha(v)=1$ unless $\alpha(v)=2=\# V(T)$.
\end{prop}

Finally, we need to recall the following result due to Levit
and Mandrescu \cite{LM07}. They proved that for any bipartite graphs the last one-third of the coefficients of $I_G(t)$ are weakly decreasing.
\begin{thm}[{\cite[Corollary 3.3]{LM07}}]\label{thmLM07-2}
   For a fixed tree $T$, let $t$ denote the size of a maximum stable set in $T$, and let $c_k$ denote the coefficient of $x^k$ in the independence polynomial of $T$. Then
    \[c_{\lceil (2t-1)/3\rceil}\ge \cdots\ge c_{t-1}\ge c_t.\]
\end{thm}

\section{Notations and properties}\label{sec3}

In this section, we introduce some notation and establish several auxiliary properties that will be used frequently in the proofs of Theorems \ref{thm-uni-t3mn} and \ref{thm-uni-t3mn-star}. We begin with two definitions.

\begin{defi}
Given a graph $G$, an induced subgraph $H$ of $G$, and a map $\alpha: V(G) \to \mathbb{N}$, we introduce the following two maps derived from $\alpha$.
\begin{itemize}
    \item The map $\alpha|^H \colon V(G) \to \mathbb{N}$ is defined by
  \begin{equation}\label{eq-alpha-H}
    \alpha\mid^H(v)=\begin{cases}
    \alpha(v),&\text{if }v\in V(H);\\
    0,&\text{if }v\not\in V(H).
    \end{cases}
  \end{equation}
\item The restriction of $\alpha$ to $V(H)$ is denoted $\alpha|_H \colon V(H) \to \mathbb{N}$ and is given by
\[\alpha\mid_H(v)=\alpha(v)\ \text{for any}\ v\in V(H).\]
\end{itemize}
\end{defi}

\begin{rem}\label{rem-h-alpha}
For the sake of convenience, in the rest of this paper, we use $H^\alpha$ to denote $H^{\alpha\mid_H}$. Similarly, we use $X_H^\alpha$ to denote $X_H^{\alpha\mid_H}$.
Moreover,  by definition, each of $G^{\alpha\mid^H}$ and $H^\alpha$ is isomorphic to an induced subgraph of $G^\alpha$, and we regard them as such.  \end{rem}

Recall that Stanley  \cite[Proposition 2.3]{Sta95} noted that $X_{G + H} = X_GX_H$, where $G + H$ is the disjoint union of $G$ and $H$. Thus the following result is clear.

\begin{prop}\label{prop-42}
  Given a graph $G$ and an induced subgraph $H$ of $G$. Let $G\setminus H$ denote the induced subgraph on the vertex set $V(G)\setminus V(H)$. For any map $\alpha\colon V(G)\rightarrow \mathbb{N}$, if there are no edges between $G^{\alpha\mid^H}$ and $G^{\alpha\mid^{G\setminus H}}$, then we have
  \begin{equation}\label{eq-x-alpha-g+h}
    X_{G}^\alpha=X_{H}^{\alpha} X_{G\setminus H}^{\alpha}=X_G^{\alpha\mid^{H}}X_G^{\alpha\mid^{G\setminus H}}.
  \end{equation}
\end{prop}

Let $S(2^n)$ denote a spider graph with $n$ legs of length $2$, as labeled in Figure \ref{fig-S2n}. We now define two particular families of maps $\alpha\colon V(S(2^n))\rightarrow \mathbb{N}$, denoted by $x_{S,n}^t$ and $\mathcal{A}_k$.

\begin{figure}[h]
  \centering
\begin{tikzpicture}
[thick, every label/.style={font=\footnotesize}, place/.style={thick,fill=black!100,circle,inner sep=0pt,minimum size=1mm,draw=black!100}]
\node [place,label=left:{$v_0$}] (v5) at (-1.5,1.5) {};
\node [place,label=left:{$v_{1}$}] (v6) at (-2.5,0.5) {};
\node [place,label=left:{$v_{1}'$}] (v7) at (-2.5,0) {};
\node[place,label=left:{${v_{2}}$}] (v16) at (-1.5,0.5) {};
\node[place,label=left:{$v_{2}'$}] (v17) at (-1.5,0) {};
\draw (v7) -- (v6) -- (v5);
\draw  (v5) -- (v16) -- (v17);
\node [place,label=right:{$v_{n}$}] (v1) at (-0.5,0.5) {};
\node at (-1,0.5) {$\cdots$};
\node [place,label=right:{$v_{n}'$}] (v2) at (-0.5,0) {};
\draw (v5) -- (v1) -- (v2);
\end{tikzpicture}
  \caption{Spider $S(2^n)$.}\label{fig-S2n}
\end{figure}
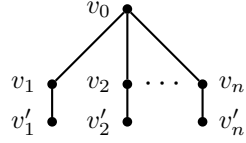

\begin{defi}\label{def-Xin}
Let $S\subseteq \{1,2,\ldots,n\}$ and let $t$ be an integer such that $n\ge t\ge \max S$.
A map $\alpha\colon V(S(2^n))\rightarrow \mathbb{N}$ is said to belong to $x^t_{S,n}$ if it satisfies the following three conditions:
\begin{itemize}
\item[(1)] $\alpha(v_0)=0$;
\item[(2)] $\alpha(v_i)+\alpha(v_i')\le 2$ for $1\le i\le n$;
\item[(3)] Define
\begin{equation}\label{eq-def-T}
T=\{j\colon \alpha(v_j)\ge 1,\ \alpha(v_j')=0\}.
\end{equation}
Then $\#T=t$. Write $T=\{k_1,k_2,\ldots,k_t\}$ with $k_1<k_2<\cdots<k_t$. We require that
\[
\alpha(v_{k_i})=\begin{cases}
2, &\text{if }i\in S;\\
1, &\text{if }i\not\in S.
\end{cases}
\]
\end{itemize}
\end{defi}

\begin{defi}
For $0\le k\le n$, let $\mathcal{A}_k$ denote the set of maps $\alpha\colon V(S(2^n))\rightarrow \mathbb{N}$ satisfying the following conditions:
\begin{itemize}
  \item[(1)] $\alpha(v_0)=1$;
  \item[(2)] For each $1\le i\le n$, $\alpha(v_i)\le 1$ and $\alpha(v_i)+\alpha(v_i')\le 2$;
  \item[(3)] $\#\{j\colon \alpha(v_j)=1,\ \alpha(v_j')=0\}=k$.
\end{itemize}
\end{defi}

Now we define a map $\phi_S\colon \mathcal{A}_k\to x^k_{S,n}$ for any $S\subseteq\{1,2,\ldots,k\}$, which will be used frequently in the remainder of this paper.

\begin{defi}\label{def-phi_S-j-alpha}
Let $\alpha\colon V(S(2^n))\to\mathbb{N}$ and $\alpha\in\mathcal{A}_k$.
Define
\[
T'=\{j\colon \alpha(v_j)=1,\ \alpha(v_j')=0\}.
\]
By the definition of $\mathcal{A}_k$, we see that $\#T'=k$ and write $T'=\{i_1,i_2,\ldots,i_k\}$ with $i_1<i_2<\cdots<i_k$.
Note that $S\subseteq \{1,2,\ldots,k\}$, we define $\psi_S(\alpha)$ as follows
\begin{equation}\label{eq-def-phi_S-s-alpha}
\phi_S(\alpha)(v)=\begin{cases}
2, &\text{if }v=v_{i_j} \text{ and }j\in S;\\
0, &\text{if }v=v_0;\\
\alpha(v), & otherwise.
\end{cases}
\end{equation}
For simplicity, we write $\phi_j(\alpha)$ for $\phi_{\{j\}}(\alpha)$.
\end{defi}


The following result concerns $\phi_S$.

\begin{prop}\label{pro-bijection}
For any $n\ge k\ge 0$ and $S\subseteq \{1,2,\ldots,k\}$,  $\phi_S$ is a bijection between $\mathcal{A}_k$ and $x^k_{S,n}$.
\end{prop}

\begin{proof}
For any $\alpha\in \mathcal{A}_k$, we first show that $\phi_S(\alpha)\in x^k_{S,n}$. By definition we have
\begin{equation}\label{eq-df-k}
    \#\{j\colon \alpha(v_j)= 1,\alpha(v_j')=0\}=k.
\end{equation}
Write $\{j\colon \alpha(v_j)= 1,\alpha(v_j')=0\}=\{r_1,r_2,\ldots,r_k\}$ with $1\le r_1<r_2<\cdots<r_k\le n$. Thus by the construction of $\phi_S(\alpha)$ \eqref{eq-def-phi_S-s-alpha}, we have
\begin{equation}
\phi_S(\alpha)(v)=\begin{cases}
2, & \text{if }v=v_{r_i} \text{ and }i\in S;\\
1, & \text{if }v=v_{r_i} \text{ and }i\not\in S;\\
0, & \text{if } v=v_0;\\
\alpha(v),& \text{otherwise.}
\end{cases}
\end{equation}
Note that for any $1\le h\le n$ and $h\not\in \{j\colon \alpha(v_j)\ge 1,\alpha(v_j')=0\}$, we have either $\alpha(v_h')\ne 0$ or $\alpha(v_h)=0$. Since $\phi_S(\alpha)(v_h)=\alpha(v_h)$ and $\phi_S(\alpha)(v'_h)=\alpha(v'_h)$, we have either $\phi_S(\alpha)(v_h')\ne 0$ or $\phi_S(\alpha)(v_h)=0$.
This implies the set $T$ (as defined in \eqref{eq-def-T}) for $\phi_S(\alpha)$ is $\{r_1,r_2,\ldots,r_k\}$. A routine verification shows  that $\phi_S(\alpha)\in x^k_{S,n}$.

To show that $\phi_S$ is a bijection, we construct its inverse. For any $\beta\in x^k_{S,n}$, define $\phi_S^{-1}$ as follows. Let $T$ be the set defined in \eqref{eq-def-T} for $\beta$, and write $T=\{a_1,a_2,\ldots,a_k\}$, where $k\ge \max S$. By the definition of $x^k_{S,n}$, we see that $\beta(v_{a_i})=2$ for $i\in S$, $\beta(v_{a_i})=1$ for $i\not\in S$ and $\beta(v_{a_i}')=0$ for $1\le i\le k$. Moreover, from the definition of $T$,  if $j\not\in T$, then either $\beta(v_j)=0$ or $\beta(v_j)=\beta(v_j')=1$. Define
\begin{equation}\label{eq-def-phi-s-1}
\phi_S^{-1}(\beta)(v)=\begin{cases}
1,&\text{if }v=v_{i} \text{ and } \beta(v_i)=2\text{ or }v=v_0;\\
\beta(v),&\text{otherwise.}
\end{cases}
\end{equation}
Clearly $\phi_S^{-1}(\beta)(v_0)=1$ and $\phi_S^{-1}(\beta)(v_i)\le 1$ for all $1\le i\le n$. Moreover, since $\phi_S^{-1}(\beta)(v_i)\le \beta(v_i)$ and $\beta(v_i)+\beta(v_i')\le 2$, we obtain $\phi_S^{-1}(\beta)(v_i)+\phi_S^{-1}(\beta)(v_i')\le 2$.
 Furthermore, $\phi_S^{-1}(\beta)(v_{a_i})=1$ and $\phi_S^{-1}(\beta)(v_{a_i}')=0$ for any $1\le i\le k$, and for any $j\not\in T$,  either $\phi_S^{-1}(\beta)(v_j)=0$ or $\phi_S^{-1}(\beta)(v_j)=\phi_S^{-1}(\beta)(v_j')=1$. Hence
\[\{j\colon \phi_S^{-1}(\beta)(v_j)=1, \phi_S^{-1}(\beta)(v_j')=0\}=\{a_1,a_2,\ldots, a_k\}.\]
Thus $\phi_S^{-1}(\beta)\in \mathcal{A}_k$. It is clear that $\phi_S(\phi_S^{-1}(\beta))=\beta$ and $\phi_S^{-1}(\phi_S(\alpha))=\alpha$. This yields $\phi_S$ is a bijection.
\end{proof}

\begin{rem}\label{rem-phi-1}
  Note that in the construction of $\phi_S^{-1}$ \eqref{eq-def-phi-s-1}, it does not concern $S$ or $k$. Thus we may write $\phi^{-1}$ instead of $\phi^{-1}_S$, which will be frequently used in the remainder of this paper. \end{rem}


The following corollary is a direct consequence of Proposition \ref{pro-bijection}.
\begin{cor}\label{cor-phi-distinct}
For any $S_1,S_2\subseteq{1,2,\ldots,n}$ and any maps $\alpha_1,\alpha_2\colon V(S(2^n))\rightarrow \mathbb{N}$, if $\phi_{S_1}(\alpha_1)=\phi_{S_2}(\alpha_2)$ then $\alpha_1=\alpha_2$ and $S_1=S_2$.
\end{cor}
\begin{proof}
Since $\phi_{S_1}(\alpha_1)=\phi_{S_2}(\alpha_2)$, we have
\[\alpha_1=\phi^{-1}(\phi_{S_1}(\alpha_1))=\phi^{-1}(\phi_{S_2}(\alpha_2))=\alpha_2.\]
Moreover, from the construction of $\phi_{S_1}$ and $\phi_{S_2}$, it is easy to see that $\phi_{S_1}(\alpha)\neq \phi_{S_2}(\alpha)$ whenever $S_1\neq S_2$. Hence $S_1=S_2$.
\end{proof}

We next introduce another definition $X^t_{S,n}$, which is closely related to $x_{S,n}^t$ and will be used frequently  in the proofs of Theorem \ref{thm-uni-t3mn} and Theorem \ref{thm-uni-t3mn-star}.

\begin{defi}\label{defi-xtsn}
Let
\[X^t_{S,n}=\{{S(2^n)}^{\alpha}\colon \alpha\in x^t_{S,n}\}, \quad X_{S,n}=\bigcup_{t=\max S}^nX^t_{S,n}.\]
We also abbreviate $X_{\{j\},n}$ as $X_{j,n}$. Moreover, we set
\begin{equation}\label{equ-def-x0n}
X_{0,n}=\{{S(2^n)}^{\alpha}\colon {S(2^n)}^{\alpha} \text{ is 2-$s$-positive}\}\setminus \bigcup_{j=1}^n X_{j,n}.
\end{equation}
\end{defi}

The following corollary can be deduced from Corollary \ref{cor-phi-distinct}, and it will be used in the classification of the set $\alpha\colon V(T_{3,m,n})\rightarrow \mathbb{N}$.

\begin{cor}\label{cor-xsn-distinct}
For any $S_1,S_2\subseteq\{1,2,\ldots,n\}$, if $S_1\ne S_2$, then $X_{S_1,n}\cap X_{S_2,n}=\emptyset$.
\end{cor}
\begin{proof}
Assume the contrary, if there exists $\beta$ such that $S(2^n)^\beta\in X_{S_1,n}\cap X_{S_2,n}$, then by definition, there exists $t_1,t_2$ such that $\beta\in x_{S_1,n}^{t_1}$ and $\beta\in x_{S_2,n}^{t_2}$. From Proposition \ref{pro-bijection},  there exists $\alpha_1\in \mathcal{A}_{t_1}$ and $\alpha_2\in \mathcal{A}_{t_2}$ such that
\[\phi_{S_1}(\alpha_1)=\beta=\phi_{S_2}(\alpha_2).\]
From Corollary \ref{cor-phi-distinct}, we obtain $S_1=S_2$, a contradiction. Hence $X_{S_1,n}\cap X_{S_2,n}=\emptyset$.
\end{proof}

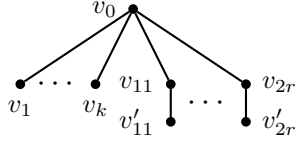
\begin{figure}[h]
  \centering
\begin{tikzpicture}
[thick, every label/.style={font=\footnotesize}, place/.style={thick,fill=black!100,circle,inner sep=0pt,minimum size=1mm,draw=black!100}]
\node [place,label=left:{$v_0$}] (v5) at (-1.5,1.5) {};
\node [place,label=left:{$v_{11}$}] (v6) at (-1,0.5) {};
\node [place,label=left:{$v_{11}'$}] (v7) at (-1,0) {};
\node[place,label=right:{${v_{2r}}$}] (v16) at (0,0.5) {};
\node[place,label=right:{$v_{2r}'$}] (v17) at (0,0) {};
\node at (-0.5,0.25) {$\cdots$};
\draw (v7) -- (v6) -- (v5);
\draw  (v5) -- (v16) -- (v17);
\node [place,label=below:{$v_{1}$}] (v2) at (-3,0.5) {};
\node [place,label=below:{$v_{k}$}] (v1) at (-2,0.5) {};
\node at (-2.5,0.5) {$\cdots$};
\draw (v1) -- (v5) -- (v2);
\end{tikzpicture}
  \caption{Spider $S(1^k, 2^r)$.}\label{fig-S1k2r}
\end{figure}

Let $S(1^k, 2^r)$ be the spider graph with $k$ legs of length 1 and $r$ legs of length 2. The labelling of this graph is shown in Figure \ref{fig-S1k2r}. We have the following result on $X_{S(2^n)}^{\phi_j(\alpha)}$.

\begin{prop}\label{pro-alpha-vs2n-phij}
 Let $\alpha\colon V(S(2^n))\rightarrow \mathbb{N}$ be such that $\alpha(v_0)=1$ and $S(2^n)^\alpha\cong S(1^k,2^r)$ for some $k\ge 3$. Then we have
  \begin{equation}\label{eq-xg-s-r+k-r+1}
X_{S(2^n)}^\alpha=X_{S(1^k,2^r)}=_{2s}s_{(r+k,r+1)}-s_{(r+k-1,r+2)}.
\end{equation}
Moreover, for any $1\le j\le k$,
\begin{equation}\label{equ-s2nphij}
X_{S(2^n)}^{\phi_j(\alpha)}\ge_{2s} s_{(r+k,r+1)}+s_{(r+k-1,r+2)}.
\end{equation}
\end{prop}

\begin{proof}
  Since $S(1^k,2^r)$ contains no $3$-clique, we have $\alpha(u)+\alpha(v)\le 2$ for every edge $uv\in E(S(2^n))$.
  Moreover, since $S(2^n)^\alpha\cong S(1^k,2^r)$ is connected, we deduce that $\alpha(v)\in\{0,1\}$ for all $v\in V(S(2^n))$, and that the subgraph of $S(2^n)$ induced by $\{v\in V(S(2^n))\colon \alpha(v)=1\}$ is isomorphic to $S(1^k,2^r)$. Thus
  \begin{equation}\label{eq-xs2n-alpha-1}
    X_{S(2^n)}^\alpha=X_{S(2^n)^\alpha}=X_{S(1^k,2^r)}.
  \end{equation}
  It is obvious that the bipartition of $S(1^k,2^r)$ is $(r+k,r+1)$. Thus by Proposition \ref{prop-2s-coe} we have
  \begin{equation}\label{eq-xs2n-alpha-2}
    X_{S(2^n)}^\alpha=s_{(r+k,r+1)}-s_{(r+k-1,r+2)}.
  \end{equation}

  On the other hand, by the construction of ${\phi_j(\alpha)}$, the graph ${S(2^n)}^{\phi_j(\alpha)}$ consists of $r+1$ copies of $K_2$ and $k-1$ isolated vertices. Thus by the Littlewood-Richardson rule,
\begin{equation}
X_{{S(2^n)}^{\phi_j(\alpha)}}=(2s_{(1,1)})^{r+1}s_1^{k-1}=_{2s}\sum_{a=1}^{\lfloor \frac{k+1}{2}\rfloor}2^{r+1} f^{(k-a,a-1)} s_{(r+1+k-a,r+1+a-1)},
\end{equation}
where $f^{(k-a,a-1)}$ denotes the number of standard Young tableaux of shape $(k-a,a-1)$. It is clear that $\frac{k+1}{2}\ge 2$ and $f^{(k-a,a-1)}\ge 1$ for each $a$, thus
\begin{equation}\label{eq-46}
X_{{S(2^n)}^{\phi_j(\alpha)}}\ge_{2s}2^{r+1}  s_{(r+k,r+1)}+2^{r+1}  s_{(r+k-1,r+2)}.
\end{equation}
Combining \eqref{eq-xg-alpha}, we have
\begin{equation}
X_{S(2^n)}^{\phi_j(\alpha)}=\frac{1}{2}X_{{S(2^n)}^{\phi_j(\alpha)}}\ge_{2s}2^{r}  s_{(r+k,r+1)}+2^{r}  s_{(r+k-1,r+2)}\ge_{2s} s_{(r+k,r+1)}+s_{(r+k-1,r+2)}.
\end{equation}
\end{proof}

Given a graph $G$ and a map $\alpha\colon V(G)\rightarrow\mathbb{N}$, let $T$ be a connected component of $G^\alpha$. We now define two subgraphs of $G$, denoted $\hat{T}$ and $\overline{T}$, via the following proposition.

\begin{prop}\label{pro-vg-alpha-conn-t-v-alpha}
For any graph $G$ and any map $\alpha\colon V(G)\rightarrow\mathbb{N}$, let $T$ be a connected component of $G^\alpha$. Then for any $v\in V(G)$, either $V(v^\alpha)\cap V(T)=\emptyset$ or $V(v^\alpha)\subseteq V(T)$, where $v^\alpha$ is as in Remark \ref{rem-h-alpha}.
\end{prop}

\begin{proof}
  Assume the contrary, if there exists $v_1,v_2\in V(v^\alpha)$ with $v_1\in V(T)$ and $v_2\not\in V(T)$. Then from the definition of $G^\alpha$, we see that $v_1v_2\in E(G^\alpha)$, which is contradict to $T$ is a connected component of $G^\alpha$.
\end{proof}

Proposition \ref{pro-vg-alpha-conn-t-v-alpha} enables us to give the following definition.

\begin{defi}\label{defi-t-hat}
For any graph $G$ and $\alpha\colon V(G)\rightarrow \mathbb{N}$, let $T$ be a connected component of $G^\alpha$. If $V'=\{v\in V(G)\colon V(v^\alpha)\subseteq V(T),\alpha(v)>0\}$, then let $\hat{T}$ denote the induced subgraph $G[V']$. Moreover, let $\overline{T}$ denote the subgraph of $G$ induced by $V(G)\setminus V(\hat{T})$.
\end{defi}

Using Proposition \ref{prop-42}, we obtain the following result.

\begin{prop}\label{pro-X-hat-overline-T}
    For any graph $G$,   $\alpha\colon V(G)\rightarrow \mathbb{N}$ and $\beta\colon V(G)\rightarrow \mathbb{N}$,  let $T$ be a connected component of $G^\alpha$ and let $\hat{T}$ and $\overline{T}$ be as in Definition \ref{defi-t-hat}.  Then
    \begin{equation}\label{eq-xg-alpha-xthat-xoverlinet}
        X_G^\alpha=X_{\hat{T}}^\alpha X_{\overline{T}}^\alpha.
    \end{equation}
    Moreover, if $\beta(v)=\alpha(v)$ for any $v\in V(\overline{T})$, then
    \begin{equation}\label{eq-xg-beta-xthat-xoverlinet}
        X_G^\beta=X_{\hat{T}}^\beta X_{\overline{T}}^\alpha.
    \end{equation}
\end{prop}

\begin{proof}
    Since $T$ is a connected component of $G^\alpha$, Proposition \ref{prop-42} implies \eqref{eq-xg-alpha-xthat-xoverlinet}. Moreover, since $\beta(v)=\alpha(v)$ for all $v\in V(\overline{T})$ and $T$ is a connected component of $G^\alpha$, we have that if $u\in V(\overline{T})$, $v\in V(\hat{T})$, and $uv\in E(G)$, then $\alpha(u)=0=\beta(u)$. Consequently, there are no edges between $\hat{T}^\beta$ and $\overline{T}^\beta$, and \eqref{eq-xg-beta-xthat-xoverlinet} follows from Proposition \ref{prop-42}.
\end{proof}

We next study the 2-$s$-positivity of $X_{S(2^n)}^\alpha+X_{S(2^n)}^{\phi_j(\alpha)}$, which is crucial for the proofs of Theorems \ref{thm-uni-t3mn} and \ref{thm-uni-t3mn-star}.

\begin{prop}\label{pro-xgn+xgalpha}
Fix $\alpha\colon V(S(2^n))\rightarrow \mathbb{N}$. If $S(2^n)^\alpha$ is not 2-$s$-positive, then $\alpha\in \mathcal{A}_k$ for some $k\ge 3$, where $\mathcal{A}_k$ is as in Proposition \ref{pro-bijection}. Moreover, for any $1\le j\le k$, we have
\begin{equation}
X_{S(2^n)}^\alpha+X_{S(2^n)}^{\phi_j(\alpha)}\ge_{2s} 0.
\end{equation}
\end{prop}

\begin{proof}
Let $T_0$ be the connected component which contains $v_0^\alpha$. From Corollary \ref{cor-2s0}, we deduce that every connected component in $S(2^n)^\alpha$ other than $T_0$ is either an isolated vertex or a path. This yields $T_0$ must be non-2-$s$-positive. Thus by Proposition \ref{prop-416}, we have $\alpha(v_0)=1$. Moreover, Corollary \ref{cor-2s0} implies $\alpha(v_j)\le 1$ and $\alpha(v_j)+\alpha(v_j')\le 2$ for all $j$. Hence $T_0$ is a spider $S(1^k, 2^r)$, which implies that
  $$\#\{i\colon \alpha(v_i)=1,\alpha(v_i')=0\}=k.$$
  This yields $\alpha\in \mathcal{A}_k$.
  Furthermore, we see that the bipartition of $T_0$ is $(r+k,r+1)$. From  Corollary \ref{cor-2s-con-bipar}, we see that  $(r+k,r+1)$ is non-balance, hence  $k\ge 3$.

Let $\hat{T}_0$ and $\overline{T_0}$ be as defined in Definition \ref{defi-t-hat}. Assume $T_0\cong S(1^k, 2^r)$, then $X_{T_0}=_{2s} s_{(r+k,r+1)}-s_{(r+k-1,r+2)}$. From Proposition \ref{pro-X-hat-overline-T}, we see that
 \begin{equation}\label{equ-X-S2n-alpha}
X_{S(2^n)}^{\alpha}=_{2s} (s_{(k+r,r+1)}-s_{(k+r-1,r+2)})X_{\overline{T_0}}^{\alpha}.
 \end{equation}

On the other hand, we calculate $X_{S(2^n)}^{\phi_j(\alpha)}$. It is obvious that
\begin{equation}\label{eq-subset-vt0alpha-1}
\{v_j\colon \alpha(v_j)=1,\alpha(v_j')=0\}\subseteq V(\hat{T}_0).
\end{equation}
  From the construction of $\phi_j$, we see that $\alpha(v)=\phi_j(\alpha)(v)$ for any $v\in  V(\overline{T_0})$.
  Thus by Proposition \ref{pro-X-hat-overline-T},
 \begin{equation}\label{equ-X-S2n-phij}
 X_{S(2^n)}^{\phi_j(\alpha)}=X_{\hat{T}_0}^{\phi_j(\alpha)} X_{\overline{T_0}}^{\alpha}.
 \end{equation}
 Moreover, from the construction of $\phi_j$ \eqref{eq-def-phi_S-s-alpha} and \eqref{eq-subset-vt0alpha-1}, one may check that
 \[\phi_j(\alpha\mid^{\hat{T}_0})=\phi_j(\alpha)\mid^{\hat{T}_0}.\]
Therefore
\begin{equation}\label{eq-temp-xs2n-phi-j}
X_{\hat{T}_0}^{\phi_j(\alpha)}=X_{S(2^n)}^{\phi_j(\alpha)\mid^{\hat{T}_0}}=X_{S(2^n)}^{\phi_j(\alpha\mid^{\hat{T}_0})}.
\end{equation}
 Notice that $\hat{T}_0^{\alpha}=T_0\cong S(1^k,2^r)$. Combining \eqref{equ-X-S2n-alpha}, \eqref{equ-X-S2n-phij}, \eqref{eq-temp-xs2n-phi-j}, and Proposition \ref{pro-alpha-vs2n-phij}, we deduce that
\begin{align*}
X_{S(2^n)}^\alpha+X_{S(2^n)}^{\phi_j(\alpha)}&
=(X^{\alpha|^{\hat{T}_0}}_{S(2^n)}+X^{\phi_j(\alpha|^{\hat{T}_0})}_{S(2^n)} )\cdot X^{\alpha}_{\overline{T_0}}\\
 &\ge_{2s}2s_{(r+k,r+1)} X^{\alpha|_{\overline{T_0}}}_{S(2^n)}\ge_{2s} 0.
 \end{align*}
 where the last inequality follows from the fact that every connected component in $S(2^n)^\alpha$ other than $T_0$ is 2-$s$-positive.
\end{proof}

We also need the following result on the 2-$s$-positivity of $S(1,2^n)$, which is analogous to Proposition \ref{pro-xgn+xgalpha}.

\begin{prop}\label{pro-s12n}

The labelling of $S(1,2^n)$ is shown in Figure \ref{fig-S1k2r} for $k=1$ and $r=n$. Let $T$ be the subgraph of $S(1,2^n)$ induced by $V(S(1,2^n))\setminus \{v_1\}$; note that $T\cong S(2^n)$. Suppose $\alpha\colon V(S(1,2^n))\rightarrow \mathbb{N}$ is such that $\alpha\mid_T\in\mathcal{A}_k$ for some $k\ge 2$.
For any $1\le j\le k$, we define $\beta\colon V(S(1,2^n))\rightarrow \mathbb{N}$ as follows:
  \begin{equation}
      \beta(v)=\begin{cases}
          \phi_j(\alpha\mid_T)(v),&\text{if }v\in V(T);\\
          \alpha(v_1),&\text{if }v=v_1.
      \end{cases}
  \end{equation}
  Then we have $X_{S(1,2^n)}^\alpha+X_{S(1,2^n)}^\beta\ge_{2s} 0$.
\end{prop}

\begin{proof}

If $\alpha(v_1)=0$, then clearly $X_{S(1,2^n)}^\alpha=X_T^\alpha$ and $X_{S(1,2^n)}^\beta=X_T^\beta$, and the desired result follows from Proposition \ref{pro-xgn+xgalpha}.

    Now we assume $\alpha(v_1)\ge 1$.   From $\alpha\mid_T\in \mathcal{A}_k$, we see that $\alpha(v_0)=1$. If $\alpha(v_1)\ge 2$, then clearly $S(1,2^n)^\alpha$ is non-bipartite. By Corollary \ref{cor-2s-con-bipar} we have $X_{S(1,2^n)}^\alpha=_{2s}0$. Moreover, from the construction of $\beta$, we have $X_{S(1,2^n)}^\beta\ge_{2s} 0$, which yields the desired conclusion.

    We next consider the case $\alpha(v_1)=1$. Let $T_0$ be the connected component of $S(1,2^n)^\alpha$ containing $v_0^\alpha$, and let $\hat{T}_0$ and $\overline{T}_0$ be as in Definition \ref{defi-t-hat}.    On the one hand, since $\alpha\mid_T\in\mathcal{A}_k$ we see that $T_0\cong S(1^{k+1},2^r)$ for some $r\ge 0$. Thus $X_{\hat{T}_0}^\alpha=X_{T_0}=s_{(r+k+1,r+1)}-s_{(r+k,r+2)}$. Using Proposition \ref{pro-X-hat-overline-T}, we have
    \begin{equation}\label{eq-x-s1-2n-alpha}
        X_{S(1,2^n)}^\alpha=_{2s} (s_{(r+k+1,r+1)}-s_{(r+k,r+2)})X_{\overline{T_0}}^\alpha.
    \end{equation}

    On the other hand, let $T_0'$ be the subgraph of $\hat{T}_0$ obtained by removing $v_1$. From the construction of $\beta$, we have
    \begin{equation}\label{eq-xt0'-beta-ge2s-0}
        X_{\hat{T}_0}^\beta=s_1\cdot X_{T_0'}^\beta.
    \end{equation}
   Since $\hat{T}_0^\alpha=T_0\cong S(1^{k+1},2^r)$ we have $T_0'^\alpha\cong S(1^k,2^r)$. Because $\beta\mid_{T_0'}=\phi_j(\alpha\mid_{T_0'})$, \eqref{equ-s2nphij} yields
   \begin{equation}\label{eq-xt0'-beta-ge2s}
       X_{T_0'}^\beta\ge_{2s} s_{(k+r,r+1)}+s_{(k+r-1,r+2)}.
   \end{equation}
   Combining \eqref{eq-xt0'-beta-ge2s-0} and \eqref{eq-xt0'-beta-ge2s} gives
   \begin{equation}
       X_{\hat{T}_0}^\beta\ge_{2s} s_1(s_{(k+r,r+1)}+s_{(k+r-1,r+2)})\ge_{2s} s_{(r+k+1,r+1)}+s_{(r+k,r+2)}.
   \end{equation}
   By Proposition \ref{pro-X-hat-overline-T}, we deduce that
   \begin{equation}\label{eq-x-s-1-2n-beta}
       X_{S(1,2^n)}^\beta\ge_{2s} (s_{(r+k+1,r+1)}+s_{(r+k,r+2)})X_{\overline{T_0}}^\alpha.
   \end{equation}
   Combining \eqref{eq-x-s1-2n-alpha} and \eqref{eq-x-s-1-2n-beta}, we have
   \begin{equation}
       X_{S(1,2^n)}^\alpha+X_{S(1,2^n)}^\beta\ge_{2s} 2s_{(r+k+1,r+1)}X_{\overline{T_0}}^\alpha\ge_{2s} 0.
   \end{equation}
\end{proof}



The following result is a generalization of Proposition \ref{pro-xgn+xgalpha}, which will also be used in Section \ref{sec4}.

\begin{prop}\label{Prop-con-com-phi-pos}
Let $F$ be a forest with connected components $F_1,F_2,\ldots, F_n$, where each $F_i\cong S(2^{n_i})$. Let $\alpha\colon V(F)\rightarrow \mathbb{N}$ be such that $\alpha\mid_{F_i}\in \mathcal{A}_{k_i}$ for some $k_i\ge 3$. For any $1\le a_i\le k_i$, define $\beta\colon V(F)\rightarrow \mathbb{N}$ by
\begin{equation}
\beta(v)=
    \phi_{a_i}(\alpha\mid_{F_i})(v),\quad\text{if }v\in V(F_i).
\end{equation}
Then we have $X_F^\alpha+X_F^\beta\ge_{2s} 0$.
\end{prop}

\begin{proof}

As in the proof of Proposition \ref{pro-xgn+xgalpha}, for each $F_i^\alpha$, let $T_i$ be the unique connected component containing $v_i^\alpha$, where $v_i$ is the torso of $F_i$, and define $\hat{T}_i$ as in Definition \ref{defi-t-hat}. Moreover, let $\overline{T_i}$ denote the subgraph of $F_i$ induced by $V(F_i)\setminus V(\hat{T}_i)$. It is clear that $T_i\cong S(1^{k_i},2^{r_i})$ for some $r_i\ge 0$. By Proposition \ref{pro-X-hat-overline-T},  we have
 \begin{equation}\label{equ-X-S2n-alpha-forest}
X_{F_i}^{\alpha}=X_{{T_i}}
X_{\overline{T_i}}^{\alpha}=_{2s}(s_{(k_i+r_i,r_i+1)}-s_{(k_i+r_i-1,r_i+2)})X_{\overline{T_i}}^{\alpha}.
 \end{equation}

 On the other hand, by the definition of $\beta$, we see that $\beta(v)=\alpha(v)$ for any $v\in \overline{T_i}$. Again by Proposition \ref{pro-X-hat-overline-T}, we deduce that
  \begin{equation}\label{equ-X-S2n-phij-forest}
 X_{F_i}^{\phi_{a_i}(\alpha\mid_{F_i})}=X_{\hat{T}_i}^{\phi_{a_i}(\alpha\mid_{F_i})} X_{\overline{T_i}}^{\alpha}=X_{F_i}^{\phi_{a_i}(\alpha\mid_{F_i})\mid^{{\hat{T_i}}}} X_{\overline{T_i}}^{\alpha}.
 \end{equation}
 From \eqref{equ-s2nphij} and $F_i^{(\alpha\mid_{F_i})\mid^{\hat{T}_i}}=\hat{T}_i^\alpha=T_i\cong S(1^{k_i},2^{r_i})$, we see that
\begin{equation}\label{equ-X-S2n-phij-forest-1}
    X_{{F_i}}^{\phi_{a_i}((\alpha\mid_{F_i})\mid^{\hat{T}_i})}\ge_{2s} s_{(k_i+r_i,r_i+1)}+s_{(k_i+r_i-1,r_i+2)}.
\end{equation}
By definition it is easy to check that $\phi_{a_i}((\alpha\mid_{F_i})\mid^{\hat{T}_i})=\phi_{a_i}(\alpha\mid_{F_i})\mid^{\hat{T}_i}$. Thus combining \eqref{equ-X-S2n-phij-forest}, \eqref{equ-X-S2n-phij-forest-1} we have
\begin{equation}\label{equ-X-S2n-phij-forest-2}
    X_{F_i}^{\phi_{a_i}(\alpha\mid_{F_i})}\ge_{2s} (s_{(k_i+r_i,r_i+1)}+s_{(k_i+r_i-1,r_i+2)})X_{\overline{T_i}}^{\alpha}.
\end{equation}
By \eqref{eq-x-alpha-g+h}, we have
\begin{equation}
    X_F^\alpha+X_F^\beta=\prod_{i=1}^n X_{F_i}^\alpha+\prod_{i=1}^n X_{F_i}^{\phi_{a_i}(\alpha\mid_{F_i})}.
\end{equation}
From \eqref{equ-X-S2n-alpha-forest} and \eqref{equ-X-S2n-phij-forest-2}, we have
\begin{align*}
    X_F^\alpha+X_F^\beta &\ge_{2s}\prod_{i=1}^n (s_{(k_i+r_i,r_i+1)}-s_{(k_i+r_i-1,r_i+2)})X_{\overline{T_i}}^{\alpha}+\prod_{i=1}^n (s_{(k_i+r_i,r_i+1)}+s_{(k_i+r_i-1,r_i+2)})X_{\overline{T_i}}^{\alpha}\\
    &=\left(\prod_{i=1}^n (s_{(k_i+r_i,r_i+1)}-s_{(k_i+r_i-1,r_i+2)})+\prod_{i=1}^n (s_{(k_i+r_i,r_i+1)}+s_{(k_i+r_i-1,r_i+2)})\right)\prod_{i=1}^nX_{\overline{T_i}}^{\alpha}\\
    &\ge_{2s} 0.
\end{align*}
\end{proof}

In fact, Proposition \ref{Prop-con-com-phi-pos} can be generalized as follows.

\begin{cor}\label{cor-forest-2-pos-phi}
Using the same notation as in Proposition \ref{Prop-con-com-phi-pos}, assume that $S_i\subseteq\{1,2,\ldots,k_i\}$ for each $i$ and that $\sum_{i=1}^n \#S_i=n$. Define $\gamma\colon V(F)\rightarrow\mathbb{N}$ as
\begin{equation}
\gamma(v)=
\phi_{S_i}(\alpha)(v),\text{if }v\in V(F_i).
\end{equation}
In other words, $\gamma\mid_{F_i}=\phi_{S_i}(\alpha\mid_{F_i})$. Then we have $X_F^\alpha+X_F^\gamma\ge_{2s} 0$.
\end{cor}

\begin{proof}
It is easy to see that $X_F^\gamma=X_F^\beta$, so the desired result follows from Proposition \ref{Prop-con-com-phi-pos}.
\end{proof}

We conclude this section with the following proposition.

\begin{prop}\label{pro-claw-iso}
Let $F$ be a forest and let $\alpha$ be a map from $V(F)$ to $\mathbb{N}$.
Assume that $X_{F^\alpha}$ is non-2-$s$-positive, and $T$ is the unique connected component of $F^\alpha$ such that $T$ is non-2-$s$-positive. If the unique bipartition of $T$ is $(r+2,r)$ for some $r\ge 1$, then $F^\alpha$ contains no isolated vertices.
\end{prop}
\begin{proof}
Since the  bipartition of $T$ is $(r+2,r)$ for some $r\ge 1$,   Proposition \ref{prop-2s-coe} yields
\begin{equation}
X_T=_{2s} s_{(r+2,r)}-s_{(r+1,r+1)}.
\end{equation}
If $F^\alpha$ contained an isolated vertex $v$, then $X_v=s_1$, and
\begin{equation}
X_vX_T=_{2s}s_1(s_{(r+2,r)}-s_{(r+1,r+1)})=_{2s}s_{(r+3,r)}
\end{equation}
which is 2-$s$-positive. Thus
\[X_{F^\alpha}=X_vX_T\prod_{H}X_H\ge_{2s}0,\]
where product ranges over all the connected components of $F^\alpha$ other than $T$ and $v$. This contradicts the assumption that $X_{F^\alpha}$ is not 2-$s$-positive. Hence $F^\alpha$ contains no isolated vertices.
\end{proof}

\section{The unimodality of the independence polynomial on $T_{3,m,n}$}\label{sec4}

In this section, we present a proof of Theorem \ref{thm-uni-t3mn}.
We first outline the main idea of the proof. For any $\alpha$ for which $T_{3,m,n}^\alpha$ is non-2-$s$-positive, we associate to $\alpha$ a map $\psi(\alpha)$ such that $X_{T_{3,m,n}}^\alpha+X_{T_{3,m,n}}^{\psi(\alpha)}\ge_{2s}0$, except for $\alpha\in N_{30}$ (the explicit definition of $N_{30}$ will be given later in this section). We then show that for each $\alpha\in N_{30}$, the only  term in $X_{T_{3,m,n}}^\alpha$ that is non-$2$-$s$-positive is $-s_{(m+n+5,m+n+5)}$. Consequently $Y_{T_{3,m,n}}$ is $2$-$s$-positive except for the term $s_{(m+n+5,m+n+5)}$. On the other hand, it is evident that the independence number of $T_{3,m,n}$ is $m+n+6$. Thus we can write
\[I_{T_{3,m,n}}(t)=i_0+i_1t+\cdots+i_{m+n+6} t^{m+n+6},\]
 and then by Corollary \ref{lem-skk} we see that the sequence $\{i_k\}_{k=0}^{m+n+5}$ is log-concave. Together with Theorem \ref{thmLM07-2}, this implies that $I_{T_{3,m,n}}(t)$ is unimodal.

We first describe the map $\psi$. To this end, we introduce some notations. We use $G_1$(resp. $G_2$, $G_3$) to denote the induced spider with torso $v_1$(resp. $v_2$, $v_3$) with $3$(resp. $m$, $n$)  legs of length 2,  see Figure \ref{fig-G123} for example. Clearly,  $G_1\cong S(2^3)$, $G_2\cong S(2^m)$ and $G_3\cong S(2^n)$. Given $\alpha:V(T_{3,m,n})\rightarrow\mathbb{N}$, set  $k_i=\#\{j\colon \alpha(v_{ij})=1,\alpha(v_{ij}')=0\}$.   Furthermore, we denote by $A^+$   the set of trees $T$ for which $X_T$ is 2-$s$-positive and by $A^-$   the set of trees $T$ for which $X_T$ is non-2-$s$-positive.

\begin{figure}[h]
    \centering
    \begin{tikzpicture}
[thick, every label/.style={font=\footnotesize}, place/.style={thick,fill=black!100,circle,inner sep=0pt,minimum size=1mm,draw=black!100}]
\node [place,label=above:{$v_0$}] (v4) at (0,2) {};
\node [place,label=left:{$v_2$}] (v5) at (0,1) {};
\node [place,label=left:{$v_{21}$}] (v6) at (-1,0.5) {};
\node [place,label=left:{$v_{21}'$}] (v7) at (-1,0) {};
\node[place] (v14) at (0,0.5) {};
\node[place] (v15) at (0,0) {};
\node[place,label=right:{$\footnotesize{v_{2m}}$}] (v16) at (1,0.5) {};
\node[place,label=right:{$v_{2m}'$}] (v17) at (1,0) {};
\node[place,label=right:{$v_{13}$}] (v10) at (-3,0.5) {};
\node[place,label=right:{$v_{13}'$}] (v11) at (-3,0) {};
\node [place,label=left:{$v_{12}$}] (v8) at (-3.5,0.5) {};
\node [place,label=left:{$v_{12'}$}] (v9) at (-3.5,0) {};
\node[place,label=left:{$v_{11}$}] (v2) at (-4.5,0.5) {};
\node[place,label=left:{$v_{11}'$}] (v1) at (-4.5,0) {};
\node[place,label=left:{$v_{1}$}] (v3) at (-3.5,1) {};
\node[place,label=left:{$v_{31}$}] (v19) at (3,0.5) {};
\node[place,label=left:{$v_{31}'$}] (v20) at (3,0) {};
\node[place] (v21) at (4,0.5) {};
\node[place] (v22) at (4,0) {};
\node[place,label=right:{$v_{3n}$}] (v13) at (5,0.5) {};
\node[place,label=right:{$v_{3n}'$}] (v12) at (5,0) {};
\node[place,label=right:{$v_{3}$}] (v18) at (4,1) {};
\node at (-0.5,0.25) {$\cdots$};
\node at (0.5,0.25) {$\cdots$};
\node at (3.5,0.25) {$\cdots$};
\node at (4.5,0.25) {$\cdots$};
\draw (v1) -- (v2) -- (v3) -- (v4) -- (v18) -- (v13) -- (v12);
\draw (v9) -- (v8) -- (v3) -- (v10) -- (v11);
\draw (v7) -- (v6) -- (v5) -- (v4);
\draw (v15) -- (v14) -- (v5) -- (1,0.5) -- (v17);
\draw (v20) -- (v19) -- (v18) -- (v21) -- (v22);
\node[rotate = 0] at (0,-0.25) {$\underbrace{\hspace{2cm}}$};
\node at (0,-0.5) {\footnotesize $m$};
\node[rotate = 0] at (4,-0.25) {$\underbrace{\hspace{2cm}}$};
\node at (4,-0.5) {\footnotesize $n$};
\node (v23) at (-5.5,1.25) {};
\node (v24) at (-5.5,-0.75) {};
\node (v26) at (-2.25,1.25) {};
\node (v25) at (-2.25,-0.75) {};
\draw [red] [dashed] (v23) -- (v24) -- (v25) -- (v26) -- (v23);
\node at (-4,-1) {$\textcolor{red}{G_1}$};
\node (v27) at (-1.75,1.25) {};
\node (v28) at (-1.75,-0.75) {};
\node (v30) at (1.75,1.25) {};
\node (v29) at (1.75,-0.75) {};
\draw [red] [dashed] (v27) -- (v28) -- (v29) -- (v30) -- (v27);
\node (v31) at (2.25,1.25) {};
\node (v34) at (5.75,1.25) {};
\node (v32) at (2.25,-0.75) {};
\node (v33) at (5.75,-0.75) {};
\draw [red] [dashed] (v31) -- (v32) -- (v33) -- (v34) -- (v31);
\node at (4,-1) {$\textcolor{red}{G_3}$};
\node at (0,-1) {$\textcolor{red}{G_2}$};
\end{tikzpicture}
    \caption{$T_{3,m,n}$}
    \label{fig-G123}
\end{figure}

We next partition the set $N=\{\alpha\colon X_{T_{3,m,n}}^\alpha \text{ is non-2-$s$-positive}\}$ into $30$ disjoint subsets.

\begin{itemize}
\item[(1)] Let $N_1$ denote the subset of $N$ such that $\alpha(v_0)=0$ and $G_1^\alpha\in A^-$, $G_2^\alpha,G_3^\alpha\in A^+$.
\item[(2)] Let $N_2$ denote the subset of $N$ such that $\alpha(v_0)=0$ and $G_2^\alpha\in A^-$, $G_1^\alpha,G_3^\alpha\in A^+$ and $k_2=3$.
\item[(3)] Let $N_3$ denote the subset of $N$ such that $\alpha(v_0)=0$ and $G_3^\alpha\in A^-$, $G_1^\alpha,G_2^\alpha\in A^+$ and $k_3=3$.
\item[(4)] Let $N_4$ denote the subset of $N$ such that $\alpha(v_0)=0$ and $G_2^\alpha\in A^-$, $G_1^\alpha,G_3^\alpha\in A^+$ and $k_2\ge 4$.
\item[(5)] Let $N_5$ denote the subset of $N$ such that $\alpha(v_0)=0$ and $G_3^\alpha\in A^-$, $G_1^\alpha,G_2^\alpha\in A^+$ and $k_3\ge 4$.
\item[(6)] Let $N_6$ denote the subset of $N$ such that $\alpha(v_0)=0$ and $G_1^\alpha,G_2^\alpha\in A^-$, $G_3^\alpha\in A^+$. \item[(7)] Let $N_7$ denote the subset of $N$ such that $\alpha(v_0)=0$ and $G_1^\alpha,G_3^\alpha\in A^-$, $G_2^\alpha\in A^+$.
\item[(8)] Let $N_8$ denote the subset of $N$ such that $\alpha(v_0)=0$ and $G_2^\alpha,G_3^\alpha\in A^-$, $G_1^\alpha\in A^+$.
\item[(9)] Let $N_9$ denote the subset of $N$ such that $\alpha(v_0)=0$ and $G_1^\alpha,G_2^\alpha,G_3^\alpha\in A^-$.
\item[(10)] Let $N_{10}$ denote the subset of $N$ such that $\alpha(v_0)=\alpha(v_1)=1$, $\alpha(v_2)=\alpha(v_3)=0$ and $k_1=2$.
\item[(11)] Let $N_{11}$ denote the subset of $N$ such that $\alpha(v_0)=\alpha(v_1)=1$, $\alpha(v_2)=\alpha(v_3)=0$ and $k_1=3$.
\item[(12)] Let $N_{12}$ denote the subset of $N$ such that $\alpha(v_0)=\alpha(v_2)=1$, $\alpha(v_1)=\alpha(v_3)=0$ and $k_2=2$.
\item[(13)] Let $N_{13}$ denote the subset of $N$ such that $\alpha(v_0)=\alpha(v_2)=1$, $\alpha(v_1)=\alpha(v_3)=0$ and $k_2\ge 3$.
\item[(14)] Let $N_{14}$ denote the subset of $N$ such that $\alpha(v_0)=\alpha(v_3)=1$, $\alpha(v_1)=\alpha(v_2)=0$ and $k_3=2$.
\item[(15)] Let $N_{15}$ denote the subset of $N$ such that $\alpha(v_0)=\alpha(v_3)=1$, $\alpha(v_1)=\alpha(v_2)=0$ and $k_3\ge 3$.
\item[(16)] Let $N_{16}$ denote the subset of $N$ such that $\alpha(v_0)=\alpha(v_2)=\alpha(v_3)=1$, $\alpha(v_1)=0$.
\item[(17)] Let $N_{17}$ denote the subset of $N$ such that $\alpha(v_0)=\alpha(v_1)=\alpha(v_3)=1$, $\alpha(v_2)=0$, $G_2^\alpha\not\in X_{1,m}$ and $k_3\ge 1$.
\item[(18)] Let $N_{18}$ denote the subset of $N$ such that $\alpha(v_0)=\alpha(v_1)=\alpha(v_3)=1$, $\alpha(v_2)=0$, $G_2^\alpha\not\in X_{1,m}$ and $k_3=0$.
\item[(19)] Let $N_{19}$ denote the subset of $N$ such that $\alpha(v_0)=\alpha(v_1)=\alpha(v_3)=1$, $\alpha(v_2)=0$, $G_2^\alpha\in X_{1,m}$ and $k_1\ge 2$.
\item[(20)] Let $N_{20}$ denote the subset of $N$ such that $\alpha(v_0)=\alpha(v_1)=\alpha(v_3)=1$, $\alpha(v_2)=0$, $G_2^\alpha\in X_{1,m}$ and $k_1\le 1$.
\item[(21)] Let $N_{21}$ denote the subset of $N$ such that $\alpha(v_0)=\alpha(v_1)=\alpha(v_2)=1$, $\alpha(v_3)=0$, $G_3^\alpha\not\in X_{1,n}$ and $k_2\ge 2$.
\item[(22)] Let $N_{22}$ denote the subset of $N$ such that $\alpha(v_0)=\alpha(v_1)=\alpha(v_2)=1$, $\alpha(v_3)=0$, $G_3^\alpha\not\in X_{1,n}$, $k_2\le 1$ and $k_1=2$.
\item[(23)] Let $N_{23}$ denote the subset of $N$ such that $\alpha(v_0)=\alpha(v_1)=\alpha(v_2)=1$, $\alpha(v_3)=0$, $G_3^\alpha\not\in X_{1,n}$, $k_2\le 1$ and $k_1=3$.
\item[(24)] Let $N_{24}$ denote the subset of $N$ such that $\alpha(v_0)=\alpha(v_1)=\alpha(v_2)=1$, $\alpha(v_3)=0$, $G_3^\alpha\in X_{1,n}$ and $k_2\ge 1$.
\item[(25)] Let $N_{25}$ denote the subset of $N$ such that $\alpha(v_0)=\alpha(v_1)=\alpha(v_2)=1$, $\alpha(v_3)=0$, $G_3^\alpha\in X_{1,n}$ and $k_2=0$.

\item[(26)] Let \(N_{26}\) denote the subset of \(N\) such that \(\alpha(v_0)=\alpha(v_1)=\alpha(v_2)=\alpha(v_3)=1\) and \(k_1+k_2+k_3\ge 4\), and at least one of the following three conditions holds:
\begin{itemize}
\item[(i)] \(k_1 =3\);
\item[(ii)] \(k_2\ge 2\);
\item[(iii)] \(k_3\ge 2\).
\end{itemize}

\item[(27)] Let \(N_{27}\) denote the subset of \(N\) such that \(\alpha(v_0)=\alpha(v_1)=\alpha(v_2)=\alpha(v_3)=1\) and \(k_1=2\), \(k_2=1\), \(k_3=1\).

\item[(28)] Let \(N_{28}\) denote the subset of \(N\) such that \(\alpha(v_0)=\alpha(v_1)=\alpha(v_2)=\alpha(v_3)=1\), \(k_1+k_2+k_3=0\), there exists a vertex \(v_{ij}\) with \(\alpha(v_{ij})=1\), and \(\sum_{v\in V(T_{3,m,n})}\alpha(v)<2m+2n+10\).

\item[(29)] Let \(N_{29}\) denote the subset of \(N\) such that \(\alpha(v_0)=\alpha(v_1)=\alpha(v_2)=\alpha(v_3)=1\) and \(\alpha(v_{ij})=0\) for all \(v_{ij}\).

\item[(30)] Let \(N_{30}\) denote the subset of \(N\) such that \(\alpha(v_0)=\alpha(v_1)=\alpha(v_2)=\alpha(v_3)=1\), \(k_1+k_2+k_3=0\), there exists a vertex \(v_{ij}\) with \(\alpha(v_{ij})=1\), and \(\sum_{v\in V(T_{3,m,n})}\alpha(v)=2m+2n+10\).
\end{itemize}


Using Corollary \ref{cor-2s-con-bipar}, one may check that $\{N_1,N_2,\ldots, N_{30}\}$  is a set partition of $N$.
We next list $29$ pairwise disjoint subsets of $M=\{\alpha\colon T_{3,m,n}^\alpha \text{ is 2-$s$-positive}\}.$ To describe these subsets, we let $Z$  denote the set of all 2-$s$-positive forests that have no isolated vertices. We also recall that $X_{S,n}^t$ is defined in Definition \ref{defi-xtsn}.
\begin{itemize}
\item[(1)] Let $M_{1}$ denote the subset of $M$ such that $\alpha(v_0)=0$, $G_1^\alpha\in X^3_{1,3}$, $G_2^\alpha,G_3^\alpha\in Z$.
\item[(2)] Let $M_{2}$ denote the subset of $M$ such that $\alpha(v_0)=0$, $G_1^\alpha,G_3^\alpha\in Z$ and $G_2^\alpha\in X_{1,m}^3$.
\item[(3)] Let $M_{3}$ denote the subset of $M$ such that $\alpha(v_0)=0$, $G_1^\alpha,G_2^\alpha\in Z$  and $G_3^\alpha\in X_{1,n}^3$.
\item[(4)] Let $M_{4}$ denote the subset of $M$ such that $\alpha(v_0)=0$, $G_1^\alpha\in A^+$, $G_2^\alpha\in X_{i,m}^r$  and $G_3^\alpha\in X_{j,n}$, where $r\ge 4$, $i=3$ or $4$, $j\ge 0$ and $i+j$ odd.
\item[(5)] Let $M_{5}$ denote the subset of $M$ such that $\alpha(v_0)=0$, $G_1^\alpha\in A^+$, $G_2^\alpha\in X_{i,m}$  and $G_3^\alpha\in X_{j,n}^r$, where $r\ge 4$, $j=3$ or $4$, $i\ge 0$ and $i+j$ even.
\item[(6)] Let $M_{6}$ denote the subset of $M$ such that $\alpha(v_0)=0$, $G_1^\alpha\in X_{1,3}^3$, $G_2^\alpha\in X_{i,m}^r$  and $G_3^\alpha\in X_{j,n}$, where $r\ge 3$,  $i=1$ or $2$, $j\ge 0$ and $i+j$ odd.
\item[(7)] Let $M_{7}$ denote the subset of $M$ such that $\alpha(v_0)=0$, $G_1^\alpha\in X_{1,3}^3$, $G_2^\alpha\in X_{i,m}$  and $G_3^\alpha\in X_{j,n}^r$, where  $r\ge 3$, $j=1$ or $2$, $i\ge 0$ and $i+j$ even.
\item[(8)] Let $M_{8}$ denote the subset of $M$ such that $\alpha(v_0)=0$, $G_1^\alpha\in A^+$, $G_2^\alpha\in X_{\{1,2\},m}^r$  and $G_3^\alpha\in X_{\emptyset,n}^s$, where $r\ge 3$, $s\ge 3$.
\item[(9)] Let $M_{9}$ denote the subset of $M$ such that $\alpha(v_0)=0$, $G_1^\alpha\in X_{\emptyset,3}^3$, $G_2^\alpha\in X_{\{1,2,3\},m}^r$  and $G_3^\alpha\in X_{\emptyset,n}^s$, where $r\ge 3$, $s\ge 3$.
\item[(10)] Let $M_{10}$ denote the subset of $M$ such that $\alpha(v_0)=1$, $G_1^\alpha\in X_{1,3}^2$ and $G_2^\alpha,G_3^\alpha\in Z$.
\item[(11)] Let $M_{11}$ denote the subset of $M$ such that $\alpha(v_0)=\alpha(v_1)=\alpha(v_{11})=\alpha(v_{11}')=\alpha(v_{12})=1$ and $\alpha(v_{12}')=\alpha(v_{13})=\alpha(v_{13}')=\alpha(v_2)=\alpha(v_3)=0$.
\item[(12)] Let $M_{12}$ denote the subset of $M$ such that $\alpha(v_0)=1$, $\alpha(v_1)=\alpha(v_2)=\alpha(v_3)=0$, $G_2^\alpha\in X_{1,m}^2$ and $G_1^\alpha,G_3^\alpha\in Z$.
\item[(13)] Let $M_{13}$ denote the subset of $M$ such that $\alpha(v_0)=1$, $\alpha(v_1)=\alpha(v_2)=\alpha(v_3)=0$, $G_2^\alpha\in X_{i,m}^r$ and $G_3^\alpha\in X_{j,n}$, where $r\ge 3$, $i=2$ or $3$ and $i+j$ odd.
\item[(14)] Let $M_{14}$ denote the subset of $M$ such that $\alpha(v_0)=1$, $\alpha(v_1)=\alpha(v_2)=\alpha(v_3)=0$, $G_1^\alpha,G_2^\alpha\in Z$ and $G_3^\alpha\in X_{1,n}^2$.
\item[(15)] Let $M_{15}$ denote the subset of $M$ such that $\alpha(v_0)=1$, $\alpha(v_1)=\alpha(v_2)=\alpha(v_3)=0$, $G_2^\alpha\in X_{i,m}$ and $G_3^\alpha\in X_{j,n}^r$, where $r\ge 3$, $j=2$ or $3$ and $i+j$ even.
\item[(16)] Let $M_{16}$ denote the subset of $M$ such that $\alpha(v_0)=1$, $\alpha(v_1)=\alpha(v_2)=\alpha(v_3)=0$, either $G_2^\alpha\in X_{\{1,2\},m}$ and $G_3^\alpha\in X_{\emptyset,n}$ or $G_2^\alpha\in X_{\emptyset,m}$ and $G_3^\alpha\in X_{\{1,2\},n}$.
\item[(17)] Let $M_{17}$ denote the subset of $M$ such that $\alpha(v_0)=2$, $\alpha(v_1)=\alpha(v_2)=\alpha(v_3)=0$, $G_1^\alpha\in X_{\emptyset, 3}$, $G_2^\alpha\not\in X_{1,m}$ and $G_3^\alpha\in X_{1,n}$. Moreover, $k_1+k_3\ge 2$.
\item[(18)] Let $M_{18}$ denote the subset of $M$ such that   $\alpha(v_0)=\alpha(v_1)=\alpha(v_{11})=\alpha(v_{11}')=\alpha(v_{12})=\alpha(v_{12}')=1$ and $\alpha(v_{13})=\alpha(v_{13}')=\alpha(v_2)=0$, $G_2^\alpha\not\in X_{1,m}$ and $G_3^\alpha\in X_{\emptyset,n}$.
\item[(19)] Let $M_{19}$ denote the subset of $M$ such that $\alpha(v_0)=2$, $\alpha(v_1)=\alpha(v_2)=\alpha(v_3)=0$, $G_1^\alpha\in X_{2, 3}$, $G_2^\alpha\in X_{1,m}$ and $G_3^\alpha\in X_{\emptyset,n}$.
\item[(20)] Let $M_{20}$ denote the subset of $M$ such that $\alpha(v_0)=2$, $\alpha(v_1)=\alpha(v_2)=\alpha(v_3)=0$, $G_1^\alpha\in X_{\emptyset, 3}$, $G_2^\alpha\in X_{1,m}$ and $G_3^\alpha\in X_{2,n}^r$, where $r\ge 2$.
\item[(21)] Let $M_{21}$ denote the subset of $M$ such that $\alpha(v_0)=2$, $\alpha(v_1)=\alpha(v_2)=\alpha(v_3)=0$, $G_1^\alpha\in X_{\emptyset, 3}$, $G_2^\alpha\in X_{2,m}$ and $G_3^\alpha\not\in X_{1,n}$.
\item[(22)] Let $M_{22}$ denote the subset of $M$ such that $\alpha(v_0)=2$, $\alpha(v_1)=\alpha(v_2)=\alpha(v_3)=0$, $G_1^\alpha\in X_{\emptyset, 3}$, $G_2^\alpha\in X^1_{1,m}$ and $G_3^\alpha\not\in X_{1,n}\cup X_{2,n}$.
\item[(23)] Let $M_{23}$ denote the subset of $M$ such that $\alpha(v_0)=\alpha(v_1)=\alpha(v_{11})=\alpha(v_{11}')=\alpha(v_{12}')=\alpha(v_{13})=1$ and $\alpha(v_{12})=\alpha(v'_{13})=\alpha(v_2)=\alpha(v_3)=0$. Moreover $G_2^\alpha\in X_{\emptyset,m}$ and $G_3^\alpha\not\in X_{1,n}$.
\item[(24)] Let $M_{24}$ denote the subset of $M$ such that  $\alpha(v_0)=2$, $\alpha(v_1)=\alpha(v_2)=\alpha(v_3)=0$, $G_1^\alpha\in X_{\emptyset, 3}$, $G_2^\alpha\in X_{1,m}$ and $G_3^\alpha\in X_{1,n}$.
\item[(25)] Let $M_{25}$ denote the subset of $M$ such that $\alpha(v_0)=\alpha(v_1)=\alpha(v_{12})=\alpha(v_{13})=\alpha(v_{11}')=\alpha(v_{12}')=1$ and $\alpha(v_{11})=\alpha(v'_{13})=\alpha(v_2)=\alpha(v_3)=0$. Moreover $G_2^\alpha\in X_{\emptyset,m}$ and $G_3^\alpha\in X_{1,n}$.
\item[(26)]Let $M_{26}$ denote the subset of $M$ such that $\alpha(v_0)=2$, $\alpha(v_1)=\alpha(v_2)=\alpha(v_3)=0$. Moreover, $G_1^\alpha\in X_{P_1,3}$, $G_2^\alpha\in X_{P_2,m}$, $G_3^\alpha\in X_{P_3,n}$, where one of the following three restrictions holds:
\begin{itemize}
\item[(i)] $G_1^\alpha\in X_{\{1,2\},3}^3$, $G_2^\alpha\in X_{\emptyset,m}$ and $G_3^\alpha\in X_{\emptyset,n}$;
\item[(ii)] $G_1^\alpha\in X_{\emptyset,3}$, $G_2^\alpha\in X_{\{1,2\},m}$ and $G_3^\alpha\in X_{\emptyset,n}$;
\item[(iii)] $G_1^\alpha\in X_{\emptyset,3}$, $G_2^\alpha\in X_{\emptyset,m}$ and $G_3^\alpha\in X_{\{1,2\},n}$.
\end{itemize}
\item[(27)]Let $M_{27}$ denote the subset of $M$ such that $\alpha(v_0)=2$, $\alpha(v_1)=\alpha(v_2)=\alpha(v_3)=0$. Moreover, $G_1^\alpha\in X^2_{1,3}$, $G_2^\alpha\in X^1_{\emptyset,m}$, $G_3^\alpha\in X^1_{1,n}$.
\item[(28)] Let $M_{28}$ denote the subset of $M$ such that $\alpha(v_0)=\alpha(v_1)=\alpha(v_{2})=\alpha(v_{3})=1$ and $k_1+k_2+k_3=1$. Moreover there exists a unique $(i,j)$ where $1\le i\le 3$ such that $\alpha(v_{ij})=0$ and $\alpha(v_{ij}')=1$.
\item[(29)] Let $M_{29}$ denote the subset of $M$ such that $\alpha(v_0)=0$, $\alpha(v_1)=2$, $\alpha(v_{2})=\alpha(v_{3})=1$ and $\alpha(v_{ij})=0$ for any $v_{ij}$.
\end{itemize}
Using Corollary \ref{cor-xsn-distinct} and the fact that $Z\cap X_{S,n}^k=\emptyset$ unless $S=\{1,2,\ldots,k\}$, it is routine to check that $\{M_i\}_{i=1}^{29}$ are pairwise disjoint.

We next construct $29$ injections $\psi_i\colon N_i\rightarrow M_i$ for $1\le i\le 29$. Moreover, for $\alpha\in N_{30}$, we show that in Lemma \ref{lem-29} that the only term in $X_{T_{3,m,n}}^\alpha$ that is non-$2$-$s$-positive    is $-s_{(m+n+5,m+n+5)}$. This implies the proof of Theorem \ref{thm-uni-t3mn}.

We now describe each injection $\psi_i$.

\begin{lem}\label{lem-psi1}
There is an injection $\psi_1$ from $N_1$ to $M_1$. Moreover, $X_{T_{3,m,n}}^\alpha+X_{T_{3,m,n}}^{\psi_1(\alpha)}$ is 2-$s$-positive.
\end{lem}

\begin{proof}
For $\alpha\in N_1$, by definition we have $G_1^\alpha\in A^-$ and $G_2^\alpha,G_3^\alpha\in A^+$. From Proposition \ref{pro-xgn+xgalpha}, $G_1^\alpha\in A^-$ implies that $\alpha(v_1)=1$ and $k_1=3$. Hence $G_1^\alpha\cong S(1^3)$ and the bipartition of $G_1^\alpha$ is $(3,1)$. Using Proposition \ref{pro-claw-iso}, we deduce that there are no isolated vertices in $G_2^\alpha$ and $G_3^\alpha$. Hence $G_2^\alpha, G_3^\alpha\in Z$. Define $\psi_1(\alpha)$ as follows:
\begin{equation}
\psi_1(\alpha)(v)=\begin{cases}
\phi_1(\alpha\mid_{G_1})(v),&\text{if }v\in V(G_1);\\
\alpha(v),& \text{otherwise.}
\end{cases}
\end{equation}
Set $\beta=\psi_1(\alpha)$. It is evident that $\alpha\mid_{G_1}\in \mathcal{A}_3$. From the construction of $\psi_1$, we see that $\beta\mid_{G_1}=\phi_1(\alpha\mid_{G_1})$. By Proposition \ref{pro-bijection},  $\beta\mid_{G_1}\in x_{1,3}^3$, so $G_1^\beta=G_1^{\beta\mid_{G_1}}\in X_{1,3}^3$.  Moreover, clearly  $G_2^\beta=G_2^\alpha\in Z$ and $G_3^\beta=G_3^\alpha\in Z$. Hence $\beta\in M_1$.

We proceed to show that $\psi_1$ is an injection. Let $I_1=\{\psi_1(\alpha)\colon \alpha\in N_1\}\subseteq M_1$. For any $\beta\in I_1$, define $\varphi_1\colon I_1\rightarrow N_1$ by
\begin{equation}
    \varphi_1(\beta)(v)=\begin{cases}
        \phi^{-1}(\beta\mid_{G_1})(v),& \text{if }v\in V(G_1);\\
        \beta(v),& \text{otherwise.}
    \end{cases}
\end{equation}
Here $\phi^{-1}$ is as defined in \eqref{eq-def-phi-s-1} and Remark \ref{rem-phi-1}. One easily checks that  $\varphi_1(\psi_1(\alpha))=\alpha$ for any $\alpha\in N_1$. This implies $\psi_1$ is an injection.

We next calculate $X_{T_{3,m,n}}^\alpha+X_{T_{3,m,n}}^\beta$. Note that $\alpha\mid_{G_1}\in \mathcal{A}_3$ and  $\beta\mid_{G_1}=\phi_1(\alpha\mid_{G_1})$.  By Proposition \ref{pro-xgn+xgalpha},
\begin{equation}\label{ine-xt1+xt1'2s0}
X_{G_1}^\alpha+X_{G_1}^{\beta}=X_{G_1}^{\alpha\mid_{G_1}}+X_{G_1}^{\phi_1(\alpha\mid_{G_1})}\ge_{2s} 0.
\end{equation}
Since $\alpha(v)=\beta(v)$ for all $v\not\in G_1$, we have $X_{G_2}^\alpha=X_{G_2}^{\beta}\ge_{2s}0$ and $X_{G_3}^\alpha=X_{G_3}^{\beta}\ge_{2s}0$. Thus from Proposition \ref{prop-42},
\begin{align*}
X_{T_{3,m,n}}^{\alpha}+X_{T_{3,m,n}}^{\psi_1(\alpha)}&=X_{G_1}^\alpha X_{G_2}^\alpha X_{G_3}^\alpha+X_{G_1}^\beta X_{G_2}^\beta X_{G_3}^\beta\\
&=(X_{G_1}^{\alpha}+X_{G_1}^\beta)X_{G_2}^\alpha X_{G_3}^\alpha\ge_{2s}0.
\end{align*}
The last inequality follows from \eqref{ine-xt1+xt1'2s0}.
\end{proof}

\begin{lem}\label{lem-psi2}
There is an injection $\psi_2$ from $N_2$ to $M_2$. Moreover, $X_{T_{3,m,n}}^\alpha+X_{T_{3,m,n}}^{\psi_2(\alpha)}$ is 2-$s$-positive.
\end{lem}

\begin{proof}
Given $\alpha\in N_2$, by definition we see that $G_2^\alpha\in A^-$, $G_1^\alpha,G_3^\alpha\in A^+$ and $k_2=3$. From Proposition \ref{pro-xgn+xgalpha}, $G_2^\alpha\in A^-$ implies that $\alpha(v_2)=1$. Let ${T}_2$ denote the connected component of $G_2^\alpha$ which contains $v_2^\alpha$. From Proposition \ref{prop-416} and $k_2=3$, we see that ${T}_2\cong S(1^3,2^r)$ for some $r\ge 0$. Hence the bipartition of ${T}_2$ is $(r+3,r+1)$.   Using Proposition \ref{pro-claw-iso} we derive that there are no isolated vertices in $G_1^\alpha$  or $G_3^\alpha$. Hence $G_1^\alpha, G_3^\alpha\in Z$. Define $\psi_2(\alpha)$ as follows:
\begin{equation}
\psi_2(\alpha)(v)=\begin{cases}
\phi_1(\alpha\mid_{G_2})(v),&\text{if }v\in V(G_2);\\
\alpha(v),& \text{otherwise.}
\end{cases}
\end{equation}
Set $\beta=\psi_2(\alpha)$. Since $G_2^\alpha\in A^-$, from Proposition \ref{pro-xgn+xgalpha} and $k_2=3$, we see that $\alpha\mid_{G_2}\in\mathcal{A}_3$. By Proposition \ref{pro-bijection}, we have $\phi_1(\alpha\mid_{G_2})\in x_{1,m}^3$, this implies $G_2^\beta=G_2^{\phi_1(\alpha\mid_{G_2})}\in X_{1,m}^3$. Moreover, $G_1^\beta=G_1^\alpha\in Z$ and $G_3^\beta=G_3^\alpha\in Z$. This yields $\beta\in M_2$.

We proceed to show that $\psi_2$ is an injection. Let $I_2=\{\psi_2(\alpha)\colon \alpha\in N_2\}\subseteq M_2$. For any $\beta\in I_2$, we construct a map $\varphi_2\colon I_2\rightarrow N_2$ as follows:
\begin{equation}
    \varphi_2(\beta)(v)=\begin{cases}
        \phi^{-1}(\beta\mid_{G_2})(v),& \text{if }v\in V(G_2);\\
        \beta(v),& \text{otherwise.}
    \end{cases}
\end{equation}
 It is easy to check that $\varphi_2(\psi_2(\alpha))=\alpha$ for any $\alpha\in N_2$. This implies $\psi_2$ is an injection.

 Moreover, notice that $G_2^{\alpha\mid_{G_2}}\in A^-$ and  $X_{G_2}^\beta=X_{G_2}^{\phi_1(\alpha\mid_{G_2})}$. Using Proposition \ref{pro-xgn+xgalpha}, we see that
\begin{equation}\label{ine-xt1+xt1'2s0-2}
X_{G_2}^\alpha+X_{G_2}^{\beta}=X_{G_2}^{\alpha\mid_{G_2}}+X_{G_2}^{\phi_1(\alpha\mid_{G_2})}\ge_{2s} 0.
\end{equation}
Since for any $v\not\in G_2$, $\alpha(v)=\psi_2(\alpha)(v)$, we see that $X_{G_1}^\alpha=X_{G_1}^{\beta}\ge_{2s}0$ and $X_{G_3}^\alpha=X_{G_3}^{\beta}\ge_{2s}0$. Thus
\begin{align*}
X_{T_{3,m,n}}^{\alpha}+X_{T_{3,m,n}}^{\psi_2(\alpha)}&=X_{G_1}^\alpha X_{G_2}^\alpha X_{G_3}^\alpha+X_{G_1}^\beta X_{G_2}^\beta X_{G_3}^\beta\\
&=(X_{G_2}^\alpha+X_{G_2}^\beta)X_{G_1}^\alpha X_{G_3}^\alpha\ge_{2s}0.
\end{align*}
\end{proof}

\begin{lem}
There is an injection $\psi_3$ from $N_3$ to $M_3$. Moreover, $X_{T_{3,m,n}}^\alpha+X_{T_{3,m,n}}^{\psi_3(\alpha)}$ is 2-$s$-positive.
\end{lem}

\begin{proof}
The proof of this lemma is analogous to that   of Lemma \ref{lem-psi2}. We give the explicit definition of $\psi_3$ and omit the detail.
\begin{equation}
\psi_3(\alpha)(v)=\begin{cases}
\phi_1(\alpha\mid_{G_3})(v),&\text{if }v\in V(G_3);\\
\alpha(v),& \text{otherwise.}
\end{cases}
\end{equation}
\end{proof}

\begin{lem}\label{lem-psi4}
There is an injection $\psi_4$ from $N_4$ to $M_4$. Moreover, $X_{T_{3,m,n}}^\alpha+X_{T_{3,m,n}}^{\psi_4(\alpha)}$ is 2-$s$-positive.
\end{lem}

\begin{proof}
Given $\alpha\in N_4$, by definition we see that $G_2^\alpha\in A^-$, $G_1^\alpha,G_3^\alpha\in A^+$ and $k_2\ge 4$. Combining $G_2^\alpha\in A^-$ and Proposition \ref{pro-xgn+xgalpha}, we find that $\alpha\mid_{G_2}\in \mathcal{A}_{k_2}$. From \eqref{equ-def-x0n} and Corollary \ref{cor-xsn-distinct}, we see that there exists a unique $j\ge 0$ such that $G_3^\alpha\in X_{j,n}$.  We define the map $\psi_4(\alpha): V(G) \to \mathbb{N}$ by distinguishing the parity of the index $j$:
 \begin{equation}
\psi_4(\alpha)(v)=\begin{cases}
\phi_3(\alpha\mid_{G_2})(v),&\text{if }v\in V(G_2) \text{ and } j \text{ even};\\
\phi_4(\alpha\mid_{G_2})(v),&\text{if }v\in V(G_2) \text{ and } j \text{ odd};\\
\alpha(v),& \text{otherwise.}
\end{cases}
\end{equation}
Set $\beta=\psi_4(\alpha)$.  From the construction of $\psi_4$, we have $\beta\mid_{G_2}=\phi_3(\alpha\mid_{G_2})\in x_{3,m}^{k_2}$ when $j$ is even and $\beta\mid_{G_2}=\phi_4(\alpha\mid_{G_2})\in x_{4,m}^{k_2}$ when $j$ is odd. This yields $G_2^\beta\in X^{k_2}_{3,m}$ when $j$ even and $G_2^\beta\in X^{k_2}_{4,m}$ when $j$ odd.
  Moreover, $G_1^\beta=G_1^\alpha\in A^+$. Hence $\beta\in M_4$.

We next show that $\psi_4$ is an injection.  Let $I_4=\{\psi_4(\alpha)\colon \alpha\in N_4\}\subseteq M_4$. For any $\beta\in I_4$, we construct a map $\varphi_4\colon I_4\rightarrow N_4$ as follows:
\begin{equation}
\varphi_4(\beta)(v)=\begin{cases}
\phi^{-1}(\beta\mid_{G_2})(v),&\text{if }v\in V(G_2);\\
\beta(v),& \text{otherwise.}
\end{cases}
\end{equation}
Since $\phi^{-1}$ is the inverse map of both $\phi_3$ and $\phi_4$, it is easy to check that  $\varphi_4(\psi_4(\alpha))=\alpha$ for any $\alpha\in N_4$. This yields $\psi_4$ is an injection.

Moreover,  from $G_2^\alpha\in A^-$ and $\beta\mid_{G_2}=\phi_3(\alpha\mid_{G_2})$ or $\beta\mid_{G_2}=\phi_4(\alpha\mid_{G_2})$, using Proposition \ref{pro-xgn+xgalpha}, in either case we see that
\begin{equation}\label{ine-xt1+xt1'2s0-4}
X_{G_2}^\alpha+X_{G_2}^{\beta}=X_{G_2}^{\alpha\mid_{G_2}}+X_{G_2}^{\beta\mid_{G_2}}\ge_{2s} 0.
\end{equation}
Since for any $v\not\in G_2$, $\alpha(v)=\psi_4(\alpha)(v)$, we see that $X_{G_1}^\alpha=X_{G_1}^{\beta}\ge_{2s}0$ and $X_{G_3}^\alpha=X_{G_3}^{\beta}\ge_{2s}0$. Thus
\begin{align*}
X_{T_{3,m,n}}^{\alpha}+X_{T_{3,m,n}}^{\psi_4(\alpha)}&=X_{G_1}^\alpha X_{G_2}^\alpha X_{G_3}^\alpha+X_{G_1}^\beta X_{G_2}^\beta X_{G_3}^\beta\\
&=(X_{G_2}^\alpha+X_{G_2}^\beta)X_{G_1}^\alpha X_{G_3}^\alpha\ge_{2s}0.
\end{align*}
\end{proof}

\begin{lem}
There is an injection $\psi_5$ from $N_5$ to $M_5$. Moreover, $X_{T_{3,m,n}}^\alpha+X_{T_{3,m,n}}^{\psi_5(\alpha)}$ is 2-$s$-positive.
\end{lem}

\begin{proof}
The proof is analogous to that of Lemma \ref{lem-psi4}. We give the explicit map of $\psi_5$ and omit the other detail.
\begin{equation}
\psi_5(\alpha)(v)=\begin{cases}
\phi_4(\alpha\mid_{G_2})(v),&\text{if }v\in V(G_2) \text{ and } j \text{ even};\\
\phi_3(\alpha\mid_{G_2})(v),&\text{if }v\in V(G_2) \text{ and } j \text{ odd};\\
\alpha(v),& \text{otherwise.}
\end{cases}
\end{equation}
\end{proof}

\begin{lem}\label{lem-psi6}
There is an injection $\psi_6$ from $N_6$ to $M_6$. Moreover, $X_{T_{3,m,n}}^\alpha+X_{T_{3,m,n}}^{\psi_6(\alpha)}$ is 2-$s$-positive.
\end{lem}

\begin{proof}
Given $\alpha\in N_6$, by definition we see that $G_1^\alpha, G_2^\alpha\in A^-$, $G_3^\alpha\in A^+$.   From $G_1^\alpha\in A^-$ we have $G_1^\alpha\cong S(1^3)$. Moreover, from $G_2^\alpha\in A^-$ and Proposition \ref{pro-xgn+xgalpha} we see that $\alpha\mid_{G_2}\in \mathcal{A}_{k_2}$ and $k_2\ge 3$. Similar as in Lemma \ref{lem-psi4}, we see that there exists a unique $j\ge 0$ satisfies that $G_3^\alpha\in X_{j,n}$.  Define $\psi_6(\alpha)$ as follows:
\begin{equation}
\psi_6(\alpha)(v)=\begin{cases}
\phi_1(\alpha\mid_{G_1})(v),&\text{if }v\in V(G_1);\\
\phi_2(\alpha\mid_{G_2})(v),&\text{if }v\in V(G_2) \text{ and } j \text{ odd};\\
\phi_1(\alpha\mid_{G_2})(v),&\text{if }v\in V(G_2) \text{ and } j \text{ even};\\
\alpha(v),& \text{otherwise.}
\end{cases}
\end{equation}
Set $\beta=\psi_6(\alpha)$.  From the construction of $\psi_6$, we have $\beta\mid_{G_1}=\phi_1(\alpha\mid_{G_1})\in x_{1,3}^3$, thus $G_1^\beta\in X^{3}_{1,3}$. Moreover, $\beta\mid_{G_2}=\phi_2(\alpha\mid_{G_2})\in x_{2,m}^{k_2}$ when $j$ is odd and $\beta\mid_{G_2}=\phi_1(\alpha\mid_{G_2})\in x_{1,m}^{k_2}$ when $j$ is even. Thus $G_2^\beta\in X_{2,m}^{k_2}$ when $j$ odd and $G_2^\beta\in X_{1,m}^{k_2}$ when $j$ even.
  This implies $\beta\in M_6$.

We next show that $\psi_6$ is an injection.  Let $I_6=\{\psi_6(\alpha)\colon \alpha\in N_6\}\subseteq M_6$. For any $\beta\in I_6$, we construct a map $\varphi_6\colon I_6\rightarrow N_6$ as follows:
\begin{equation}
\varphi_6(\beta)(v)=\begin{cases}
\phi^{-1}(\beta\mid_{G_1})(v),&\text{if }v\in V(G_1);\\
\phi^{-1}(\beta\mid_{G_2})(v),&\text{if }v\in V(G_2);\\
\beta(v),& \text{otherwise.}
\end{cases}
\end{equation}
From the definition of $\phi^{-1}$ we see that $\phi^{-1}$ is the inverse map of both $\phi_1$ and $\phi_2$. Thus it is easy to check that $\varphi_6(\psi_6(\alpha))=\alpha$ for any $\alpha\in N_6$. This yields $\psi_6$ is an injection.

Note that $G_1^\alpha,G_2^{\alpha}\in A^-$ imply   $\alpha\mid_{G_1}\in\mathcal{A}_3$ and $\alpha\mid_{G_2}\in\mathcal{A}_{k_2}$. Moreover $\beta\mid_{G_1}=\phi_1(\alpha\mid_{G_1})$, and either $\beta\mid_{G_2}=\phi_1(\alpha\mid_{G_2})$ or $\beta\mid_{G_2}=\phi_2(\alpha\mid_{G_2})$.   Let $F$ be the subgraph of $T_{3,m,n}$ induced by $V(G_1)\cup V(G_2)$. From Proposition \ref{Prop-con-com-phi-pos}, in either case we see that
\begin{equation}\label{ine-xt1+xt1'2s0-6}
X_{F}^\alpha+X_{F}^{\beta}\ge_{2s} 0.
\end{equation}

Since for any $v\in V(G_3)$, $\alpha(v)=\psi_6(\alpha)(v)$, we see that $X_{G_3}^\alpha=X_{G_3}^{\beta}\ge_{2s}0$. Thus
\begin{align*}
X_{T_{3,m,n}}^{\alpha}+X_{T_{3,m,n}}^{\psi_6(\alpha)}&=X_{F}^\alpha  X_{G_3}^\alpha+X_{F}^\beta  X_{G_3}^\beta\\
&=(X_{F}^\alpha+X_{F}^\beta) X_{G_3}^\alpha\ge_{2s}0.
\end{align*}
\end{proof}

\begin{lem}\label{lem-psi7}
There is an injection $\psi_7$ from $N_7$ to $M_7$. Moreover, $X_{T_{3,m,n}}^\alpha+X_{T_{3,m,n}}^{\psi_7(\alpha)}$ is 2-$s$-positive.
\end{lem}

\begin{proof}
The proof of this lemma is analogue to the proof of Lemma \ref{lem-psi6}. Assume $G_2^\alpha\in X_{j,m}$, where $0\le j\le m$. We give the explicit map of $\psi_7$ and omit the other detail.
\begin{equation}
\psi_7(\alpha)(v)=\begin{cases}
\phi_1(\alpha\mid_{G_1})(v),&\text{if }v\in V(G_1);\\
\phi_2(\alpha\mid_{G_3})(v),&\text{if }v\in V(G_3) \text{ and } j \text{ even};\\
\phi_1(\alpha\mid_{G_3})(v),&\text{if }v\in V(G_3) \text{ and } j \text{ odd};\\
\alpha(v),& \text{otherwise.}
\end{cases}
\end{equation}
\end{proof}

\begin{lem}\label{lem-psi8}
There is an injection $\psi_8$ from $N_8$ to $M_8$. Moreover, $X_{T_{3,m,n}}^\alpha+X_{T_{3,m,n}}^{\psi_8(\alpha)}$ is 2-$s$-positive.
\end{lem}

\begin{proof}
Given $\alpha\in N_8$, by definition we see that $G_2^\alpha, G_3^\alpha\in A^-$, $G_1^\alpha\in A^+$.  From $G_2^\alpha,G_3^\alpha\in A^-$ and Proposition \ref{pro-xgn+xgalpha}, we see that $\alpha\mid_{G_2}
\in \mathcal{A}_{k_2}$, $\alpha\mid_{G_3}
\in \mathcal{A}_{k_3}$, $k_2\ge 3$ and $k_3\ge 3$.  Define $\psi_8(\alpha)$ as follows:
\begin{equation}
\psi_8(\alpha)(v)=\begin{cases}
\phi_{\{1,2\}}(\alpha\mid_{G_2})(v),&\text{if }v\in V(G_2);\\
\phi_\emptyset(\alpha\mid_{G_3})(v),&\text{if }v\in V(G_3);\\
\alpha(v),& \text{otherwise.}
\end{cases}
\end{equation}
Set $\beta=\psi_8(\alpha)$. From the construction of $\psi_8$, using the same argument as in Lemma \ref{lem-psi4}, we see that $G_1^\beta=G_1^\alpha\in A^+$, $G_2^\beta\in X_{\{1,2\},m}^{k_2}$ and $G_3^\beta\in X_{\emptyset,n}^{k_3}$.   This implies $\beta\in M_8$.

We next show that $\psi_8$ is an injection. Let $I_8=\{\psi_8(\alpha)\colon \alpha\in N_8\}\subseteq M_8$. For any $\beta\in I_8$, we construct a map $\varphi_8\colon I_8\rightarrow N_8$ as follows:
\begin{equation}
\varphi_8(\beta)(v)=\begin{cases}
\phi^{-1}(\beta\mid_{G_2})(v),&\text{if }v\in V(G_2);\\
\phi^{-1}(\beta\mid_{G_3})(v),&\text{if }v\in V(G_3);\\
\beta(v),& \text{otherwise.}
\end{cases}
\end{equation}
From the definition of $\phi^{-1}$ we see that $\phi^{-1}$ is the inverse map of both $\phi_{\{1,2\}}$ and $\phi_\emptyset$. Thus it is easy to check that $\varphi_8(\psi_8(\alpha))=\alpha$ for any $\alpha\in N_8$. This yields $\psi_8$ is an injection.

Note that $G_2^\alpha\in \mathcal{A}_{k_2}$, $G_3^\alpha\in \mathcal{A}_{k_3}$ and $\beta\mid_{G_2}=\phi_{\{1,2\}}(\alpha\mid_{G_2})$, $\beta\mid_{G_3}=\phi_{\emptyset}(\alpha\mid_{G_3})$. From Corollary \ref{cor-forest-2-pos-phi}, if we let $F$ be the subgraph of $T_{3,m,n}$  induced by $V(G_2)\cup V(G_3)$, then
\begin{equation}\label{ine-xt1+xt1'2s0-8}
X_{F}^\alpha+X_{F}^{\beta}\ge_{2s} 0.
\end{equation}

Since for any $v\in V(G_1)$, $\alpha(v)=\psi_8(\alpha)(v)$, we see that $X_{G_1}^\alpha=X_{G_1}^{\beta}\ge_{2s}0$. Thus
\begin{align*}
X_{T_{3,m,n}}^{\alpha}+X_{T_{3,m,n}}^{\psi_8(\alpha)}&=X_{F}^\alpha  X_{G_1}^\alpha+X_{F}^\beta  X_{G_1}^\beta\\
&=(X_{F}^\alpha+X_{F}^\beta) X_{G_1}^\alpha\ge_{2s}0.
\end{align*}
\end{proof}

\begin{lem}\label{lem-psi9}
There is an injection $\psi_9$ from $N_9$ to $M_9$. Moreover, $X_{T_{3,m,n}}^\alpha+X_{T_{3,m,n}}^{\psi_9(\alpha)}$ is 2-$s$-positive.
\end{lem}

\begin{proof}
The proof of this lemma is analogue to the proof of Lemma \ref{lem-psi8}. We give the explicit map of $\psi_9$ and omit the other detail.
\begin{equation}
\psi_9(\alpha)(v)=\begin{cases}
\phi_{\{1,2,3\}}(\alpha\mid_{G_2})(v),&\text{if }v\in V(G_2);\\
\phi_\emptyset(\alpha\mid_{G_1})(v),&\text{if }v\in V(G_1);\\
\phi_\emptyset(\alpha\mid_{G_3})(v),&\text{if }v\in V(G_3);\\
0,&\text{if }v=v_0.
\end{cases}
\end{equation}
\end{proof}

\begin{lem}\label{lem-psi10}
There is an injection $\psi_{10}$ from $N_{10}$ to $M_{10}$. Moreover, $X_{T_{3,m,n}}^\alpha+X_{T_{3,m,n}}^{\psi_{10}(\alpha)}$ is 2-$s$-positive.
\end{lem}

\begin{proof}
Given $\alpha\in N_{10}$, by definition, we have $\alpha(v_0)=\alpha(v_1)=1$, $k_1=2$ and $\alpha(v_2)=\alpha(v_3)=0$. Clearly the unique non-2-$s$-positive connected component $T$ in $T_{3,m,n}^\alpha$ is the component containing $v_0^\alpha$ (see Figure \ref{fig-N10} for an example). It is evident that $T$ is a spider with torso $v_1^\alpha$,  and either $T\cong S(1^3)$ or $T\cong S(1^3,2)$. Thus $T$ has  bipartition either $(3,1)$ or $(4,2)$. From Corollary \ref{pro-claw-iso}, we see that $G_2^\alpha,G_3^\alpha\in Z$.
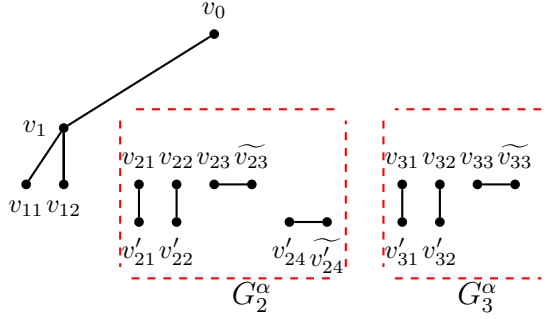
\begin{figure}[h]
    \centering
\begin{tikzpicture}
[thick, every label/.style={font=\footnotesize}, place/.style={thick,fill=black!100,circle,inner sep=0pt,minimum size=1mm,draw=black!100}]
\node [place,label=above:{$v_0$}] (v4) at (-1.5,2.5) {};
\node[place,label=above:{$v_{21}$}] (v6) at (-2.5,0.5) {};
\node [place,label=below:{$v_{21}'$}] (v7) at (-2.5,0) {};
\node[place,label=above:{$v_{23}$}] (v14) at (-1.5,0.5) {};
\node[place,label=above:{$\widetilde{v_{23}}$}] (v15) at (-1,0.5) {};
\node[place,label=below:{$v_{24}'$}] (v16) at (-0.5,0) {};
\node[place,label=below:{$\widetilde{v_{24}'}$}] (v17) at (0,0) {};
\node [place,label=below:{$v_{12}$}] (v8) at (-3.5,0.5) {};
\node[place,label=below:{$v_{11}$}] (v2) at (-4,0.5) {};
\node[place,label=left:{$v_{1}$}] (v3) at (-3.5,1.25) {};
\node [place,label=above:{$v_{31}$}] (v19) at (1,0.5) {};
\node [place,label=below:{$v_{31}'$}] (v20) at (1,0) {};
\node[place,label=above:{$v_{32}$}] (v21) at (1.5,0.5) {};
\node[place,label=below:{$v_{32}'$}] (v22) at (1.5,0) {};
\node [place,label=above:{$v_{33}$}] (v13) at (2,0.5) {};
\node [place,label=above:{$\widetilde{v_{33}}$}] (v12) at (2.5,0.5) {};
\draw (v4) -- (v3) -- (v8) -- (v3) -- (v2);
\draw  (v6) edge (v7);
\draw  (v14) edge (v15);
\draw  (v16) edge (v17);
\node [place,label=above:{$v_{22}$}] (v1) at (-2,0.5) {};
\node [place,label=below:{$v_{22}'$}] (v5) at (-2,0) {};
\draw  (v1) edge (v5);
\draw  (v19) edge (v20);
\draw  (v21) edge (v22);
\draw  (v13) edge (v12);
\node (v9) at (-2.75,1.5) {};
\node (v10) at (-2.75,-0.75) {};
\node (v11) at (0.25,-0.75) {};
\node (v18) at (0.25,1.5) {};
\draw [red] [dashed] (v9) -- (v10) -- (v11) -- (v18) -- (v9);
\node at (-1,-1) {$G_2^\alpha$};
\node (v23) at (0.75,1.5) {};
\node (v24) at (0.75,-0.75) {};
\node (v25) at (3,-0.75) {};
\node (v26) at (3,1.5) {};
\draw [red] [dashed] (v23) -- (v24) -- (v25) -- (v26) -- (v23);
\node at (2,-1) {$G_3^\alpha$};
\end{tikzpicture}
    \caption{An example of $T_{3,4,3}^\alpha(\alpha\in N_{10})$.}
    \label{fig-N10}
\end{figure}

Define $\psi_{10}(\alpha)$ as follows:
\begin{equation}\label{equ-def-psi10}
\psi_{10}(\alpha)(v)=\begin{cases}
\phi_{1}(\alpha\mid_{G_1})(v),&\text{if }v\in V(G_1);\\
\alpha(v),& \text{otherwise.}
\end{cases}
\end{equation}
Set $\beta=\psi_{10}(\alpha)$. From the construction of $\psi_{10}$ and using the same argument as in Lemma \ref{lem-psi4}, we have $G_1^\beta\in X_{1,3}^2$, $G_2^\beta=G_2^\alpha\in Z$ and $G_3^\beta=G_3^\alpha\in Z$.   This implies $\beta\in M_{10}$.

We next show that $\psi_{10}$ is an injection. Let $I_{10}=\{\psi_{10}(\alpha)\colon \alpha\in N_{10}\}\subseteq M_{10}$. For any $\beta\in I_{10}$, we construct a map $\varphi_{10}\colon I_{10}\rightarrow N_{10}$ as follows:
\begin{equation}
\varphi_{10}(\beta)(v)=\begin{cases}
\phi^{-1}(\beta\mid_{G_1})(v),&\text{if }v\in V(G_1);\\
\beta(v),& \text{otherwise.}
\end{cases}
\end{equation}
It is easy to check that $\varphi_{10}(\psi_{10}(\alpha))=\alpha$ for any $\alpha\in N_{10}$. This yields $\psi_{10}$ is an injection.

We next calculate $X_{T_{3,m,n}}^\alpha+X_{T_{3,m,n}}^\beta$. Let $G_{1,0}$ denote the subgraph of $T_{3,m,n}$ induced by $V(G_1)\cup \{v_0\}$ and let $\overline{G_{1,0}}$ denote the subgraph of $T_{3,m,n}$ induced by $V(T_{3,m,n})\setminus V(G_{1,0})$. Clearly $G_{1,0}\cong S(1,2^3)$. Moreover, $\alpha\mid_{G_1}\in \mathcal{A}_2$. From Proposition \ref{pro-s12n}, we deduce that
\begin{equation}\label{eq-x-g-10-alpha-beta}
    X_{G_{1,0}}^\alpha+X_{G_{1,0}}^\beta\ge_{2s} 0.
\end{equation}

From $\alpha(v_2)=\alpha(v_3)=0$, we see that there are no edges between  $G_{1,0}^\alpha$ and $\overline{G_{1,0}}^\alpha$. Similarly, $\beta(v_2)=\beta(v_3)=0$ implies that there are no edges between $G_{1,0}^\beta$ and $\overline{G_{1,0}}^\beta$.   Using Proposition \ref{prop-42}, we find that
\begin{equation}\label{eq-combin-x-alpha-x-beta}
X_{T_{3,m,n}}^\alpha=X_{G_{1,0}}^\alpha    X_{\overline{G_{1,0}}}^\alpha,\quad\text{and}\quad X_{T_{3,m,n}}^\beta=X_{G_{1,0}}^\beta    X_{\overline{G_{1,0}}}^\beta
\end{equation}
By the construction of $\psi_{10}$, we see that $X_{\overline{G_{1,0}}}^\beta=X_{\overline{G_{1,0}}}^\alpha$. Combining \eqref{eq-x-g-10-alpha-beta} and \eqref{eq-combin-x-alpha-x-beta}, we have
\[X_{T_{3,m,n}}^\alpha+X_{T_{3,m,n}}^\beta=(X_{G_{1,0}}^\alpha+X_{G_{1,0}}^\beta)X_{\overline{G_{1,0}}}^\alpha\ge_{2s} 0.\]

\end{proof}

\begin{lem}\label{lem-psi11}
There is an injection $\psi_{11}$ from $N_{11}$ to $M_{11}$. Moreover, $X_{T_{3,m,n}}^\alpha+X_{T_{3,m,n}}^{\psi_{11}(\alpha)}$ is 2-$s$-positive.
\end{lem}

\begin{proof}
Given $\alpha\in N_{11}$, by definition we see that $\alpha(v_0)=\alpha(v_1)=1$, $\alpha(v_2)=\alpha(v_3)=0$ and $k_1=3$. From $k_1=3$ it is clear that $\alpha(v_{11})=\alpha(v_{12})=\alpha(v_{13})=1$ and $\alpha(v'_{11})=\alpha(v'_{12})=\alpha(v'_{13})=0$. Define
\begin{equation}\label{equ-def-psi11}
\psi_{11}(\alpha)(v)=\begin{cases}
1,&\text{if }v=v'_{11};\\
0,&\text{if }v=v_{13};\\
\alpha(v),& \text{otherwise.}
\end{cases}
\end{equation}
Set $\beta=\psi_{11}(\alpha)$. It is obvious $\beta\in M_{11}$.
We next show that $\psi_{11}$ is an injection. Let $I_{11}=\{\psi_{11}(\alpha)\colon \alpha\in N_{11}\}\subseteq M_{11}$. For any $\beta\in I_{11}$, we construct a map $\varphi_{11}\colon I_{11}\rightarrow N_{11}$ as follows:
\begin{equation}
\varphi_{11}(\beta)(v)=\begin{cases}
1,&\text{if }v=v_{13};\\
0,&\text{if }v=v'_{11};\\
\beta(v),& \text{otherwise.}
\end{cases}
\end{equation}
It is easy to check that $\varphi_{11}(\psi_{11}(\alpha))=\alpha$ for any $\alpha\in N_{11}$. This yields $\psi_{11}$ is an injection.

Let $G_{1,0}$ be the subgraph of $T_{3,m,n}$ induced by $V(G_1)\cup\{v_0\}$, and let $\overline{G_{1,0}}$ denote the subgraph of $T_{3,m,n}$ induced by $V(T_{3,m,n})\setminus V({G_{1,0}})$.
From $\alpha(v_2)=\alpha(v_3)=0$, we see that there are no edges between $G_{1,0}^\alpha$ and $\overline{G_{1,0}}^\alpha$. Similarly, $\beta(v_2)=\beta(v_3)=0$ implies that there are no edges between $G_{1,0}^\beta$ and $\overline{G_{1,0}}^\beta$.   Using Proposition \ref{prop-42}, we find that
\begin{equation}\label{eq-combin-x-alpha-x-beta-11}
X_{T_{3,m,n}}^\alpha=X_{G_{1,0}}^\alpha    X_{\overline{G_{1,0}}}^\alpha,\quad\text{and}\quad X_{T_{3,m,n}}^\beta=X_{G_{1,0}}^\beta    X_{\overline{G_{1,0}}}^\beta
\end{equation}

Evidently, the bipartition of $G_{1,0}^\alpha$ is $(4,1)$ and the bipartition of $G_{1,0}^\beta$ is $(3,2)$. From Proposition \ref{prop-2s-coe}, we have
\begin{equation}\label{eq-xt3mn-alpha-xhatt-11}
    X_{G_{1,0}}^\alpha =_{2s}(s_{(4,1)}-s_{(3,2)}) \quad \text{and} \quad   X_{G_{1,0}}^\beta =_{2s}s_{(3,2)}.
\end{equation}

Thus substituting \eqref{eq-xt3mn-alpha-xhatt-11} into \eqref{eq-combin-x-alpha-x-beta-11}, and note that $\overline{G_{1,0}}^\alpha=\overline{G_{1,0}}^\beta$, we have
\begin{equation}
X_{T_{3,m,n}}^\alpha+X_{T_{3,m,n}}^\beta=_{2s}s_{(4,1)}X_{\overline{G_{1,0}}}^\alpha\ge_{2s} 0.
\end{equation}
\end{proof}

\begin{lem}\label{lem-psi12}
There is an injection $\psi_{12}$ from $N_{12}$ to $M_{12}$. Moreover, $X_{{T_{3,m,n}}}^\alpha+X_{T_{3,m,n}}^{\psi_{12}(\alpha)}$ is 2-$s$-positive.
\end{lem}

\begin{proof}
The proof of this lemma is analogue to the proof of Lemma \ref{lem-psi10}. We give the explicit map of $\psi_{12}$ and omit the other detail.
\begin{equation}
\psi_{12}(\alpha)(v)=\begin{cases}
\phi_{1}(\alpha\mid_{G_2})(v),&\text{if }v\in V(G_2);\\
\alpha(v),&\text{otherwise.}
\end{cases}
\end{equation}
\end{proof}

\begin{lem}\label{lem-psi13}
There is an injection $\psi_{13}$ from $N_{13}$ to $M_{13}$. Moreover, $X_{T_{3,m,n}}^\alpha+X_{T_{3,m,n}}^{\psi_{13}(\alpha)}$ is 2-$s$-positive.
\end{lem}

\begin{proof}
Given $\alpha\in N_{13}$, by definition we see that $\alpha(v_0)=\alpha(v_2)=1$, $\alpha(v_1)=\alpha(v_3)=0$ and $k_2\ge 3$. From \eqref{equ-def-x0n} and Corollary \ref{cor-xsn-distinct}, we see that there exists a unique $j\ge 0$ satisfies that $G_3^\alpha\in X_{j,n}$.   Define
\begin{equation}\label{equ-def-psi13}
\psi_{13}(\alpha)(v)=\begin{cases}
\phi_2(\alpha\mid_{G_2})(v),&\text{if }v\in V(G_2), \text{ and $j$ odd};\\
\phi_3(\alpha\mid_{G_2})(v),&\text{if }v\in V(G_2),  \text{ and $j$ even};\\
\alpha(v),& \text{otherwise.}
\end{cases}
\end{equation}
Set $\beta=\psi_{13}(\alpha)$. From the construction of $\psi_{13}$, we see that $\beta(v_0)=1,$ $\beta(v_1)=\beta(v_3)=0$ and  $G_3^\beta=G_3^\alpha\in X_{j,n}$. Moreover,   $G_2^\beta\in X_{2,m}^{k_2}$ when $j$ is odd and $G_2^\beta\in X_{3,m}^{k_2}$ when $j$ is even.    This implies $\beta\in M_{13}$.

We next show that $\psi_{13}$ is an injection. Let $I_{13}=\{\psi_{13}(\alpha)\colon \alpha\in N_{13}\}\subseteq M_{13}$. For any $\beta\in I_{13}$, we construct a map $\varphi_{13}\colon I_{13}\rightarrow N_{13}$ as follows:
\begin{equation}
\varphi_{13}(\beta)(v)=\begin{cases}
\phi^{-1}(\beta\mid_{G_2})(v),&\text{if }v\in V(G_2);\\
\beta(v),& \text{otherwise.}
\end{cases}
\end{equation}
It is easy to check that $\varphi_{13}(\psi_{13}(\alpha))=\alpha$ for any $\alpha\in N_{13}$. This yields $\psi_{13}$ is an injection.

Let $G_{2,0}$ denote the subgraph of $T_{3,m,n}$ induced by $V(G_2)\cup \{v_0\}$, and let $\overline{G_{2,0}}$ denote the subgraph of $T_{3,m,n}$ induced by $V(T_{3,m,n})\setminus V({G_{2,0}})$. It is clear that $G_{2,0}\cong S(1,2^m)$, and $G_{2}^\alpha\in \mathcal{A}_{k_2}$. Using Proposition \ref{pro-s12n}, we see that
\begin{equation}\label{eq-g20-alpha-beta-13}
    X_{G_{2,0}}^\alpha+X_{G_{2,0}}^\beta\ge_{2s} 0.
\end{equation}

Moreover, from $\alpha(v_1)=\alpha(v_3)=0$ we see that $X_{T_{3,m,n}}^\alpha=X_{G_{2,0}}^\alpha X_{\overline{G_{2,0}}}^\alpha.$ Similarly, $\beta(v_1)=\beta(v_3)=0$ we see that $X_{T_{3,m,n}}^\beta=X_{G_{2,0}}^\beta X_{\overline{G_{2,0}}}^\beta.$ Furthermore, by the definition of $\psi_{13}$, we find that $X_{\overline{G_{2,0}}}^\beta=X_{\overline{G_{2,0}}}^\alpha$. From the above analysis and \eqref{eq-g20-alpha-beta-13}, we have
\begin{equation}
    X_{T_{3,m,n}}^\alpha+X_{T_{3,m,n}}^\beta=(X_{G_{2,0}}^\alpha+X_{G_{2,0}}^\beta)X_{\overline{G_{2,0}}}^\alpha\ge_{2s} 0.
\end{equation}
\end{proof}

\begin{lem}\label{lem-psi14}
There is an injection $\psi_{14}$ from $N_{14}$ to $M_{14}$. Moreover, $X_{{T_{3,m,n}}}^\alpha+X_{T_{3,m,n}}^{\psi_{14}(\alpha)}$ is 2-$s$-positive.
\end{lem}

\begin{proof}
The proof of this lemma is analogue to the proof of Lemma \ref{lem-psi10}. We give the explicit map of $\psi_{14}$ and omit the other detail.
\begin{equation}
\psi_{14}(\alpha)(v)=\begin{cases}
\phi_{1}(\alpha\mid_{G_3})(v),&\text{if }v\in V(G_3);\\
\alpha(v),&\text{otherwise.}
\end{cases}
\end{equation}
\end{proof}

\begin{lem}\label{lem-psi15}
There is an injection $\psi_{15}$ from $N_{15}$ to $M_{15}$. Moreover, $X_{T_{3,m,n}}^\alpha+X_{T_{3,m,n}}^{\psi_{15}(\alpha)}$ is 2-$s$-positive.
\end{lem}

\begin{proof}
The proof of this lemma is analogue to the proof of Lemma \ref{lem-psi13}. From \eqref{equ-def-x0n} and Corollary \ref{cor-xsn-distinct}, we see that there exists a unique $i\ge 0$ satisfies that $G_2^\alpha\in X_{i,m}$.  We give the explicit map of $\psi_{15}$ and omit the other detail.
\begin{equation}
\psi_{15}(\alpha)(v)=\begin{cases}
\phi_3(\alpha\mid_{G_3})(v),&\text{if }v\in V(G_3), i \text{ odd};\\
\phi_2(\alpha\mid_{G_3})(v),&\text{if }v\in V(G_3), i \text{ even};\\
\alpha(v),& \text{otherwise.}
\end{cases}
\end{equation}
\end{proof}

\begin{lem}\label{lem-psi16}
There is an injection $\psi_{16}$ from $N_{16}$ to $M_{16}$. Moreover, $X_{T_{3,m,n}}^\alpha+X_{T_{3,m,n}}^{\psi_{16}(\alpha)}$ is 2-$s$-positive.
\end{lem}

\begin{proof}
Given $\alpha\in N_{16}$, by definition we see that $\alpha(v_0)=\alpha(v_2)=\alpha(v_3)=1$ and $\alpha(v_1)=0$. Let $T$ denote the connected component in $T_{3,m,n}^\alpha$ contains $v_0^\alpha$. Assume
\[\#\{j\colon \alpha(v_{2j})=\alpha(v_{2j}')=1\}=r_2,\quad\text{and}\quad \#\{j\colon \alpha(v_{3j})=\alpha(v_{3j}')=1\}=r_3.\]
We see that $T$ has bipartition $(r_2+r_3+k_2+k_3+1,r_2+r_3+2)$. Since $T$ is the only possible non-2-$s$-positive connected component in $T_{3,m,n}^\alpha$, from Corollary \ref{cor-2s-con-bipar}, we have $k_2+k_3\ge 3$. There are two cases.

Case 1. If $k_2\ge 2$, then define
\begin{equation}\label{equ-def-psi16-1}
\psi_{16}(\alpha)(v)=\begin{cases}
\phi_{\{1,2\}}(\alpha\mid_{G_2})(v),&\text{if }v \in V(G_2);\\
\phi_\emptyset(\alpha\mid_{G_3})(v),&\text{if }v\in V(G_3);\\
\alpha(v),& \text{otherwise.}
\end{cases}
\end{equation}

Case 2. If $k_2\le 1$, then from $k_2+k_3\ge 3$ we deduce that $k_3\ge 2$. Define
\begin{equation}\label{equ-def-psi16-2}
\psi_{16}(\alpha)(v)=\begin{cases}
\phi_{\{1,2\}}(\alpha\mid_{G_3})(v),&\text{if }v \in V(G_3);\\
\phi_\emptyset(\alpha\mid_{G_2})(v),&\text{if }v\in V(G_2);\\
\alpha(v),& \text{otherwise.}
\end{cases}
\end{equation}

Set $\beta=\psi_{16}(\alpha)$.
From the construction of $\psi_{16}$, it is easy to see that $\beta(v_0)=1$ and $\beta(v_1)=\beta(v_2)=\beta(v_3)=0$. Moreover, either $G_2^\beta\in X_{\{1,2\},m}$ and $G_3^\beta\in X_{\emptyset,n}$, or $G_2^\beta\in X_{\emptyset,m}$ and $G_3^\beta\in X_{\{1,2\},n}$.
 This implies $\beta\in M_{16}$.

Let $I_{16}=\{\psi_{16}(\alpha)\colon \alpha\in N_{16}\}\subseteq M_{16}$. For any $\beta\in I_{16}$, define $\varphi_{16}\colon I_{16}\rightarrow N_{16}$ as follows:
\begin{equation}
\varphi_{16}(\beta)(v)=\begin{cases}
\phi^{-1}(\beta\mid_{G_2})(v),&\text{if }v\in V(G_2);\\
\phi^{-1}(\beta\mid_{G_3})(v),&\text{if }v\in V(G_3);\\
\beta(v),& \text{otherwise.}
\end{cases}
\end{equation}
Clearly $\varphi_{16}(\psi_{16}(\alpha))=\alpha$ for any $\alpha\in N_{16}$. This yields that $\psi_{16}$ is an injection.

Let $\hat{T}$ and $\overline{T}$ be as in Definition \ref{defi-t-hat}. By the definition of $\psi_{16}$, we have $\alpha(v)=\beta(v)$ for any $v\in\overline{T}$. From Proposition \ref{pro-X-hat-overline-T}, we find that
\begin{equation}\label{eq-psi16-1}
    X_{T_{3,m,n}}^\alpha+X_{T_{3,m,n}}^\beta=X_{\overline{T}}^\alpha(X_{\hat{T}}^\alpha+X_{\hat{T}}^\beta).
\end{equation}
Since $T$ has bipartition $(r_2+r_3+k_2+k_3+1,r_2+r_3+2)$, we see that
\begin{equation}\label{eq-psi16-2}
    X_{\hat{T}}^\alpha=X_T=_{2s} s_{(r_2+r_3+k_2+k_3+1,r_2+r_3+2)}-s_{(r_2+r_3+k_2+k_3,r_2+r_3+3)}.
\end{equation}
Moreover,  from the construction of $\psi_{16}$,   in either case $\hat{T}^\beta$ is the disjoint union of $r_2 + r_3 + 2$ copies of $K_2$ and $k_2 + k_3 - 1$ isolated vertices. Therefore,
\begin{equation}\label{eq-psi16-3}
    X_{\hat{T}}^\beta=\frac{1}{2^2}(2s_{(1,1)})^{r_2+r_3+2}s_1^{k_2+k_3-1}\ge_{2s} s_{(r_2+r_3+k_2+k_3,r_2+r_3+3)}.
\end{equation}
Substituting \eqref{eq-psi16-2} and \eqref{eq-psi16-3} into \eqref{eq-psi16-1}, we get $X_{T_{3,m,n}}^\alpha+X_{T_{3,m,n}}^\beta\ge_{2s} 0$.
\end{proof}

\begin{lem}\label{lem-psi17}
There is an injection $\psi_{17}$ from $N_{17}$ to $M_{17}$. Moreover, $X_{T_{3,m,n}}^\alpha+X_{T_{3,m,n}}^{\psi_{17}(\alpha)}$ is 2-$s$-positive.
\end{lem}

\begin{proof}
Given $\alpha\in N_{17}$, by definition we see that $\alpha(v_0)=\alpha(v_1)=\alpha(v_3)=1$, $\alpha(v_2)=0$, $G_2^\alpha\not\in X_{1,m}$ and $k_3\ge 1$. Define
\begin{equation}\label{equ-def-psi17}
\psi_{17}(\alpha)(v)=\begin{cases}
2,&\text{if }v=v_0;\\
\phi_{\emptyset}(\alpha\mid_{G_1})(v),&\text{if }v \in V(G_1);\\
\phi_1(\alpha\mid_{G_3})(v),&\text{if }v\in V(G_3);\\
\alpha(v),& \text{otherwise.}
\end{cases}
\end{equation}
Set $\beta=\psi_{17}(\alpha)$. From the construction of $\psi_{17}$, it is easy to see that $\beta(v_0)=2$, $\beta(v_1)=\beta(v_2)=\beta(v_3)=0$. Moreover  $G_1^\beta\in X_{\emptyset,3}$, $G_3^\beta\in X_{1,n}$ and $G_2^\beta=G_2^\alpha\not\in X_{1,m}$.    This implies $\beta\in M_{17}$.

Let $I_{17}=\{\psi_{17}(\alpha)\colon \alpha\in N_{17}\}\subseteq M_{17}$. For any $\beta\in I_{17}$, define $\varphi_{17}\colon I_{17}\rightarrow N_{17}$ as follows:
\begin{equation}
\varphi_{17}(\beta)(v)=\begin{cases}
1,&\text{if }v=v_0;\\
\phi^{-1}(\beta\mid_{G_1})(v),&\text{if }v\in V(G_1);\\
\phi^{-1}(\beta\mid_{G_3})(v),&\text{if }v\in V(G_3);\\
\beta(v),& \text{otherwise.}
\end{cases}
\end{equation}
Clearly $\varphi_{17}(\psi_{17}(\alpha))=\alpha$ for any $\alpha\in N_{17}$. This yields that $\psi_{17}$ is an injection.

Similar as in Lemma \ref{lem-psi16}, let $T$ denote the connected component in $T_{3,m,n}^\alpha$ contains $v_0^\alpha$. Moreover, let $\hat{T}$ and $\overline{T}$ be as in Definition \ref{defi-t-hat}. Assume
\[\#\{j\colon \alpha(v_{1j})=\alpha(v_{1j}')=1\}=r_1,\quad\text{and}\quad \#\{j\colon \alpha(v_{3j})=\alpha(v_{3j}')=1\}=r_3.\]
We see that $T$ has bipartition $(r_1+r_3+k_1+k_3+1,r_1+r_3+2)$. Since $T$ is the only possible non-2-$s$-positive connected component in $T_{3,m,n}^\alpha$, from Corollary \ref{cor-2s-con-bipar}, we have $k_1+k_3\ge 3$. Thus
\begin{equation}\label{eq-psi17-1}
    X_T=X_{\hat{T}}^\alpha=_{2s}s_{(r_1+r_3+k_1+k_3+1,r_1+r_3+2)}-s_{(r_1+r_3+k_1+k_3,r_1+r_3+3)}.
\end{equation}

From the construction of  $\psi_{17}$, we have $\alpha(v)=\beta(v)$ for any $v\in V(\overline{T})$. Using Proposition \ref{pro-X-hat-overline-T}, we derive that
\begin{equation}\label{eq-psi17-2}
    X_{T_{3,m,n}}^\alpha+X_{T_{3,m,n}}^\beta=(X_{\hat{T}}^\alpha+X_{\hat{T}}^\beta)X_{\overline{T}}^\alpha.
\end{equation}
From the construction of $\psi_{17}$, we derive that $\hat{T}^\beta$ is the disjoint union of $r_1 + r_3 + 2$ copies of $K_2$ and $k_1 + k_3 - 1$ isolated vertices. This yields
\begin{equation}\label{eq-psi17-3}
    X_{\hat{T}}^\beta=\frac{1}{2^2}(2s_{(1,1)})^{r_1+r_3+2} s_1^{k_1+k_3-1}\ge_{2s} s_{(r_1+r_3+k_1+k_3,r_1+r_3+3)}.
\end{equation}
Substituting \eqref{eq-psi17-1} and \eqref{eq-psi17-3} into \eqref{eq-psi17-2}, we derive that $X_{T_{3,m,n}}^\alpha+X_{T_{3,m,n}}^\beta\ge_{2s} 0$.
\end{proof}

\begin{lem}\label{lem-psi18}
There is an injection $\psi_{18}$ from $N_{18}$ to $M_{18}$. Moreover, $X_{T_{3,m,n}}^\alpha+X_{T_{3,m,n}}^{\psi_{18}(\alpha)}$ is 2-$s$-positive.
\end{lem}

\begin{proof}
Given $\alpha\in N_{18}$, by definition we see that $\alpha(v_0)=\alpha(v_1)=\alpha(v_3)=1$, $\alpha(v_2)=0$, $G_2^\alpha\not\in X_{1,m}$ and $k_3=0$. Let $T$ denote the connected component in $T_{3,m,n}^\alpha$ contains $v_0^\alpha$.  Assume
\[\#\{j\colon \alpha(v_{1j})=\alpha(v_{1j}')=1\}=r_1,\quad\text{and}\quad \#\{j\colon \alpha(v_{3j})=\alpha(v_{3j}')=1\}=r_3.\]
From $k_3=0$, we see that $T$ has bipartition $(r_1+r_3+k_1+1,r_1+r_3+2)$. Since $T^\alpha$ is the only possible non-2-$s$-positive connected component in $T_{3,m,n}^\alpha$, from Corollary \ref{cor-2s-con-bipar}, we have $ k_1\ge 3$, which leads to $3=r_1+k_1\ge k_1\ge 3$. Hence $r_1=0$ and $k_1=3$. This yields $\alpha(v_{11})=\alpha(v_{12})=\alpha(v_{13})=1$ and $\alpha(v'_{11})=\alpha(v'_{12})=\alpha(v'_{13})=0$. Define
\begin{equation}\label{equ-def-psi18}
\psi_{18}(\alpha)(v)=\begin{cases}
1,&\text{if }v\in\{v'_{11},v'_{12}\};\\
0,&\text{if }v\in\{v_3,v_{13}\};\\
\alpha(v),& \text{otherwise.}
\end{cases}
\end{equation}
Set $\beta=\psi_{18}(\alpha)$, it is routine to verify $\beta\in M_{18}$.

Let $I_{18}=\{\psi_{18}(\alpha)\colon \alpha\in N_{18}\}\subseteq M_{18}$. For any $\beta\in I_{18}$, define $\varphi_{18}\colon I_{18}\rightarrow N_{18}$ as follows:
\begin{equation}
\varphi_{18}(\beta)(v)=\begin{cases}
0,&\text{if }v\in\{v'_{11},v'_{12}\};\\
1,&\text{if }v\in\{v_3,v_{13}\};\\
\beta(v),& \text{otherwise.}
\end{cases}
\end{equation}
Clearly $\varphi_{18}(\psi_{18}(\alpha))=\alpha$ for any $\alpha\in N_{18}$. This yields that $\psi_{18}$ is an injection.

Let $\hat{T}$ be as defined in Definition \ref{defi-t-hat}, let $T_1$ denote the subgraph of $T_{3,m,n}$ induced by  $V(\hat{T})\cup\{v'_{11},v'_{12}\}$, and let $\overline{T_1}$ denote the subgraph of $T_{3,m,n}$ induced by $V(T_{3,m,n})\setminus V(T_1)$. From $\alpha(v'_{11})=\alpha(v'_{12})=0$, we get $T=T_1^\alpha$, thus $T_1^\alpha$ is also a connected component in $T_{3,m,n}^\alpha$. Moreover, $\alpha(v)=\beta(v)$ for any $v\in\overline{T}_1$. Using Proposition \ref{pro-X-hat-overline-T}, we derive that
\begin{equation}\label{eq-psi-18-1}
    X_{T_{3,m,n}}^\alpha+X_{T_{3,m,n}}^\beta=(X_{T_1}^\alpha+X_{T_1}^\beta) X_{\overline{T}_1}^\alpha.
\end{equation}
Since $T$ has a bipartition $(r_3+4,r_3+2)$ and $T=T_1^\alpha$, we derive that
\begin{equation}\label{eq-psi-18-2}
    X_{T_1}^\alpha=X_T=s_{(r_3+4,r_3+2)}-s_{(r_3+3,r_3+3)}.
\end{equation}
Moreover, from the construction of $\psi_{18}$, see Figure \ref{fig-t1} as an example, we have $T_1^\beta$ consists of one spider $S(1,2^2)$ and $r_3$ disjoint $K_2$.
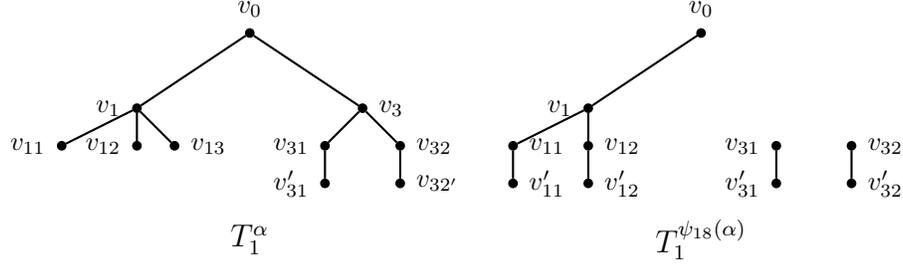
\begin{figure}[h]
    \centering
    \begin{tikzpicture}
[thick, every label/.style={font=\footnotesize}, place/.style={thick,fill=black!100,circle,inner sep=0pt,minimum size=1mm,draw=black!100}]
\node [place,label=above:{$v_0$}] (v4) at (-8,3) {};
\node[place,label=right:{$v_{13}$}] (v10) at (-9,1.5) {};
\node [place,label=left:{$v_{12}$}] (v8) at (-9.5,1.5) {};
\node[place,label=left:{$v_{11}$}] (v2) at (-10.5,1.5) {};
\node[place,label=left:{$v_{1}$}] (v3) at (-9.5,2) {};
\node[place,label=left:{$v_{31}$}] (v19) at (-7,1.5) {};
\node[place,label=left:{$v_{31}'$}] (v20) at (-7,1) {};
\node[place,label=right:{$v_{32}$}] (v21) at (-6,1.5) {};
\node[place,label=right:{$v_{32'}$}] (v22) at (-6,1) {};
\node[place,label=right:{$v_{3}$}] (v18) at (-6.5,2) {};
\draw (v2) -- (v3) -- (v4) -- (v18) -- (v19) -- (v20);
\draw (v10) -- (v3) -- (v8);
\draw (v22) -- (v21) -- (v18);
 \node [place,label=above:{$v_0$}]  at (-2,3) {};
\node [place,label=right:{$v_{12}$}] (v11) at (-3.5,1.5) {};
\node [place,label=right:{$v_{12}'$}] (v12) at (-3.5,1) {};
\node[place,label=right:{$v_{11}$}] (v5) at (-4.5,1.5) {};
\node[place,label=right:{$v_{11}'$}] (v6) at (-4.5,1) {};
\node[place,label=left:{$v_{1}$}] (v7) at (-3.5,2) {};
\node[place,label=left:{$v_{31}$}] (v14) at (-1,1.5) {};
\node[place,label=left:{$v_{31}'$}] (v13) at (-1,1) {};
\node[place,label=right:{$v_{32}$}] (v16) at (0,1.5) {};
\node[place,label=right:{$v_{32}'$}] (v15) at (0,1) {};
\draw (-2,3) -- (v7) -- (v5) -- (v6);
\draw (v7) -- (v11) -- (v12);
\draw  (v13) edge (v14);
\draw  (v15) edge (v16);
\node at (-8,0.25) {$T_1^\alpha$};
\node at (-2,0.25) {$T_1^{\psi_{18}(\alpha)}$};
\end{tikzpicture}
    \caption{$T_1^\alpha$ and $T_1^{\psi_{18}(\alpha)}$($\alpha\in N_{18}$).}
    \label{fig-t1}
\end{figure}

Thus
\begin{equation}\label{eq-psi-18-3}
    X_{T_1}^\beta=_{2s}s_{(3,3)} (2s_{(1,1)})^{r_3}=_{2s}2^{r_3}s_{(r_3+3,r_3+3)}.
\end{equation}
Substituting \eqref{eq-psi-18-2} and \eqref{eq-psi-18-3} into \eqref{eq-psi-18-1}, we derive that
\[X_{T_{3,m,n}}^\alpha+X_{T_{3,m,n}}^\beta\ge_{2s} 0.\]
\end{proof}

\begin{lem}\label{lem-psi19}
There is an injection $\psi_{19}$ from $N_{19}$ to $M_{19}$. Moreover, $X_{T_{3,m,n}}^\alpha+X_{T_{3,m,n}}^{\psi_{19}(\alpha)}$ is 2-$s$-positive.
\end{lem}

\begin{proof}
The proof of this lemma is analogue to the proof of Lemma \ref{lem-psi17}. We give the explicit map of $\psi_{19}$ and omit the other detail.
\begin{equation}
\psi_{19}(\alpha)(v)=\begin{cases}
2,& \text{if }v=v_0;\\
\phi_2(\alpha\mid_{G_1})(v),&\text{if }v\in V(G_1);\\
\phi_\emptyset(\alpha\mid_{G_3})(v),&\text{if }v\in V(G_3);\\
\alpha(v),& \text{otherwise.}
\end{cases}
\end{equation}
\end{proof}

\begin{lem}\label{lem-psi20}
There is an injection $\psi_{20}$ from $N_{20}$ to $M_{20}$. Moreover, $X_{T_{3,m,n}}^\alpha+X_{T_{3,m,n}}^{\psi_{20}(\alpha)}$ is 2-$s$-positive.
\end{lem}

\begin{proof}
The proof of this lemma is analogue to the proof of Lemma \ref{lem-psi17}. Note that $T_{3,m,n}^\alpha$ is non-2-$s$-positive implies $k_3\ge 2$. We give the explicit map of $\psi_{20}$ and omit the other detail.
\begin{equation}
\psi_{20}(\alpha)(v)=\begin{cases}
2,& \text{if }v=v_0;\\
\phi_2(\alpha\mid_{G_3})(v),&\text{if }v\in V(G_3);\\
\phi_\emptyset(\alpha\mid_{G_1})(v),&\text{if }v\in V(G_1);\\
\alpha(v),& \text{otherwise.}
\end{cases}
\end{equation}
\end{proof}

\begin{lem}\label{lem-psi21}
There is an injection $\psi_{21}$ from $N_{21}$ to $M_{21}$. Moreover, $X_{T_{3,m,n}}^\alpha+X_{T_{3,m,n}}^{\psi_{21}(\alpha)}$ is 2-$s$-positive.
\end{lem}

\begin{proof}
The proof of this lemma is analogue to the proof of Lemma \ref{lem-psi17}.  We give the explicit map of $\psi_{21}$ and omit the other detail.
\begin{equation}
\psi_{21}(\alpha)(v)=\begin{cases}
2,& \text{if }v=v_0;\\
\phi_2(\alpha\mid_{G_2})(v),&\text{if }v\in V(G_2);\\
\phi_\emptyset(\alpha\mid_{G_1})(v),&\text{if }v\in V(G_1);\\
\alpha(v),& \text{otherwise.}
\end{cases}
\end{equation}
\end{proof}

\begin{lem}\label{lem-psi22}
There is an injection $\psi_{22}$ from $N_{22}$ to $M_{22}$. Moreover, $X_{T_{3,m,n}}^\alpha+X_{T_{3,m,n}}^{\psi_{22}(\alpha)}$ is 2-$s$-positive.
\end{lem}

\begin{proof}
Given $\alpha\in N_{22}$, by definition we have $\alpha(v_0)=\alpha(v_1)=\alpha(v_2)=1$, $\alpha(v_3)=0$, $G_3^\alpha\not\in X_{1,n}$, $k_2\le 1$ and $k_1=2$. Let $T$ denote the connected component in $T_{3,m,n}^\alpha$ contains $v_0^\alpha$.  Assume
\[\#\{j\colon \alpha(v_{1j})=\alpha(v_{1j}')=1\}=r_1,\quad\text{and}\quad \#\{j\colon \alpha(v_{2j})=\alpha(v_{2j}')=1\}=r_2.\]
We see that $T$ has bipartition $(r_1+r_2+k_1+k_2+1,r_1+r_2+2)$. Since $T$ is the only possible non-2-$s$-positive connected component in $T_{3,m,n}^\alpha$, from Corollary \ref{cor-2s-con-bipar}, we have $k_1+k_2\ge 3$. Together with $k_1=2$ and $k_2\le 1$ we arrive at $k_2=1$. From Proposition \ref{pro-claw-iso}, we find that there are no isolated vertices in $T_{3,m,n}^\alpha$. In particular, there are no isolated vertex in $G_3^\alpha$, which implies $G_3^\alpha\not\in X_{2,n}$. Define
\begin{equation}
\psi_{22}(\alpha)(v)=\begin{cases}
2,& \text{if }v=v_0;\\
\phi_1(\alpha\mid_{G_2})(v),&\text{if }v\in V(G_2);\\
\phi_\emptyset(\alpha\mid_{G_1})(v),&\text{if }v\in V(G_1);\\
\alpha(v),& \text{otherwise.}
\end{cases}
\end{equation}

The rest part of the proof is analogue to the proof of Lemma \ref{lem-psi17}. We omit the detailed proof.
\end{proof}

\begin{lem}\label{lem-psi23}
There is an injection $\psi_{23}$ from $N_{23}$ to $M_{23}$. Moreover, $X_{T_{3,m,n}}^\alpha+X_{T_{3,m,n}}^{\psi_{23}(\alpha)}$ is 2-$s$-positive.
\end{lem}

\begin{proof}
The proof of this lemma is analogue to the proof of Lemma \ref{lem-psi18}.  We give the explicit map of $\psi_{23}$ and omit the other detail.
\begin{equation}
\psi_{23}(\alpha)(v)=\begin{cases}
1,&\text{if }v\in\{v'_{11},v'_{12}\};\\
0,&\text{if }v\in\{v_2,v_{12}\};\\
\alpha(v),& \text{otherwise.}
\end{cases}
\end{equation}

As an illustration, consider the following example.
\begin{figure}[h]
    \centering
\begin{tikzpicture}
[thick, every label/.style={font=\footnotesize}, place/.style={thick,fill=black!100,circle,inner sep=0pt,minimum size=1mm,draw=black!100}]
\node [place,label=above:{$v_0$}] (v4) at (-8,3) {};
\node[place,label=right:{$v_{13}$}] (v10) at (-9,1.5) {};
\node [place,label=left:{$v_{12}$}] (v8) at (-9.5,1.5) {};
\node[place,label=left:{$v_{11}$}] (v2) at (-10.5,1.5) {};
\node[place,label=left:{$v_{1}$}] (v3) at (-9.5,2) {};
\node[place,label=left:{$v_{21}$}] (v19) at (-7,1.5) {};
\node[place,label=left:{$v_{21}'$}] (v20) at (-7,1) {};
\node[place,label=right:{$v_{22}$}] (v21) at (-6,1.5) {};
\node[place,label=right:{$v_{22}'$}] (v22) at (-6,1) {};
\node[place,label=right:{$v_{2}$}] (v18) at (-6.5,2) {};
\draw (v2) -- (v3) -- (v4) -- (v18) -- (v19) -- (v20);
\draw (v10) -- (v3) -- (v8);
\draw (v22) -- (v21) -- (v18);
 \node [place,label=above:{$v_0$}]  at (-2,3) {};
\node [place,label=right:{$v_{12}'$}] (v12) at (-3.5,1) {};
\node[place,label=right:{$v_{11}$}] (v5) at (-4.5,1.5) {};
\node[place,label=right:{$v_{11}'$}] (v6) at (-4.5,1) {};
\node[place,label=left:{$v_{1}$}] (v7) at (-3.5,2) {};
\node[place,label=left:{$v_{21}$}] (v14) at (-1,1.5) {};
\node[place,label=left:{$v_{21}'$}] (v13) at (-1,1) {};
\node[place,label=right:{$v_{22}$}] (v16) at (0,1.5) {};
\node[place,label=right:{$v_{22}'$}] (v15) at (0,1) {};
\draw (-2,3) -- (v7) -- (v5) -- (v6);
\draw  (v13) edge (v14);
\draw  (v15) edge (v16);
\node at (-8,0.25) {$T_{3,m,n}^\alpha$};
\node at (-2,0.25) {$T_{3,m,n}^{\psi_{23}(\alpha)}$};
\node [place,label=right:{$v_{13}$}] at (-2.5,1.5) {};
\draw (-2.5,1.5)--(-3.5,2);
\end{tikzpicture}
    \caption{An example for $\psi_{23}$.}
\end{figure}
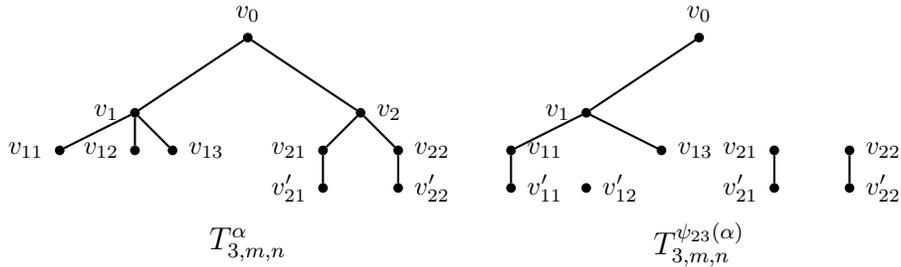
\end{proof}

\begin{lem}\label{lem-psi24}
There is an injection $\psi_{24}$ from $N_{24}$ to $M_{24}$. Moreover, $X_{T_{3,m,n}}^\alpha+X_{T_{3,m,n}}^{\psi_{24}(\alpha)}$ is 2-$s$-positive.
\end{lem}

\begin{proof}
The proof of this lemma is analogue to the proof of Lemma \ref{lem-psi17}.  We give the explicit map of $\psi_{24}$ and omit the other detail.
\begin{equation}
\psi_{24}(\alpha)(v)=\begin{cases}
2,&\text{if }v=v_0;\\
\phi_\emptyset(\alpha\mid_{G_1})(v),&\text{if }v\in V(G_1);\\
\phi_1(\alpha\mid_{G_2})(v),&\text{if }v\in V(G_2);\\
\alpha(v),& \text{otherwise.}
\end{cases}
\end{equation}
\end{proof}

\begin{lem}\label{lem-psi25}
There is an injection $\psi_{25}$ from $N_{25}$ to $M_{25}$. Moreover, $X_{T_{3,m,n}}^\alpha+X_{T_{3,m,n}}^{\psi_{25}(\alpha)}$ is 2-$s$-positive.
\end{lem}

\begin{proof}
Given $\alpha\in N_{25}$, by definition we have $\alpha(v_0)=\alpha(v_1)=\alpha(v_2)=1$, $\alpha(v_3)=0$, $G_3^\alpha\in X_{1,n}$ and $k_2=0$. Let $T$ denote the connected component in $T_{3,m,n}^\alpha$ contains $v_0^\alpha$.  Assume
\[\#\{j\colon \alpha(v_{1j})=\alpha(v_{1j}')=1\}=r_1,\quad\text{and}\quad \#\{j\colon \alpha(v_{2j})=\alpha(v_{2j}')=1\}=r_2.\]
We see that $T$ has bipartition $(r_1+r_2+k_1+1,r_1+r_2+2)$. Since $T$ is the only possible non-2-$s$-positive connected component in $T_{3,m,n}^\alpha$, from Corollary \ref{cor-2s-con-bipar}, we have $ k_1\ge 3$. Combining $3=k_1+r_1\ge k_1\ge 3$, we derive that  $k_1=3$ and $r_1=0$. Hence $\alpha(v_{11})=\alpha(v_{12})=\alpha(v_{13})=1$ and $\alpha(v'_{11})=\alpha(v'_{12})=\alpha(v'_{13})=0$. Define
\begin{equation}
\psi_{25}(\alpha)(v)=\begin{cases}
1,& \text{if }v\in\{v'_{11},v'_{12}\};\\
0,&\text{if }v\in \{v_{11},v_2\};\\
\alpha(v),& \text{otherwise.}
\end{cases}
\end{equation}

We give an example for illustration.
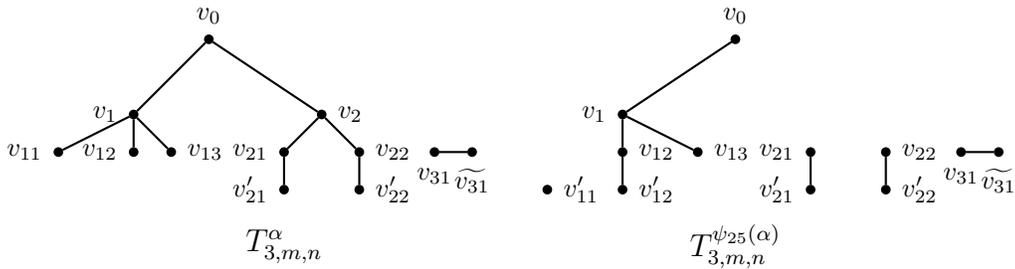
\begin{figure}[h]
    \centering
\begin{tikzpicture}
[thick, every label/.style={font=\footnotesize}, place/.style={thick,fill=black!100,circle,inner sep=0pt,minimum size=1mm,draw=black!100}]
\node [place,label=above:{$v_0$}] (v4) at (-9,3) {};
\node[place,label=right:{$v_{13}$}] (v10) at (-9.5,1.5) {};
\node [place,label=left:{$v_{12}$}] (v8) at (-10,1.5) {};
\node[place,label=left:{$v_{11}$}] (v2) at (-11,1.5) {};
\node[place,label=left:{$v_{1}$}] (v3) at (-10,2) {};
\node[place,label=left:{$v_{21}$}] (v19) at (-8,1.5) {};
\node[place,label=left:{$v_{21}'$}] (v20) at (-8,1) {};
\node[place,label=right:{$v_{22}$}] (v21) at (-7,1.5) {};
\node[place,label=right:{$v_{22}'$}] (v22) at (-7,1) {};
\node[place,label=right:{$v_{2}$}] (v18) at (-7.5,2) {};
\draw (v2) -- (v3) -- (v4) -- (v18) -- (v19) -- (v20);
\draw (v10) -- (v3) -- (v8);
\draw (v22) -- (v21) -- (v18);
 \node [place,label=above:{$v_0$}]  at (-2,3) {};
\node [place,label=right:{$v_{12}'$}] (v12) at (-3.5,1) {};
\node[place,label=right:{$v_{11}'$}] (v6) at (-4.5,1) {};
\node[place,label=left:{$v_{1}$}] (v7) at (-3.5,2) {};
\node[place,label=left:{$v_{21}$}] (v14) at (-1,1.5) {};
\node[place,label=left:{$v_{21}'$}] (v13) at (-1,1) {};
\node[place,label=right:{$v_{22}$}] (v16) at (0,1.5) {};
\node[place,label=right:{$v_{22}'$}] (v15) at (0,1) {};
\draw (-2,3) -- (v7);
\draw  (v13) edge (v14);
\draw  (v15) edge (v16);
\node at (-8,0.25) {$T_{3,m,n}^\alpha$};
\node at (-2,0.25) {$T_{3,m,n}^{\psi_{25}(\alpha)}$};
\node [place,label=right:{$v_{13}$}] (v11) at (-2.5,1.5) {};
\node [place,label=below:{$v_{31}$}] (v1) at (-6,1.5) {};
\node [place,label=below:{$\widetilde{v_{31}}$}] (v9) at (-5.5,1.5) {};
\draw (v1) -- (v9) ;
\node [place,label=right:{$v_{12}$}] (v5) at (-3.5,1.5) {};
\draw (v12) -- (v5) -- (v7) -- (v11);
\node [place,label=below:{$v_{31}$}] (v17) at (1,1.5) {};
\node [place,label=below:{$\widetilde{v_{31}}$}] (v23) at (1.5,1.5) {};
\draw  (v17) edge (v23);
\end{tikzpicture}
    \caption{An example for $\psi_{25}$.}
\end{figure}

The rest part of the proof is analogue to the proof of Lemma \ref{lem-psi18}. We omit the detailed proof.
\end{proof}

\begin{lem}\label{lem-psi26}
There is an injection $\psi_{26}$ from $N_{26}$ to $M_{26}$. Moreover, $X_{T_{3,m,n}}^\alpha+X_{T_{3,m,n}}^{\psi_{26}(\alpha)}$ is 2-$s$-positive.
\end{lem}

\begin{proof}
Given $\alpha\in N_{26}$, by definition we see that $\alpha(v_0)=\alpha(v_1)=\alpha(v_2)=\alpha(v_3)=1$, $k_1+k_2+k_3\ge 4$ and at least one of $k_1=3$, $k_2\ge 2$, $k_3\ge 2$ is true. There are three cases.

Case 1. $k_1=3$. In this case, define
\begin{equation}\label{equ-def-psi26}
\psi_{26}(\alpha)(v)=\begin{cases}
2,&\text{if }v=v_0;\\
\phi_{\{1,2\}}(\alpha\mid_{G_1})(v),&\text{if }v\in V(G_1);\\
\phi_\emptyset(\alpha\mid_{G_2})(v),&\text{if }v\in V(G_2);\\
\phi_\emptyset(\alpha\mid_{G_3})(v),&\text{if }v\in V(G_3).
\end{cases}
\end{equation}
Set $\beta=\psi_{26}(\alpha)$. It is clear that $G_1^\beta\in X^3_{\{1,2\},3}$, $G_2^\beta\in X_{\emptyset, m}$ and $G_3^\beta\in X_{\emptyset,n}$. This implies $\beta\in M_{26}$.

Case 2. $k_1\le 2$ and $k_2\ge 2$. In this case, define
\begin{equation}\label{equ-def-psi26-1}
\psi_{26}(\alpha)(v)=\begin{cases}
2,&\text{if }v=v_0;\\
\phi_\emptyset(\alpha\mid_{G_1})(v),&\text{if }v\in V(G_1);\\
\phi_{\{1,2\}}(\alpha\mid_{G_2})(v),&\text{if }v\in V(G_2);\\
\phi_\emptyset(\alpha\mid_{G_3})(v),&\text{if }v\in V(G_3).
\end{cases}
\end{equation}
Set $\beta=\psi_{26}(\alpha)$. It is clear that $G_2^\beta\in X_{\{1,2\},m}$, $G_1^\beta\in X_{\emptyset, 3}$ and $G_3^\beta\in X_{\emptyset,n}$. This implies $\beta\in M_{26}$.

Case 3. $k_1\le 2$, $k_2\le 1$ and $k_3 \ge 2$. In this case, define
\begin{equation}\label{equ-def-psi26-2}
\psi_{26}(\alpha)(v)=\begin{cases}
2,&\text{if }v=v_0;\\
\phi_\emptyset(\alpha\mid_{G_1})(v),&\text{if }v\in V(G_1);\\
\phi_\emptyset(\alpha\mid_{G_2})(v),&\text{if }v\in V(G_2);\\
\phi_{\{1,2\}}(\alpha\mid_{G_3})(v),&\text{if }v\in V(G_3).
\end{cases}
\end{equation}
Set $\beta=\psi_{26}(\alpha)$. It is clear that $G_3^\beta\in X_{\{1,2\},n}$, $G_1^\beta\in X_{\emptyset, 3}$ and $G_2^\beta\in X_{\emptyset,m}$. This implies $\beta\in M_{26}$.

So in each case, we have $\beta\in M_{26}$.
Let $I_{26}=\{\psi_{26}(\alpha)\colon \alpha\in N_{26}\}\subseteq M_{26}$. For any $\beta\in I_{26}$, define $\varphi_{26}\colon I_{26}\rightarrow N_{26}$ as follows:
\begin{equation}
\varphi_{26}(\beta)(v)=\begin{cases}
1,&\text{if }v=v_0;\\
\phi^{-1}(\beta\mid_{G_1})(v)& \text{if }v\in V(G_1);\\
\phi^{-1}(\beta\mid_{G_2})(v)& \text{if }v\in V(G_2);\\
\phi^{-1}(\beta\mid_{G_3})(v)& \text{if }v\in V(G_3).
\end{cases}
\end{equation}
Clearly $\varphi_{26}(\psi_{26}(\alpha))=\alpha$ for any $\alpha\in N_{26}$. This yields that $\psi_{26}$ is an injection.

Let $T$ denote the connected component in $T_{3,m,n}^\alpha$ contains $v_0^\alpha$.  Assume
\[\#\{(i,j)\colon \alpha(v_{ij})=\alpha(v_{ij}')=1\}=r.\]
We see that $T$ has bipartition $(r+k_1+k_2+k_3+1,r+3)$. From Proposition \ref{prop-2s-coe}, we have
\begin{equation}\label{equ-cal-scalpha-26'}
X_T=_{2s}s_{(r+k_1+k_2+k_3+1,r+3)}-s_{(r+k_1+k_2+k_3,r+4)}.
\end{equation}

Moreover, define $\hat{T}$ and $\overline{T}$ as in Definition \ref{defi-t-hat}. From the construction of $\psi_{26}$, we see that in each case $\alpha(v)=\beta(v)$ for any $v\in\overline{T}$. By Proposition \ref{pro-X-hat-overline-T}, we have
\begin{equation}\label{equ-cal-scalpha-26'-1}
    X_{T_{3,m,n}}^\alpha+X_{T_{3,m,n}}^\beta=(X_{\hat{T}}^\alpha+X_{\hat{T}}^\beta)X_{\overline{T}}^\alpha.
\end{equation}

From the construction of $\psi_{26}$,   in each case we see that $\hat{T}^\beta$ consists of $r+3$ copies of $K_2$ and $k_1+k_2+k_3-2$ isolated vertices. This yields that
\begin{equation}\label{equ-cal-scalpha-26'-2}
    X_{\hat{T}}^\beta=\frac{1}{2^3}(2s_{(1,1)})^{r+3} s_{1}^{k_1+k_2+k_3-2}\ge_{2s} s_{(r+k_1+k_2+k_3,r+4)}.
\end{equation}
Substituting \eqref{equ-cal-scalpha-26'} and \eqref{equ-cal-scalpha-26'-2} into \eqref{equ-cal-scalpha-26'-1}, we get $X_{T_{3,m,n}}^\alpha+X_{T_{3,m,n}}^\beta\ge_{2s} 0$.
\end{proof}

\begin{lem}\label{lem-psi27}
There is an injection $\psi_{27}$ from $N_{27}$ to $M_{27}$. Moreover, $X_{T_{3,m,n}}^\alpha+X_{T_{3,m,n}}^{\psi_{27}(\alpha)}$ is 2-$s$-positive.
\end{lem}

\begin{proof}
Given $\alpha\in N_{27}$, by definition we see that $\alpha(v_0)=\alpha(v_1)=\alpha(v_2)=\alpha(v_3)=1$, $k_1=2$ and $k_2=k_3=1$. Define
\begin{equation}\label{equ-def-psi26'}
\psi_{27}(\alpha)(v)=\begin{cases}
2,&\text{if }v=v_0;\\
\phi_{1}(\alpha\mid_{G_1})(v),&\text{if }v\in V(G_1);\\
\phi_\emptyset(\alpha\mid_{G_2})(v),&\text{if } v\in V(G_2);\\
\phi_1(\alpha\mid_{G_3})(v),&\text{if } v\in V(G_3).\\
\end{cases}
\end{equation}
Set $\beta=\psi_{27}(\alpha)$. It is clear that $G_1^\beta\in X_{1,3}^2$, $G_2^\beta\in X^1_{\emptyset, m}$ and $G_3^\beta\in X^1_{1,n}$. This implies $\beta\in M_{27}$.

Let $I_{27}=\{\psi_{27}(\alpha)\colon \alpha\in N_{27}\}\subseteq M_{27}$. For any $\beta\in I_{27}$, define $\varphi_{27}\colon I_{27}\rightarrow N_{27}$ as follows:
\begin{equation}
\varphi_{27}(\beta)(v)=\begin{cases}
1,&\text{if }v=v_0;\\
\phi^{-1}(\beta\mid_{G_1})(v),& \text{if }v\in V(G_1);\\
\phi^{-1}(\beta\mid_{G_2})(v),& \text{if }v\in V(G_2);\\
\phi^{-1}(\beta\mid_{G_3})(v),& \text{if }v\in V(G_3).
\end{cases}
\end{equation}
Clearly $\varphi_{27}(\psi_{27}(\alpha))=\alpha$ for any $\alpha\in N_{27}$. This yields that $\psi_{27}$ is an injection.

Let $T$ denote the connected component in $T_{3,m,n}^\alpha$ contains $v_0^\alpha$. Moreover, define $\hat{T}$ and $\overline{T}$ as in Definition \ref{defi-t-hat}. Assume
\[\#\{(i,j)\colon \alpha(v_{ij})=\alpha(v_{ij}')=1\}=r.\]
We see that $T$ has bipartition $(r+5,r+3)$. Thus
\begin{equation}\label{eq-x-t-27-1}
    X_T=X_{\hat{T}}^\alpha=_{2s} s_{(r+5,r+3)}-s_{(r+4,r+4)}.
\end{equation}
Moreover, from the construction of $\psi_{27}$, we have $\alpha(v)=\beta(v)$ for any $v\in\overline{T}$. By Proposition \ref{pro-X-hat-overline-T}, we deduce that
\begin{equation}\label{eq-x-t-27-2}
    X_{T_{3,m,n}}^\alpha+X_{T_{3,m,n}}^\beta=(X_{\hat{T}}^\alpha+X_{\hat{T}}^\beta)X_{\overline{T}}^\alpha.
\end{equation}
Furthermore, again from the construction of $\psi_{27}$, We see that $\hat{T}^\beta$ consists of $r+3$ copies of $K_2$ and two isolated vertices. Hence
\begin{equation}\label{eq-x-t-27-3}
    X_{\hat{T}}^\beta=\frac{1}{2^3}(2s_{(1,1)})^{r+3} s_1^2=_{2s} 2^r (s_{(r+5,r+3)}+s_{(r+4,r+4)}).
\end{equation}
Substituting \eqref{eq-x-t-27-1} and \eqref{eq-x-t-27-3} into \eqref{eq-x-t-27-2}, we obtain $X_{T_{3,m,n}}^\alpha+X_{T_{3,m,n}}^\beta\ge_{2s} 0$.
\end{proof}

In order to describe the structure in $T_{3,m,n}^\alpha$ when $\alpha\in N_{28}\cup N_{29}$, we need to introduce the following definition.

\begin{defi}
    For fixed $\alpha\in N_{28}\cup N_{29}\cup N_{30}$,   let $\alpha_{i,j}$ denote the integer pair $(\alpha(v_{ij}),\alpha(v'_{ij}))$, where $v_{ij}, v'_{ij}\in V(T_{3,m,n})$.
\end{defi}

\begin{lem}\label{lem-psi28}
There is an injection $\psi_{28}$ from $N_{28}$ to $M_{28}$. Moreover, $X_{T_{3,m,n}}^\alpha+X_{T_{3,m,n}}^{\psi_{28}(\alpha)}$ is 2-$s$-positive.
\end{lem}

\begin{proof}
Given $\alpha\in N_{28}$, by definition we see that $\alpha(v_0)=\alpha(v_1)=\alpha(v_2)=\alpha(v_3)=1$ and $k_1+k_2+k_3=0$. Moreover, there exists $v_{ij}$ such that $\alpha(v_{ij})=1$ and $\sum_{v\in V(T_{3,m,n})} \alpha(v)<2m+2n+10$. From $k_1+k_2+k_3=0$ we see that there is no $k,l$ such that $\alpha_{k,l}=(1,0)$. From Corollary \ref{cor-2s0}, we see that $\alpha(v_{kl})\le 1$ for any vertex $v_{kl}$. Thus $\alpha_{k,l}$ must be chosen from the following four cases: $\{(0,0), (0,1), (1,1), (0,2)\}$. Let $T$ denote the connected component in $T_{3,m,n}^\alpha$ which contains $v_0^\alpha$. Assume
\begin{equation}\label{eq-psi-28-1}
    \#\{(k,l)\colon \alpha_{k,l}=(1,1)\}=r\quad \text{and}\quad\#\{(k,l)\colon \alpha_{k,l}=(0,2)\}=h.
\end{equation}
Then clearly $T$ has bipartition $(r+3,r+1)$. From Proposition \ref{pro-claw-iso}, there are no isolated vertex in $T_{3,m,n}^\alpha$, which implies there does not exist $(k,l)$ such that $\alpha_{k,l}=(0,1)$. Furthermore, from $\sum_v \alpha(v)<2m+2n+10$ we see that there exists $\alpha_{p,q}=(0,0)$. We first choose such $p$ to be maximum, then choose the minimum $q$ among the same value $p$. On the other hand, by the definition of $N_{28}$, there exists $(i,j)$ such that $\alpha(v_{ij})=1$, this  implies  $\alpha_{i,j}=(1,1)$. We first choose such $i$ to be minimum, and then choose the minimum $j$ among the same value $i$. Define
\begin{equation}\label{equ-def-psi27}
\psi_{28}(\alpha)(v)=\begin{cases}
1,&\text{if }v=v_{pq};\\
0,&\text{if }v=v_{ij};\\
\alpha(v),&\text{otherwise.}
\end{cases}
\end{equation}
Set $\beta=\psi_{28}(\alpha)$.  From the construction of $\psi_{28}$, it is clear that $\beta_{p,q}=(1,0)$, $\beta_{i,j}=(0,1)$ and all the other $\beta_{k,l}=\alpha_{k,l}\in\{(0,2),(1,1),(0,0)\}$ for any $(k,l)\not\in \{(i,j),(p,q)\}$.  In other words, there is a unique $(p,q)$ such that $\beta_{p,q}=(1,0)$ and a unique $(i,j)$ satisfies $\beta_{i,j}=(0,1)$. This yields   $\beta\in M_{28}$.

We next show that $\psi_{28}$ is an injection. Let $I_{28}=\{\psi_{28}(\alpha)\colon \alpha\in N_{28}\}\subseteq M_{28}$. For any $\beta\in I_{28}$, let  $(i,j)$ be the unique pair such that $\beta_{i,j}=(0,1)$ and   let  $(p,q)$ be the unique pair such that $\beta_{p,q}=(1,0)$. We construct a map $\varphi_{28}\colon I_{28}\rightarrow N_{28}$ as follows:
\begin{equation}
\varphi_{28}(\beta)(v)=\begin{cases}
0,&\text{if }v=v_{pq};\\
1,&\text{if }v=v_{ij};\\
\beta(v),&\text{otherwise.}
\end{cases}
\end{equation}
It is easy to check that the pair $(p,q)$ in $\beta$ is coincides with the pair $(p,q)$ in the construction of $\psi_{28}$. Moreover, the pair $(i,j)$ in $\beta$ is coincides with the pair $(i,j)$ in the construction of $\psi_{28}$. Thus $\varphi_{28}(\psi_{28}(\alpha))=\alpha$ for any $\alpha\in N_{28}$. This yields $\psi_{28}$ is an injection.

We next calculate $X_{T_{3,m,n}}^\alpha$. Note that $r$ and $h$ are defined as in \eqref{eq-psi-28-1}. Then $T_{3,m,n}^\alpha$ consists of $T$ together with $h$ copies of $K_2$. Since $T$ has a bipartition $(r+3,r+1)$, using Proposition \ref{prop-2s-coe},
\begin{equation}\label{equ-xgalpha-27}
    X_{T_{3,m,n}}^\alpha=_{2s} \frac{1}{2^h}(s_{(r+3,r+1)}-s_{(r+2,r+2)})(2 s_{(1,1)})^h=_{2s}s_{(h+r+3,h+r+1)}-s_{(h+r+2,h+r+2)}.
\end{equation}

Now we calculate $X_{T_{3,m,n}}^\beta$. Let $T'$ denote the connect component in $T_{3,m,n}^\beta$ which contains $v_0^\beta$. Then we see that $T'$ has a bipartition $(r+2,r+1)$. From Proposition \ref{prop-2s-coe}, we have  $X_{T'}=_{2s}s_{(r+2,r+1)}$. Moreover, $T_{3,m,n}^\beta$ consists of $h$ copies of $K_2$ together with an isolated vertex $v'_{ij}$ (outside $T'$). This yields
\begin{equation}\label{equ-xgbeta-27}
X_{T_{3,m,n}}^\beta=\frac{1}{2^h}X_{T'} (2s_{(1,1)})^h s_1=_{2s}s_{(r+h+3,r+h+1)}+s_{(r+h+2,r+h+2)}.
\end{equation}

Combining \eqref{equ-xgalpha-27} and \eqref{equ-xgbeta-27}, we obtain $X_{T_{3,m,n}}^\alpha+X_{T_{3,m,n}}^\beta\ge_{2s} 0$.
\end{proof}

\begin{lem}\label{lem-psi29}
There is an injection $\psi_{29}$ from $N_{29}$ to $M_{29}$. Moreover, $X_{T_{3,m,n}}^\alpha+X_{T_{3,m,n}}^{\psi_{29}(\alpha)}$ is 2-$s$-positive.
\end{lem}

\begin{proof}
Given $\alpha\in N_{29}$, by definition we see that $\alpha(v_0)=\alpha(v_1)=\alpha(v_2)=\alpha(v_3)=1$ and $k_1+k_2+k_3=0$. Moreover, $\alpha(v_{ij})=0$ for any $i,j$. Using the same argument as in Lemma \ref{lem-psi28},  we see that $\alpha_{i,j}\in\{(0,2),(0,0)\}$ for any $i,j$.  Define
\begin{equation}\label{equ-def-psi28}
\psi_{29}(\alpha)(v)=\begin{cases}
2,&\text{if }v=v_1;\\
0,&\text{if }v=v_0;\\
\alpha(v),&\text{otherwise.}
\end{cases}
\end{equation}
It is evident that $\psi_{29}(\alpha)\in M_{29}$.

We next show that $\psi_{29}$ is an injection. Let $I_{29}=\{\psi_{29}(\alpha)\colon \alpha\in N_{29}\}\subseteq M_{29}$. For any $\beta\in I_{29}$, we construct a map $\varphi_{29}\colon I_{29}\rightarrow N_{29}$ as follows:
\begin{equation}
\varphi_{29}(\beta)(v)=\begin{cases}
1,&\text{if }v\in\{v_0,v_1\};\\
\beta(v),&\text{otherwise.}
\end{cases}
\end{equation}
Clearly $\varphi_{29}(\psi_{29}(\alpha))=\alpha$ for any $\alpha\in N_{29}$. This yields $\psi_{29}$ is an injection.

We next calculate $X_{T_{3,m,n}}^\alpha$. Assume
\[\#\{(i,j)\colon \alpha_{i,j}=(0,2)\}=h.\]
 Using the same argument as in \eqref{equ-xgalpha-27}, we have
\begin{equation}\label{equ-xgalpha-28}
X_{T_{3,m,n}}^\alpha=_{2s} s_{(h+3,h+1)}-s_{(h+2,h+2)}.
\end{equation}
Now we calculate $X_{T_{3,m,n}}^\beta$. Clearly $T_{3,m,n}^\beta$ consists of $(h+1)$ copies of $K_2$ and two isolated vertices, namely $v_2^\beta$ and $v_3^\beta$. Thus
\begin{equation}\label{equ-xgbeta-28}
    X_{T_{3,m,n}}^\beta=\frac{1}{2^{h+1}} (2s_{(1,1)})^{h+1} s_1^2=_{2s}s_{(h+3,h+1)}+s_{(h+2,h+2)}.
\end{equation}
Combining \eqref{equ-xgalpha-28} and \eqref{equ-xgbeta-28}, we obtain $X_{T_{3,m,n}}^\alpha+X_{T_{3,m,n}}^\beta\ge_{2s} 0$.
\end{proof}

\begin{lem}\label{lem-29}
For any $\alpha\in N_{30}$ and $k\ne m+n+5$, we have
\[[s_{(k,k)}] X_{T_{3,m,n}}^\alpha\ge 0.\]
\end{lem}

\begin{proof}
Given $\alpha\in N_{30}$, by definition we see that $\alpha(v_0)=\alpha(v_1)=\alpha(v_2)=\alpha(v_3)=1$ and $k_1+k_2+k_3=0$. Moreover $\sum_v \alpha(v)=2m+2n+10$. If $[s_{(k,k)}] X_{T_{3,m,n}}^\alpha\ne 0$, then we have $2k=\sum_{v\in V(T_{3,m,n})} \alpha(v)=2m+2n+10$. Thus $k=m+n+5$.

\end{proof}

{\noindent \it Proof of Theorem \ref{thm-uni-t3mn}.} It is obvious that the independence number of $T_{3,m,n}$ is $m+n+6$. Assume
\[I_{T_{3,m,n}}(t)=\sum_{j=0}^{m+n+6}i_j t^j.\]

Define $\psi\colon N\setminus N_{30}\rightarrow M$ as follows:
\begin{equation}\label{equ-psi-all}
\psi(\alpha)=\psi_i(\alpha),\quad\text{if }\alpha\in N_i.
\end{equation}
From Lemma \ref{lem-psi1} $\sim$ Lemma \ref{lem-psi29}, we see that $\psi$ is an injection and \begin{equation}\label{equ-thm-1-xt3mn-alpha-psi-alpha}
X_{T_{3,m,n}}^\alpha+X_{T_{3,m,n}}^{\psi(\alpha)}\ge_{2s} 0.
\end{equation}
Recall that $I_i=\{\psi_i(\alpha)\colon \alpha\in N_i\}\subseteq M_i$. We have
\begin{align*}
Y_{T_{3,m,n}}&=\sum_{\alpha\in N}X_{T_{3,m,n}}^\alpha+\sum_{\alpha\in M}X_{T_{3,m,n}}^\alpha\\
&=\sum_{i=1}^{29}\sum_{\alpha\in N_i}\left(X_{T_{3,m,n}}^\alpha+X_{T_{3,m,n}}^{\psi(\alpha)}\right)+\sum_{\alpha\in M\atop\alpha\not\in\bigcup_{i=1}^{29}I_i}X_{T_{3,m,n}}^\alpha+\sum_{\alpha\in N_{30}}X_{T_{3,m,n}}^\alpha.
\end{align*}
Using \eqref{equ-thm-1-xt3mn-alpha-psi-alpha}, we have $$\sum_{i=1}^{29}\sum_{\alpha\in N_i}\left(X_{T_{3,m,n}}^\alpha+X_{T_{3,m,n}}^{\psi(\alpha)}\right)\ge_{2s} 0.$$
From the definition of $M$ we have $X_{T_{3,m,n}}^\alpha\ge_{2s} 0$ for each $\alpha\in M\setminus\bigcup_{i=1}^{29}I_i$.
 Moreover, by Lemma \ref{lem-29}, for any $k\ne m+n+5$, and $\alpha\in N_{30}$, we have $[s_{(k,k)}]X_{T_{3,m,n}}^\alpha\ge 0$. From the above analysis, we get for $1\le k\le m+n+4$, we have
\[[s_{(k,k)}]Y_{T_{3,m,n}}\ge 0\]
 Hence by Corollary \ref{lem-skk}, we have $\{i_j\}_{j=0}^{m+n+5}$ is log-concave, which yields the sequence $\{i_j\}_{j=0}^{m+n+5}$ is unimodal. Thus there exists $k$ such that
\[i_0\le i_1\le \cdots\le i_k\ge i_{k+1}\ge\cdots\ge i_{m+n+5}.\]
Combining Theorem \ref{thmLM07-2} we get
\[i_0\le i_1\le \cdots\le i_k\ge i_{k+1}\ge\cdots\ge i_{m+n+5}\ge i_{m+n+6}.\]
This completes the proof.\qed

\section{Unimodality of the independence polynomial on $T^*_{3,m,n}$}\label{sec5}

This section is devoted to proving Theorem \ref{thm-uni-t3mn-star}. To this end, we first list two properties that are crucial to our proof.  As in Section \ref{sec4}, we partition the set $N'=\{\alpha\colon {X_{T_{3,m,n}}^*}^\alpha$ is non-\break2-$s$-positive$\}$ into four subsets $N'_1,N'_2,N'_3$ and $N'_4$. Moreover, we list three pairwise disjoint subsets of the set $M'=\{\alpha\colon {X_{T_{3,m,n}}^*}^\alpha\text{ is 2-$s$-positive}\}$, namely $M'_1$, $M'_2$ and $M'_3$. Furthermore, we show that for each $i\in\{1,2,3\}$, there exists an injection $\Psi_i$ from $N'_i$ into $M'_i$ such that $X_{T_{3,m,n}^*}^\alpha+X_{T_{3,m,n}^*}^{\Psi_i(\alpha)}\ge_{2s} 0$,   and that for any $\alpha\in N'_4$ and $k\le m+n+4$, we have $[s_{(k,k)}]X_{T_{3,m,n}^*}^\alpha=0$. In this way, analogously to the proof of Theorem \ref{thm-uni-t3mn}, we obtain a proof of Theorem \ref{thm-uni-t3mn-star}.

We first prove the following property concerning the map $\psi$ defined in \eqref{equ-psi-all}.
\begin{prop}\label{pro-psi-alpha-v13}
    For any $\alpha\in N\setminus N_{30}$, we have
    \begin{equation}\label{equ-alpha-psi-alpha-v-11'}
\alpha(v'_{13})=\psi(\alpha)(v'_{13}).
\end{equation}
\end{prop}

\begin{proof}
    This result can be verified by checking that \eqref{equ-alpha-psi-alpha-v-11'} holds for each $\alpha\in N_i$, where $1\le i\le 29$. In fact, since applying the map $\phi_{S}$ to $\alpha\mid_{G_1}$ does not change the value of $\alpha(v'_{13})$ for any $S$, we only need to check $\psi_{11}, \psi_{18}, \psi_{23}$, $\psi_{25}$ and $\psi_{28}$. In all  five cases, we have $\alpha(v'_{13})=\psi(\alpha)(v'_{13})=0$. This yields our desired result.
\end{proof}

We next record a result that will be essential in the proof.

\begin{prop}
Let $G$ be a graph and let $v \in V(G)$. Denote by  $G_v$ the graph obtained from
 $G$ by appending a path of length $2$ at $v$, this path consists of two new vertices $c$, $d$ and edges $vc$, $cd$. For any $\alpha\colon V(G_v)\rightarrow \mathbb{N}$, the following hold.
     \begin{itemize}
        \item[(1)] If $\alpha(c)=0$, then
        \begin{equation}\label{eq-xgv-alpha-xg-xd-alpha-0}
            X_{G_v}^\alpha=X_G^\alpha X_d^\alpha.
        \end{equation}
        \item[(2)] If $\alpha(c)=2$ and $X_{G_v}^\alpha\neq_{2s} 0$, then
        \begin{equation}\label{eq-xgv-alpha-xg-xd-alpha-1}
            X_{G_v}^\alpha=X_G^\alpha s_{(1,1)}.
        \end{equation}
        \item[(3)] If $\alpha(c)=\alpha(d)=\alpha(v)=1$, then
        \begin{equation}\label{eq-xgv-alpha-xg-xd-alpha-2}
            X_{G_v}^\alpha=_{2s}X_G^\alpha s_{(1,1)}.
        \end{equation}
        \item[(4)] If $\alpha(c)=\alpha(d)=1$ and $\alpha(v)=0$, then
        \begin{equation}\label{eq-xgv-alpha-xg-xd-alpha-3}
            X_{G_v}^\alpha=X_G^\alpha 2s_{(1,1)}.
        \end{equation}
    \end{itemize}
\end{prop}

\begin{proof}
We prove each part separately.

(1). If $\alpha(c)=0$, then there are no edges between $G^\alpha$ and $d^\alpha$  and $G_v^\alpha=G^\alpha+d^\alpha$. Proposition \ref{prop-42} immediately gives \eqref{eq-xgv-alpha-xg-xd-alpha-0}.

(2). If $\alpha(c)=2$, then from Corollary \ref{cor-2s0}, the condition  $X_{G_v}^\alpha\neq_{2s} 0$ forces  $\alpha(x)+\alpha(y)\le 2$ for any edge $xy\in E(G_v)$. Thus $\alpha(v)=\alpha(d)=0$. Consequently $G_v^\alpha=G^\alpha+c^\alpha$. By Proposition \ref{prop-42}, we get
\begin{equation}
    X_{G_v}^\alpha=X_G^\alpha X_c^\alpha=X_G^\alpha \frac12 (2s_{(1,1)})=X_G^\alpha s_{(1,1)}.
\end{equation}

(3). If $\alpha(c)=\alpha(d)=\alpha(v)=1$, then let $T$ be the  connected component of $G^\alpha$  containing $v^\alpha$, and let $T'$ be the   connected component of $G_v^\alpha$  containing $v^\alpha$. Define $\hat{T}, \hat{T'}$ and $\overline{T'}$  as in Definition \ref{defi-t-hat}. If $T$ is non-bipartite, then we have $G^\alpha$ and $G_v^\alpha$ are both non-bipartite. From Corollary \ref{cor-2s0}, we get
$$X_{G_v}^\alpha=0=0\cdot s_{(1,1)}=X_G^\alpha s_{(1,1)},$$
which is \eqref{eq-xgv-alpha-xg-xd-alpha-2}.

Now assume that $T$ has bipartition $(p,q)$. Since $T'$ is obtained from $T$ by attaching a path of length $2$ at $v^\alpha$, it has bipartition $(p+1,q+1)$. If $p\ge q+2$, then we have $X_T=_{2s} s_{(p,q)}-s_{(p-1,q+1)}$ and $X_{T'}=_{2s} s_{(p+1,q+1)}-s_{(p,q+2)}$. If $q\le p\le q+1$, then we have $X_T=_{2s} s_{(p,q)}$ and $X_{T'}=_{2s} s_{(p+1,q+1)}$. It is trivial to check that in either case,
\begin{equation}\label{eq-xt'-2s-xts-11-1}
    X_{T'}=_{2s}X_Ts_{(1,1)}
\end{equation}
holds.

On the other hand, observe that $G=\hat{T}+\overline{T'}$. Using Proposition \ref{prop-42}, we get
\begin{equation}\label{eq-xt'-2s-xts-11-2}
    X_{G}^\alpha=X_{\hat{T}}^\alpha X_{\overline{T'}}^\alpha.
\end{equation}
From Proposition \ref{pro-X-hat-overline-T}, we get
\begin{equation}\label{eq-xt'-2s-xts-11-3}
    X_{G_v}^\alpha=X_{\hat{T'}}^\alpha X_{\overline{T'}}^\alpha.
\end{equation}
It is clear that $X_T=X_{\hat{T}}^\alpha$ and $X_{T'}=X_{\hat{T'}}^\alpha$. From \eqref{eq-xt'-2s-xts-11-1}, \eqref{eq-xt'-2s-xts-11-2} and \eqref{eq-xt'-2s-xts-11-3}, we deduce \eqref{eq-xgv-alpha-xg-xd-alpha-2}.

(4). Since $\alpha(c)=\alpha(d)=1$ and $\alpha(v)=0$, Proposition \ref{prop-42} directly yields \eqref{eq-xgv-alpha-xg-xd-alpha-3}.
\end{proof}

Now we are in a position to give explicit definitions of $N'_1$, $N'_2$, $N'_3$ and $N'_4$.
Let $T$ be the subgraph of $T^*_{3,m,n}$ induced by $V(T^*_{3,m,n})\setminus \{x,y\}$. Clearly $T\cong T_{3,m,n}$.
 As stated in the beginning of this section, we partition the set $N'$ into the following four  subsets:
\begin{itemize}
\item[(1)] $N'_1=\{\alpha\in N'\colon \alpha(x)\ne 1,\alpha\mid_T\not\in N_{30}\}$;
\item[(2)] $N'_2=\{\alpha\in N'\colon \alpha(x)=\alpha(y)=1,\alpha\mid_T\not\in N_{30}\}$;
\item[(3)] $N'_3=\{\alpha\in N'\colon \alpha(x)=1, \alpha(y)=0, \sum_{v\in V(T)}\alpha(v)\le 2m+2n+7\}$.
\item[(4)] $N'_4=\{\alpha\in N'\colon \alpha\mid_T\in N_{30}\}\cup \{\alpha\in N'\colon \alpha(x)=1, \alpha(y)=0, \sum_{v\in V(T)}\alpha(v)\ge 2m+2n+8\}$.
\end{itemize}

Similarly,
we now list three pairwise disjoint subsets of $M'$ as follows:
\begin{itemize}
\item[(1)] $M'_1=\{\alpha\in M'\colon \alpha(x)\ne 1\}$;
\item[(2)] $M'_2=\{\alpha\in M'\colon \alpha(x)=\alpha(y)=1\}$;
\item[(3)] $M'_3=\{\alpha\in M'\colon \alpha(x)=1,\alpha(y)=0\}$.
\end{itemize}

We next build three injections $\Psi_i\colon N'_i\rightarrow M'_i$, where $1\le i\le 3$, and  show that for any
$\alpha \in N'_i$, we have $X_{T_{3,m,n}^*}^\alpha+X_{T_{3,m,n}^*}^{\Psi_i(\alpha)}\ge_{2s} 0$. Then using the same method as in the proof of Theorem \ref{thm-uni-t3mn}, we confirm Theorem \ref{thm-uni-t3mn-star}.

\begin{lem}\label{lem-t-star-psi-1}
There is an injection $\Psi_{1}$ from $N'_{1}$ to $M'_{1}$. Moreover, $X_{T_{3,m,n}^*}^\alpha+X_{T_{3,m,n}^*}^{\Psi_{1}(\alpha)}$ is 2-$s$-positive.
\end{lem}

\begin{proof}
Given $\alpha \in N'_1$, by definition, we see that $\alpha(x)\ne 1$ and $\alpha\mid_T\not\in N_{30}$.
We first show that $X_T^\alpha$ is non-2-$s$-positive. There are two cases.

Case 1. If $\alpha(x)=0$, then by \eqref{eq-xgv-alpha-xg-xd-alpha-0},
\begin{equation}\label{equ-alpha-N'1-0}
X_{T_{3,m,n}^*}^\alpha=X_{T}^\alpha X_y^\alpha.
\end{equation}
Since ${T_{3,m,n}^*}^\alpha$ is non-2-$s$-positive and  $y^\alpha$ is always $s$-positive, it follows that $X_T^\alpha$ is non-2-$s$-positive.

Case 2. If $\alpha(x)=2$, then by \eqref{eq-xgv-alpha-xg-xd-alpha-1},
\begin{equation}\label{equ-alpha-N'1-1}
X_{T_{3,m,n}^*}^\alpha=s_{(1,1)}X_T^\alpha.
\end{equation}
Again, since ${T_{3,m,n}^*}^\alpha$ is non-2-$s$-positive, we conclude that $X_T^\alpha$ is non-2-$s$-positive.

Note that Corollary \ref{cor-2s0} implies $\alpha(x)\le 2$, and by the definition of $N_1'$ we have $\alpha(x)\ne 1$. Hence in either case we have $X_T^\alpha$ is non-2-$s$-positive. In other words, $\alpha\mid_T\in N\setminus N_{30}$.

Define $\Psi_1$ as follows:
\begin{equation}\label{equ-def-Psi-1}
\Psi_1(\alpha)(v)=\begin{cases}
\psi(\alpha\mid_T)(v),&\text{if }v\not\in \{x,y\};\\
\alpha(v),&\text{if }v\in\{x,y\}.
\end{cases}
\end{equation}
Let $\beta=\Psi_1(\alpha)$, we proceed to show that $\beta\in M'_1$. Since $\beta(x)=\alpha(x)\ne 1$, we have either $\beta(x)=0$ or $2$. If $\beta(x)=0$, then by \eqref{eq-xgv-alpha-xg-xd-alpha-0}, we have
\begin{equation}\label{eq-star-Psi-1-1}
    X_{T_{3,m,n}^*}^\beta=X_T^\beta X_y^\beta.
\end{equation}
Since $\beta\mid_T\in M$ we have that $T^\beta$ is 2-$s$-positive. Thus $\beta\in M'$. If $\beta(x)=2$, then by the construction of $\Psi_1$, we see that $\alpha(x)=2$. From Corollary \ref{cor-2s0} and the fact that ${T_{3,m,n}^*}^\alpha$ is non-2-$s$-positive, we have $\alpha(y)=\alpha(v'_{13})=0$. Again by the construction of $\Psi_1$ we have $\beta(y)=\alpha(y)=0$. Moreover, using Proposition \ref{pro-psi-alpha-v13}, we get $\alpha(v'_{13})=\beta(v'_{13})=0.$ This yields
\begin{equation}\label{eq-star-Psi-1-2}
X_{T_{3,m,n}^*}^\beta=X_T^\beta \cdot X_x^\beta=X_T^\beta \cdot \frac{1}{2}\left(2s_{(1,1)}\right)\ge_{2s} 0.
\end{equation}
So in either case $\beta\in M'$. Moreover, $\beta(x)=\alpha(x)\ne 1$. This yields $\beta\in M'_1$.

Since $\psi$ is an injection, it is easy to see that $\Psi_1$ is an injection.

We next show that $X_{T_{3,m,n}^*}^\alpha+X_{T_{3,m,n}^*}^{\Psi_{1}(\alpha)}$ is 2-$s$-positive. By the construction of $\Psi_1$, we have $\beta\mid_T=\psi(\alpha\mid_T)$. Thus
\begin{equation}\label{equ-alpha-N'1-}
    X_T^\beta+X_T^\alpha=X_T^{\beta\mid_T}+X_T^{\alpha\mid_T}\ge_{2s} 0.
\end{equation}

If $\alpha(x)=\beta(x)=0$, then combining \eqref{equ-alpha-N'1-0},  \eqref{eq-star-Psi-1-1}, \eqref{equ-alpha-N'1-} and the fact $\alpha(y)=\beta(y)$, we have
\begin{equation}
    X_{T_{3,m,n}^*}^\alpha+X_{T_{3,m,n}^*}^\beta=\left(X_T^\beta+X_T^\alpha\right) X_y^\alpha\ge_{2s} 0.
\end{equation}

Similarly, if $\alpha(x)=\beta(x)=2$, then
combining \eqref{equ-alpha-N'1-1}, \eqref{eq-star-Psi-1-2},  and \eqref{equ-alpha-N'1-} we have
\begin{equation}\label{equ-alpha-M'1}
X_{T_{3,m,n}^*}^\alpha+X_{T_{3,m,n}^*}^\beta=(X_T^\alpha+X_T^\beta)s_{(1,1)}\ge_{2s} 0.
\end{equation}
Thus in either case  $X_{T_{3,m,n}^*}^\alpha+X_{T_{3,m,n}^*}^\beta\ge_{2s} 0$.
\end{proof}

\begin{lem}\label{lem-t-star-psi-2}
There is an injection $\Psi_{2}$ from $N'_{2}$ to $M'_{2}$. Moreover, $X_{T_{3,m,n}^*}^\alpha+X_{T_{3,m,n}^*}^{\Psi_{2}(\alpha)}$ is 2-$s$-positive.
\end{lem}

\begin{proof}
Given $\alpha \in N'_2$. By definition, we see that $\alpha(x)=\alpha(y)=1$ and $\alpha\mid_T\not\in N_{30}$. We next show that $T^\alpha$ is non-2-$s$-positive, which implies $\alpha\mid_T\in N\setminus N_{30}$. There are two cases.

Case 1. If $\alpha(v'_{13})=0$, then by \eqref{eq-xgv-alpha-xg-xd-alpha-3},
\begin{equation}\label{eq-t-3mn-*-alpha-v13-0}
    X_{T_{3,m,n}^*}^\alpha=X_T^\alpha 2s_{(1,1)}.
\end{equation}
From ${T_{3,m,n}^*}^\alpha$ is non-2-$s$-positive, we have $X_T^\alpha$ is non-2-$s$-positive, which means $\alpha\mid_T\in N\setminus N_{30}$.

Case 2. If $\alpha(v'_{13})\ne 0$, then by Proposition \ref{prop-42} and $\alpha(x)=1$, we deduce that $\alpha(v'_{13})=1$. From \eqref{eq-xgv-alpha-xg-xd-alpha-2},
\begin{equation}\label{eq-xt3mn*-n2'-1}
    X_{T_{3,m,n}^*}^\alpha=s_{(1,1)} X_T^\alpha.
\end{equation}
Thus $X_T^\alpha<_{2s} 0$. This yields $\alpha\mid_T\in N\setminus N_{30}$.

Define $\Psi_2$ as follows:
\begin{equation}\label{equ-def-Psi-2}
\Psi_2(\alpha)(v)=\begin{cases}
\psi(\alpha\mid_T)(v),&\text{if }v\in V(T);\\
\alpha(v),&\text{if }v\in\{x,y\}.
\end{cases}
\end{equation}
Let $\beta=\Psi_2(\alpha)$, we proceed to show that $\beta\in M'_2$. Since $\beta(x)=\alpha(x)=1$ and $\beta(y)=\alpha(y)=1$,   it suffices to show that $\beta\in M'$. To this end, we calculate $X_{T_{3,m,n}^*}^\beta$ as follows. When $\alpha(v'_{13})=0$, by Proposition \ref{pro-psi-alpha-v13}, we have $\beta(v'_{13})=0$. Using \eqref{eq-xgv-alpha-xg-xd-alpha-3}, we have
\begin{equation}\label{eq-t-3mn-*-beta-v13-0}
    X_{T_{3,m,n}^*}^\beta=X_T^\beta 2s_{(1,1)}.
\end{equation}
Since $\beta\mid_T=\psi(\alpha\mid_T)\in M$, we have $\beta\in M'$.
Similarly, when $\alpha(v'_{13})=1$, by Proposition \ref{pro-psi-alpha-v13}, we have $\beta(v'_{13})=1$. Using \eqref{eq-xgv-alpha-xg-xd-alpha-2}, we deduce that
\begin{equation}\label{eq-t-3mn-*-beta-v13-1}
    X_{T_{3,m,n}^*}^\beta=_{2s}X_T^\beta s_{(1,1)}.
\end{equation}
Again by $\beta\mid_T=\psi(\alpha\mid_T)\in M$, we have $\beta\in M'$. Hence in either case, $\beta\in M'$. Moreover, since $\psi$ is an injection, we see that $\Psi_2$ is an injection.


We next show that $X_{T_{3,m,n}^*}^\alpha+X_{T_{3,m,n}^*}^{\beta}$ is 2-$s$-positive. Since $\beta\mid_T=\psi(\alpha\mid_T)$ we get
\begin{equation}\label{eq-xt-alpha-beta-n'2}
    X_T^\alpha+X_T^\beta\ge_{2s} 0.
\end{equation}
If $\alpha(v'_{13})=0$, then by \eqref{eq-t-3mn-*-alpha-v13-0}, \eqref{eq-t-3mn-*-beta-v13-0} and \eqref{eq-xt-alpha-beta-n'2}, we deduce that
\begin{equation}
    X_{T_{3,m,n}^*}^\alpha+X_{T_{3,m,n}^*}^\beta=2s_{(1,1)}(X_T^\alpha+X_T^\beta)\ge_{2s} 0.
\end{equation}
If $\alpha(v'_{13})=1$, then by \eqref{eq-xt3mn*-n2'-1}, \eqref{eq-t-3mn-*-beta-v13-1} and \eqref{eq-xt-alpha-beta-n'2}, we deduce that
\begin{equation}
    X_{T_{3,m,n}^*}^\alpha+X_{T_{3,m,n}^*}^\beta=s_{(1,1)}(X_T^\alpha+X_T^\beta)\ge_{2s} 0.
\end{equation}
So in either case, we get $X_{T_{3,m,n}^*}^\alpha+X_{T_{3,m,n}^*}^\beta\ge_{2s} 0$.
\end{proof}

In order to define the injection $\Psi_3$, we need a proposition on $\psi$. To this end, we introduce  a specific subset $\overline{N}_{28}$ of $N_{28}$.
\begin{defi}
Let $\overline{N}_{28}$ be the subset of $N_{28}$ consisting of those $\alpha$ such that
$\{(i,j)\colon \alpha_{i,j}=(0,0)\}=\{(1,3)\}.$
\end{defi}

We now state the property of $\psi$.

\begin{prop}\label{prop-ine-alpha-v-13-psi}
    For any $\alpha\in N\setminus (N_{30}\cup \overline{N}_{28})$, we have
    \begin{equation}\label{ine-alpha-v-13-psi}
        \alpha(v_{13})\ge \psi(\alpha)(v_{13}).
    \end{equation}
\end{prop}

\begin{proof}
    We first consider the case $\alpha\in N_{28}$. Given $\alpha\in N_{28}$, by the construction of $\psi_{28}$ \eqref{equ-def-psi27}, we find that $\alpha(v_{13})\ge \psi(\alpha)(v_{13})$ unless $(p,q)=(1,3)$. But from the choice of $(p,q)$, where we first choose the maximum $p$, and then choose the minimum $q$ among the same value $p$, we find that $(p,q)=(1,3)$ implies $\{(i,j)\colon \alpha_{i,j}=(0,0)\}=\{(1,3)\}$. Hence $\alpha(v_{13})< \psi(\alpha)(v_{13})$ means $\alpha\in \overline{N}_{28}$.

    It is routine to check that for $\alpha\in N\setminus (N_{30}\cup N_{28})$, \eqref{ine-alpha-v-13-psi} holds. This completes the proof.
\end{proof}

\begin{lem}\label{lem-t-star-psi-3}
There is an injection $\Psi_{3}$ from $N'_{3}$ to $M'_{3}$. Moreover, $X_{T_{3,m,n}^*}^\alpha+X_{T_{3,m,n}^*}^{\Psi_{3}(\alpha)}$ is 2-$s$-positive.
\end{lem}

\begin{proof}
Given $\alpha \in N'_3$. By definition, we see that $\alpha(x)=1,$ $\alpha(y)=0$ and $\sum_{v\in V(T)}\alpha(v)\le 2m+2n+7$. We use $T'$ to denote the subgraph of $T$ induced by $V(T)\setminus \{v'_{13}\}$, and let $\gamma$ denote $(\alpha\mid^{T'})\mid_T$ for simplicity. In other words, $\gamma\colon V(T)\rightarrow \mathbb{N}$ as follows:
\begin{equation}
    \gamma(v)=\begin{cases}
        \alpha(v),&\text{if }v\ne v'_{13};\\
        0,&\text{if }v=v'_{13}.
    \end{cases}
\end{equation}
We proceed to show that $\gamma\in N\setminus (N_{30}\cup \overline{N}_{28})$. Clearly $\alpha\mid_{T'}=\gamma\mid_{T'}$. There are three cases.

Case 1. $\alpha(v'_{13})=0=\gamma(v'_{13})$. In this case, $\gamma=\alpha\mid_T$, so $X_T^\alpha=X_T^\gamma$ and $T'^\alpha=T^\alpha$. By \eqref{eq-xgv-alpha-xg-xd-alpha-0}, we have
\begin{equation}\label{eq-x-t3mn*-alpha-1-n'-3-0}
    X_{T_{3,m,n}^*}^\alpha=X_{T'}^\alpha s_1=X_T^\gamma s_1.
\end{equation}
Since $X_{T_{3,m,n}^*}^\alpha$ is non-2-$s$-positive, it follows that $X_T^\gamma$ is non-2-$s$-positive. This yields $\gamma=\alpha\mid_T\in N$.

Case 2. $\alpha(v'_{13})=1$ and $\alpha(v_{13})=0$. In this case, we see that $T^\gamma=T'^\alpha$. Moreover, from \eqref{eq-xgv-alpha-xg-xd-alpha-3},
\begin{equation}\label{eq-x-t3mn*-alpha-1-n'-3-1}
    X_{T_{3,m,n}^*}^\alpha=X_{T'}^\alpha 2s_{(1,1)}=X_T^\gamma 2s_{(1,1)}.
\end{equation}
Thus since $X_{T_{3,m,n}^*}^\alpha$ is non-2-$s$-positive, we obtain that $X_T^\gamma$ is non-2-$s$-positive. This yields $\gamma=\alpha\mid_T\in N$.

Case 3. $\alpha(v'_{13})=1$ and $\alpha(v_{13})\ne 0$. In this case, from Corollary \ref{cor-2s0}, we see that $\alpha(v_{13})=1$. By \eqref{eq-xgv-alpha-xg-xd-alpha-2},
\begin{equation}\label{eq-x-t3mn*-alpha-1-n'-3-2}
    X_{T_{3,m,n}^*}^\alpha=X_{T'}^\alpha s_{(1,1)}=X_T^\gamma s_{(1,1)}.
\end{equation}
Thus from $X_{T_{3,m,n}^*}^\alpha$ is non-2-$s$-positive we get $X_T^\gamma$ is non-2-$s$-positive. This yields $\gamma=\alpha\mid_T\in N$.

In each case,  we get $\gamma\in N$. We next show that $\gamma\not\in \overline{N}_{28}\cup N_{30}$. From the definition of $N'_3$, we see that $\sum_{v\in V(T)}\alpha(v)\le 2m+2n+7$. From $\gamma(v'_{13})=0$, $\alpha(v'_{13})\le 1$, and $\alpha(v)=\gamma(v)$ for any $v\in V(T')$, we deduce that
\begin{equation}\label{ine-sum-gammavt}
    \sum_{v\in V(T)}\gamma(v)\le \sum_{v\in V(T)}\alpha(v)\le 2m+2n+7,
\end{equation}
Thus  $\gamma\not\in N_{30}$. If $\gamma\in \overline{N}_{28}$, from the analysis in Lemma \ref{lem-psi28}, we find that $\gamma_{i,j}\in \{(0,2),(1,1)\}$ unless $(i,j)=(1,3)$. Thus $\gamma(v_{ij})+\gamma(v'_{ij})=2$ for any $(i,j)\ne (1,3)$. This yields that
\begin{equation}
    \sum_{v\in V(T)}\gamma(v)=\sum_{i=0}^3\gamma(v_i)+\sum_{(i,j)\ne (1,3)}\gamma(v_{ij})+\gamma(v'_{ij})=4+2(m+n+2)=2m+2n+8,
\end{equation}
which contradicts \eqref{ine-sum-gammavt}. Hence $\gamma\in N\setminus (N_{30}\cup \overline{N}_{28})$.

Now we may define $\Psi_3$ as follows:
\begin{equation}\label{equ-def-Psi-3-1}
\Psi_3(\alpha)(v)=\begin{cases}
\psi(\gamma)(v),&\text{if }v\not\in \{x,y,v'_{13}\};\\
\alpha(v),&\text{if }v\in\{x,y,v'_{13}\}.
\end{cases}
\end{equation}
Set $\beta=\Psi_3(\alpha)$. We proceed to show that $\beta\in M'_3$. From $\beta(x)=\alpha(x)=1$ and $\beta(y)=\alpha(y)=0$, it suffices to prove $\beta\in M'$. Moreover, from Proposition \ref{pro-psi-alpha-v13}, we have $\psi(\gamma)(v'_{13})=\gamma(v'_{13})=0$. Furthermore, from the construction of $\Psi_3$, we get $\beta(v)=\psi(\gamma)(v)$ for any $v\in V(T')$. From the above analysis, we deduce that
$$X_{T'}^\beta=X_T^{\psi(\gamma)}.$$
Since $\psi(\gamma)\in M$, we have $X_{T'}^\beta=X_T^{\psi(\gamma)}\ge_{2s} 0$. There are three cases.

Case (i) $\alpha(v_{13}')=\beta(v'_{13})=0$. In this case, by \eqref{eq-xgv-alpha-xg-xd-alpha-0},
\begin{equation}\label{eq-x-t3mn*-beta-1-n'-3-0}
    X_{T_{3,m,n}^*}^\beta=s_1 X_{T'}^\beta\ge_{2s}0.
\end{equation}

Case (ii) $\alpha(v_{13}')=\beta(v'_{13})=1=\beta(v_{13})$. In this case, using \eqref{eq-xgv-alpha-xg-xd-alpha-2}, we have
\begin{equation}\label{eq-x-t3mn*-beta-1-n'-3-1}
    X_{T_{3,m,n}^*}^\beta=_{2s}X_{T'}^\beta s_{(1,1)}\ge_{2s} 0.
\end{equation}

Case (iii) $\alpha(v_{13}')=\beta(v'_{13})=1$ and $\beta(v_{13})=0$, then by \eqref{eq-xgv-alpha-xg-xd-alpha-3}, we get
\begin{equation}\label{eq-x-t3mn*-beta-1-n'-3-2}
    X_{T_{3,m,n}^*}^\beta=X_{T'}^\beta 2s_{(1,1)}\ge_{2s} 0.
\end{equation}

Clearly, in each case, we have $\beta\in M'$.

We next show that $\Psi_3$ is an injection. Let $I'_3=\{\Psi_3(\alpha)\colon \alpha\in N'_3\}.$  For any $\beta\in I'_3$, let $\mu$ denote $(\beta\mid^{T'})\mid_T$ for simplicity. To be specific, define $\mu\colon V(T)\rightarrow \mathbb{N}$ as follows:
\begin{equation}\label{defi-mu}
    \mu(v)=\begin{cases}
        \beta(v),&\text{if }v\ne v'_{13};\\
        0,& \text{if }v=v'_{13}.
    \end{cases}
\end{equation}
From the construction of $\Psi_3$ and Proposition \ref{pro-psi-alpha-v13}, we see that $\psi(\gamma)=\mu$, which implies $\mu\in \bigcup_{i=1}^{29}M_i$. Assume $\mu\in M_i$ for some $1\le i\le 29$, define $\pi\colon V(T_{3,m,n}^*)\rightarrow \mathbb{N}$ as follows:
\begin{equation}
    \pi(v)=\begin{cases}
        \varphi_i(\mu)(v),& \text{if }v\not\in\{v'_{13},x,y\};\\
        \beta(v),& \text{if }v\in\{v'_{13},x,y\}.
    \end{cases}
\end{equation}
From $\varphi_i(\psi_i(\gamma))=\gamma$ for any $1\le i\le 29$, we see that $\pi(\Psi_3(\alpha))=\alpha$. This yields $\Psi_3$ is an injection.

We proceed to show that $X_{T_{3,m,n}^*}^\alpha+X_{T_{3,m,n}^*}^{\beta}$ is 2-$s$-positive. From $\mu=\psi(\gamma)$, we see that $X_T^\mu+X_T^\gamma\ge_{2s} 0$. Since $\mu(v'_{13})=0$ and $\beta(v)=\mu(v)$ for any $v\in V(T')$, we deduce that $X_{T}^\mu=X_{T'}^\mu=X_{T'}^\beta$. Therefore $$X_{T'}^\beta+X_{T}^\gamma\ge_{2s} 0.$$
There are four cases.

Case A. $\alpha(v'_{13})=0$. From \eqref{eq-x-t3mn*-alpha-1-n'-3-0} and \eqref{eq-x-t3mn*-beta-1-n'-3-0}, we have
\begin{equation}
X_{T_{3,m,n}^*}^\alpha+X_{T_{3,m,n}^*}^\beta=_{2s}(X_{T}^{\gamma}+X_{T'}^\beta) s_1 \ge_{2s} 0.
\end{equation}

Case B. $\alpha(v'_{13})=1$ and $\alpha(v_{13})=0$. From Proposition \ref{prop-ine-alpha-v-13-psi} we have $$\beta(v_{13})=\psi(\gamma)(v_{13})\le \gamma(v_{13})=\alpha(v_{13})=0.$$
Thus $\beta(v_{13})=0$. Using \eqref{eq-x-t3mn*-alpha-1-n'-3-1} and \eqref{eq-x-t3mn*-beta-1-n'-3-2}, we get
\begin{equation}
X_{T_{3,m,n}^*}^\alpha+X_{T_{3,m,n}^*}^\beta=_{2s}2s_{(1,1)}(X_{T}^{\gamma}+X_{T'}^\beta)\ge_{2s} 0.
\end{equation}

Case C. $\alpha(v'_{13})=\alpha(v_{13})=\beta(v_{13})=1$.  Using \eqref{eq-x-t3mn*-alpha-1-n'-3-2} and \eqref{eq-x-t3mn*-beta-1-n'-3-1}, we get
\begin{equation}
X_{T_{3,m,n}^*}^\alpha+X_{T_{3,m,n}^*}^\beta=_{2s}s_{(1,1)}(X_{T}^{\gamma}+X_{T'}^\beta)\ge_{2s} 0.
\end{equation}

Case D. $\alpha(v'_{13})=1=\alpha(v_{13})$ and $\beta(v_{13})\ne 1$. Again from Proposition \ref{prop-ine-alpha-v-13-psi} we have
$$\beta(v_{13})=\psi(\gamma)(v_{13})\le \gamma(v_{13})=\alpha(v_{13})=1.$$
Hence $\beta(v_{13})=0$. Moreover, from the construction of $\Psi_3$, we have $\mu=\psi(\gamma)$, which yields $X_T^\mu\ge_{2s} 0$. Thus $X_{T'}^\beta=X_T^\mu\ge_{2s} 0$. From \eqref{eq-x-t3mn*-alpha-1-n'-3-2} and \eqref{eq-x-t3mn*-beta-1-n'-3-2},
\begin{equation}
X_{T_{3,m,n}^*}^\alpha+X_{T_{3,m,n}^*}^\beta=_{2s}s_{(1,1)}(X_{T}^{\gamma}+X_{T'}^\beta)+ s_{(1,1)}X_{T'}^\beta\ge_{2s} 0.
\end{equation}
Hence in each case we have $X_{T_{3,m,n}^*}^\alpha+X_{T_{3,m,n}^*}^\beta\ge_{2s} 0$.
\end{proof}

\begin{lem}\label{lem-n'4}
Given $\alpha\in N'_4$ and $k\le m+n+4$, we have
\[[s_{(k,k)}]X_{T_{3,m,n}^*}^\alpha=0.\]
\end{lem}

\begin{proof}
Given $\alpha\in N'_4$, by definition,  either $\alpha\mid_T\in N_{30}$ or $\alpha(x)=1$, $\alpha(y)=0$ and $\sum_{v\in V(T)}\alpha(v) \ge 2m+2n+8$.

If $\alpha\mid_T\not\in N_{30}$, we see that
$$\sum_{v\in V(T_{3,m,n}^*)}\alpha(v)=\alpha(x)+\alpha(y)+\sum_{v\in V(T)}\alpha(v)\ge 2m+2n+9.$$
If $\alpha\mid_T\in N_{30}$, by the definition of $N_{30}$, we have $\sum_{v\in V(T)}\alpha(v)= 2m+2n+10$. Thus
\[\sum_{v\in V(T_{3,m,n}^*)}\alpha(v)\ge \sum_{v\in V(T)}\alpha(v)= 2m+2n+10.\]
So in either case,  $$\sum_{v\in V(T_{3,m,n}^*)}\alpha(v)\ge 2m+2n+9.$$
This implies if $[s_{(k,k)}]X_{T_{3,m,n}^*}^\alpha\ne 0$, then $k\ge m+n+5$.
\end{proof}

{\noindent \it Proof of Theorem \ref{thm-uni-t3mn-star}.} It is obvious that the independence number of $T^*_{3,m,n}$ is $m+n+7$. Assume
\[I_{T^*_{3,m,n}}(t)=\sum_{j=0}^{m+n+7}i^*_j t^j.\]

Define $\Psi\colon N'\setminus N'_{4}\rightarrow M'$ as follows:
\begin{equation}\label{equ-psi-all-star}
\Psi(\alpha)=\begin{cases}
\Psi_1(\alpha),&\text{if }\alpha\in N'_1;\\
\Psi_2(\alpha),&\text{if }\alpha\in N'_2;\\
\Psi_3(\alpha),&\text{if }\alpha\in N'_3.
\end{cases}\end{equation}
From Lemma  \ref{lem-t-star-psi-1}, Lemma \ref{lem-t-star-psi-2} and Lemma \ref{lem-t-star-psi-3} we see that $\Psi$ is an injection and
\begin{equation}\label{equ-thm-1-xt3mn-alpha-psi-alpha-star}
X_{T^*_{3,m,n}}^\alpha+X_{T^*_{3,m,n}}^{\Psi(\alpha)}\ge_{2s} 0.
\end{equation}
Let $I^*=\{\Psi(\alpha)\colon \alpha\in N'\setminus N'_{4}\}\subseteq M'$.  We have
\begin{align*}
Y_{T^*_{3,m,n}}&=\sum_{\alpha\in N'}X_{T^*_{3,m,n}}^\alpha+\sum_{\alpha\in M'}X_{T^*_{3,m,n}}^\alpha\\
&=\sum_{i=1}^{3}\sum_{\alpha\in N'_i}\left(X_{T^*_{3,m,n}}^\alpha+X_{T^*_{3,m,n}}^{\Psi(\alpha)}\right)+\sum_{\alpha\in M'\setminus I^*}X_{T^*_{3,m,n}}^\alpha+\sum_{\alpha\in N'_{4}}X_{T^*_{3,m,n}}^\alpha.
\end{align*}
Using \eqref{equ-thm-1-xt3mn-alpha-psi-alpha-star}, we have $$\sum_{i=1}^{3}\sum_{\alpha\in N'_i}\left(X_{T^*_{3,m,n}}^\alpha+X_{T^*_{3,m,n}}^{\Psi(\alpha)}\right)\ge_{2s} 0.$$
From the definition of $M'$ we have $X_{T^*_{3,m,n}}^\alpha\ge_{2s} 0$ for each $\alpha\in M'\setminus I^*$.
 Moreover, by Lemma \ref{lem-n'4}, for any $k\le m+n+4$ and $\alpha\in N'_{4}$, we have $[s_{(k,k)}]X_{T^*_{3,m,n}}^\alpha\ge 0$. From the above analysis, we get for $1\le k\le m+n+4$,
\[[s_{(k,k)}]Y_{T^*_{3,m,n}}\ge 0\]
 Hence by Corollary \ref{lem-skk}, we have $\{i^*_j\}_{j=0}^{m+n+5}$ is log-concave, which yields the sequence $\{i^*_j\}_{j=0}^{m+n+5}$ is unimodal. Thus there exists $k$ such that
\[i^*_0\le i^*_1\le \cdots\le i^*_k\ge i^*_{k+1}\ge\cdots\ge i^*_{m+n+5}.\]
Combining Theorem \ref{thmLM07-2} we get
\[i^*_0\le i^*_1\le \cdots\le i^*_k\ge i^*_{k+1}\ge\cdots\ge i^*_{m+n+5}\ge i^*_{m+n+6}\ge i^*_{m+n+7}.\]
This completes the proof.\qed

\end{document}